\documentclass[11pt]{article}
\usepackage{amsmath,amsfonts,amssymb,amsthm}
\usepackage{graphics}
\usepackage{epsfig}
\usepackage{fullpage}
\usepackage{color}
\usepackage{epsfig,psfrag,expdlist}

\title{Relative entropy for hyperbolic-parabolic systems and application to the constitutive theory of thermoviscoelasticity }

\author{
Cleopatra Christoforou
\thanks{Department of Mathematics and Statistics,
 University of Cyprus, Nicosia 1678, Cyprus.  Email:  {\tt christoforou.cleopatra@ucy.ac.cy}},
\and 
Athanasios E. Tzavaras
\thanks{Computer, Electrical, Mathematical Sciences \& Engineering Division, King Abdullah University of Science and Technology (KAUST), Thuwal, Saudi Arabia.
Email: {\tt athanasios.tzavaras@kaust.edu.sa} }
\thanks{Institute of Applied and Computational Mathematics,
FORTH, Heraklion, Greece.}
\date{}
}

\renewcommand{\footnote}{\endnote}

\def\del{\partial}
\font\msym=msbm10

\def\Real{{\mathop{\hbox{\msym \char '122}}}}
\def\R{\Real}

\def\charf {\mbox{{\text 1}\kern-.24em {\text l}}}

\def\Z{\mathbb Z}

\def\cF{\mathcal F}
\def\cM{\mathcal M}
\def\cP{\mathcal P}
\def\cS{\mathcal S}

\def\Prob{{\rm Prob}}

\def\Fdot{\dot F}

\def\bU{\bar U}
\def\bF{\bar F}
\def\bFdot{ \dot {\bar F}}
\def\bQ{\bar Q}
\def\bZ{\bar Z}
\def\bSigma{\bar \Sigma}

\def\bu{\bar u}
\def\bv{\bar v}
\def\btheta{\bar \theta}
\def\bsigma{\bar \sigma}
\def\be{\bar e}
\def\bareta{\bar \eta}
\def\bkappa{\bar \kappa}
\def\bmu{\bar \mu}
\def\br{\bar r}
\def\barf{\bar f}

\def\heta{\hat \eta}
\def\hq{\hat q}
\def\Heta{\hat H}

\def\bonu{\boldsymbol{\nu}}
\def\bomu{\boldsymbol{\mu}}

\newcommand{\ignore}[1]{}

\newtheorem{lemma}{Lemma}[section]
\newtheorem{theorem}[lemma]{Theorem}

\newtheorem{corollary}[lemma]{Corollary}
\newtheorem{definition}[lemma]{Definition}
\newtheorem{remark}[lemma]{Remark}

\renewcommand{\del}{\partial}

\newcommand{\eps}{\varepsilon}
\newcommand{\To}{{\mathbb{T}}}
\newcommand{\barQT}{{\overline Q}_T}

\font\msym=msbm10
\def\Real{{\mathop{\hbox{\msym \char '122}}}}

\def\torus{{{\text{\rm T}} \kern-.42em {\text{\rm T}}}}
\def\T3{\torus^3}
\def\div{\hbox{div}\,}
\begin{document}
\maketitle
\baselineskip=18pt

\numberwithin{equation}{section}

\begin{abstract}
\noindent
We extend the relative entropy identity  to the class of hyperbolic-parabolic systems  whose hyperbolic part is symmetrizable.
 The resulting identity, in the general theory, is useful to provide stability of viscous solutions and yields a convergence result in the zero-viscosity limit to smooth solutions in an $L^p$
framework. Also it provides measure valued weak versus strong uniqueness
theorems for the hyperbolic problem. The relative entropy identity is also developed for the system of gas dynamics for viscous and heat conducting gases,
and for the system of thermoviscoelasticity with viscosity and heat-conduction. Existing differences in applying the relative entropy method between 
 the general hyperbolic-parabolic theory and the examples are underlined.
\end{abstract}
%
%
%
\tableofcontents

\bigskip
\bigskip

\section{Introduction}
\label{intro}

Consider a system of hyperbolic-parabolic conservation laws
\begin{equation}
\label{intro-hyppar}
\begin{aligned}
\del_t A(u) + \del_\alpha F_\alpha (u) = \eps \del_\alpha ( B_{\alpha \beta} (u) \del_\beta u)\;,
\end{aligned}
\end{equation}
where $u(t,x)$ takes values in $\R^n$, $t \in \R^+$, $x \in \R^d$ and $A , F_\alpha : \R^n \to \R^n $ , $B_{\alpha \beta} : \R^n \to \R^{n \times n}$ are given smooth functions with
$\alpha, \beta = 1, ..., d$. It is assumed that the associated hyperbolic problem
\begin{equation}
\label{intro-hypcl}
\begin{aligned}
\del_t A(u) + \del_\alpha F_\alpha (u) = 0
\end{aligned}
\end{equation}
 is symmetrizable in the sense of  Friedrichs and Lax \cite{FL71}.  The goal of this work is to extend the class of computations that go under the 
 general term \emph{relative entropy}   to the broader class of systems \eqref{intro-hypcl} and \eqref{intro-hyppar}. The main reason for pursuing this goal is that the problems that arise in applications are governed by systems in the form~\eqref{intro-hyppar} and it is important to understand whether and how the \emph{relative entropy method} can be employed in the general setting of~\eqref{intro-hyppar}.
 
 The idea of relative entropy, introduced by Dafermos \cite{dafermos79,dafermos79b} and DiPerna \cite{diperna79}, 
 is quite powerful in comparing solutions of conservation laws  ({\it e.g.}~\cite{diperna79,bds11,dst12,SV14}),
 or balance laws ({\it e.g.}~\cite{tzavaras05,MT14}),
 and has recently being applied to problems that are classified under the domain of hyperbolic-parabolic systems 
 ({\it e.g}~\cite{fn12,lt06,lt13,KV15}).
 The objective of this work is to systematize the derivation of relative entropy identities, referring to \eqref{intro-hyppar} as a unifying framework,
in order to connect the relative entropy theory with its natural framework, the $L^2$ theory of hyperbolic-parabolic systems of 
 Kawashima \cite{Kawashima84} and the developments on Green functions by Liu-Zeng \cite{LZ97},  and, even further, to exhibit  the intimate connection
 of this method with the framework of thermodynamics. We emphasize that the latter appears from the very early developments of
 the method  \cite{dafermos79,dafermos79b,iesan94}, but perhaps because it is cumbersome, it has not been always transparent in subsequent  developments 
 of the theory. We hope that our work by attempting to translate the thermodynamical structure to a partial differential equation framework will
 help strengthen this connection. In the second part of this work, we revisit the connection of relative entropy  to thermodynamics 
 in the context of the equations of gas dynamics
 with Newtonian viscosity and Fourier heat conduction, as well as in the context of the general constitutive theory of thermoviscoelasticity, whose thermodynamical structure 
 is specified in \cite{cn63,cm64, TN}. We point out crucial differences in these examples from the general framework developed in the first part and despite of that, we show that the resulting identities and their implications verify the connection of the method with thermodynamics.

 The class of systems~\eqref{intro-hyppar} and~\eqref{intro-hypcl} to which the \emph{relative entropy method} is extended here is characterized by the following hypotheses:
 \begin{enumerate}
 \item[($\text{H}_1$)] \emph{ $A:\R^n\to\R^n$ is a $C^2$ globally invertible map,}
 \item[($\text{H}_2$)] \emph{  existence of an entropy-entropy flux pair $(\eta,q)$, that is $\exists$ $G:\R^n\to\R^n$, $G=G(u)$ smooth such that
$$
\begin{aligned}
\nabla \eta &= G \cdot \nabla A
\\
\nabla q_\alpha &= G \cdot \nabla F_\alpha \, , \quad  \alpha = 1, ..., d \, ,
\end{aligned}
$$
}
 \item[($\text{H}_3$)] \emph{ the symmetric matrix $\nabla^2 \eta (u) - G(u) \cdot  \nabla^2 A (u)$ is positive definite,}
\listpart{and}
 \item[($\text{H}_4$)]  \emph{ the matrices $\nabla G(u)^T B_{\alpha \beta} (u)$ induce entropy dissipation to \eqref{intro-hyppar},  namely
 $$\sum_{\alpha,\beta} \xi_\alpha\cdot \big(\nabla G(u)^T B_{\alpha \beta} (u)\xi_\beta\big)\ge 0\qquad\forall\xi_\alpha,\xi_\beta\in\R^n,$$
 }
 \item[($\text{H}_5$)]  \emph{ or a different dissipative structure to \eqref{intro-hyppar} that is  $\exists$ $\mu>0$ such that
 $$  \sum_{\alpha,\beta}\nabla G(u)\del_\alpha u\cdot B_{\alpha\beta}(u)\del_\beta u\ge \mu\sum_\alpha|B_{\alpha\beta}(u)\del_\beta u|^2  .
 $$
 }
 \end{enumerate}
 Hypothesis ($\text{H}_1$)-($\text{H}_3$) are equivalent to the usual symmetrizability hypothesis in the sense of Friedrichs and Lax and therefore, they render system~\eqref{intro-hypcl} hyperbolic. The additional hypothesis ($\text{H}_4$) for the hyperbolic-parabolic systems~\eqref{intro-hyppar} guarantees that the entropy dissipates along the evolution. On the other hand, the other condition~\eqref{DH} induces dissipativity of different type, that allows degenerate viscosity matrices.

In this article, we extend the relative entropy method to this broader class of systems. More precisely, we derive the relative entropy identities for systems~\eqref{intro-hyppar} and~\eqref{intro-hypcl} and exploit these identities to derive significant properties of solutions such as uniqueness, stability and convergence. These are the identities~~\eqref{eqnrelen} and \eqref{relenidenhyp}, respectively,  in Section~\ref{sec-relen}.  It should be noted that theses hypotheses are presented in such a way that the reader can view their significance and utility in the framework of the relative entropy.

The article is divided into two parts. The first part is devoted in the development of the general theory of the relative entropy method mainly for the hyperbolic-parabolic systems~\eqref{intro-hyppar}. The second part is focused on two examples from thermodynamics that are presented in comparison to the general theory.
 
 The outline of the first part is the following: In Section \ref{secprel} we list the main hypotheses that make the relative entropy a workable quantity,
 and establish their connections to the theory  of symmetrizable systems \cite{FL71} and to the $L^2$ theory of hyperbolic-parabolic systems  \cite{Kawashima84}. 
 The relative entropy is defined via
 \begin{equation}
\eta ( u | \bu ) = \eta (u) - \eta (\bu) - \nabla \eta (\bu)  \cdot \nabla A (\bu)^{-1} (A(u) - A(\bu))
\end{equation}
and leads to the relative entropy identity \eqref{relenidenhyp} for system \eqref{intro-hypcl} in Section~\ref{secrelen}. The derivation is a straightforward extension of the classical works of Dafermos \cite{dafermos79,dafermos79b} and DiPerna~\cite{diperna79} when $A(u)=u$. However, the present calculation is useful for the reader to be on record and also, to what it comes next in the setting of the hyperbolic-parabolic systems~\eqref{intro-hyppar}.

The main calculation is performed in Section~\ref{secrelenhyppar} for the hyperbolic-parabolic systems~\eqref{intro-hyppar} that yields the associated relative entropy identity~\eqref{eqnrelen}. It is important that the terms in the identity are collected in a proper fashion that allows us to control them and prove theorems. Indeed, results are captured in the two following directions: In Theorem \ref{thmstability} of Section~\ref{sec-stab} we use the relative entropy to establish stability among bounded smooth solutions
 of hyperbolic-parabolic systems. Notably, a generalised version of the usual dissipative structure for hyperbolic-parabolic systems suffices to control the various error terms that appear. This is actually the role of Hypothesis~\eqref{hyp1}. Perhaps more important, in Theorem \ref{thmconv} of Section~\ref{sec-conv}, we prove a general convergence result in the zero-viscosity limit from 
 viscous system \eqref{intro-hyppar} to a smooth solution of inviscid system \eqref{intro-hypcl}, which is valid under fairly general a-priori bounds and even
 for degenerate viscosity matrices. This is accomplished by replacing~\eqref{hyp1} by hypothesis ~\eqref{DH}, which induces dissipativity of different nature and this was motivated and introduced by Dafermos in~\cite[Chapter IV]{daf10}. The significance of Theorem \ref{thmconv} is that the convergence obtained uses as ``metric" for measuring distance the relative entropy function. While this is not a metric, it operates as a combination of norms and even provides an $O(\eps)$ rate of convergence. There are alternative well developed techniques
for obtaining convergence results for the zero-viscosity limit to smooth solutions and in stronger norms - see for instance \cite[Ch V]{Kawashima84} - but the proof that we obtain is striking in its simplicity and generality.

Next, in Section~\ref{S2.2.2}, we consider dissipative measure valued solutions to hyperbolic problems for systems of conservation laws and balance laws respectively under appropriate growth conditions on constitutive functions. More precisely, we use the relative entropy identity to prove in Theorem \ref{thmweakstrong} a strong conservative solution of \eqref{intro-hypcl} is unique in the class of dissipative measure-valued solutions. Analogous theorems have been proved
 in \cite{bds11,dst12} and \cite{fnkt}; the present result is a technical extension of these works, and presents a comprehensive result
in the $L^p$ framework of approximate solutions to hyperbolic systems \eqref{intro-hypcl}. An interesting feature of the analysis is how concentration measures are
defined for a symmetrizable hyperbolic system \eqref{intro-hypcl} and the associated form of the averaged relative entropy identity. 
In Section~\ref{secrelenbal}, we study systems of balance laws and investigate the role of the source terms in the derivation 
of the relative entropy identity and in the proof  of the weak-strong uniqueness result. We remark that the proofs in Section~\ref{S2.2.2} are complemented by useful estimates established in Appendix~\ref{s-A}. These estimates allow us to control terms with respect to the relative entropy $\eta(u|\bu)$ and also, to view the relative entropy as a ``metric" measuring the distance between $u$ and $\bu$. This is actually expressed as a combination of $L^2$ and $L^p$ norms.

 The aforementioned results are quite general in nature, but require for their application a full-dissipative structure arising either from a positive definite viscosity matrix
 (see hypothesis \eqref{hyp1strong}) or the weaker condition~\eqref{DH} depending on the issue pursued. 
 It is well known that in most applications only
 a partial dissipative structure is available, and it typically originates from nonnegative but singular viscosity matrices. The second part of the article consists of the last two sections, in which we undertake this issue
in the context of specific applications. We present two interesting examples to observe how the general theory of Sections \ref{sec-relen} and~\ref{S2.2.2} applies and indicate the links to the associated thermodynamical structure. In Section \ref{sec-vhcg}, we take up the system of one dimensional gas-dynamics for viscous, heat conducting gases and derive the relative entropy formula for this system. This work is a special case of the next example, and for this reason, these calculations are helpful to the reader and serve as guidance for the following section.
In fact, in Section \ref{sec-thermov} we take up the system of thermoviscoelasticity in several space dimensions under its constitutive theory. We derive the relative entropy identity that is pertinent to this theory and describe 
how the general theory for hyperbolic-parabolic systems takes particular shape when applied to the constitutive theory of thermoviscoelasticity.
Related formulas in more special situations have been  computed  in \cite{fn12, feir15} for gases with Stokes viscosity and Fourier heat conduction
and in \cite{iesan94} for the constitutive theory of thermoelasticity. It should be indicated that for the systems studied in Sections  \ref{sec-vhcg} and \ref{sec-thermov},
the convexity of the entropy in the conserved variables  translates into the usual thermodynamic stability conditions
$\psi_{F F}(F,\theta) > 0$ and $\eta_\theta (F, \theta) > 0$ familiar from the work of Gibbs for a theory with thermal and elastic effects. Last, we establish analogous theorems to those of Sections~\ref{sec-relen} and~\ref{S2.2.2} in the context of thermoviscoelasticity. In particular, in Theorem~\ref{thmconvnonconductors}, we prove the convergence of weak solutions of the system of thermoviscoelasticity to the smooth solution of the system of thermoelastic nonconductors of heat as the parameters $\mu$, $k$ tend to zero. Then, 
in Theorem~\ref{thermoelweakstrong}, we establish uniqueness of strong solutions within the class of entropic-measure valued solutions to the system of adiabatic thermoelasticity.

Two of the major differences that arise in the relative entropy method between the examples and the general theory are: (a) Hypothesis~\eqref{DH} that is assumed to prove the convergence of the zero-viscosity limit for general hyperbolic-parabolic systems~\eqref{intro-hyppar} in Section~\ref{sec-conv} does not apply to the system of thermodynamics in Section~\ref{S4.conv}. One can confirm this by computing condition~\eqref{DH} in the case of the example. Perhaps a variant might apply, but in order to get the elegant condition~\eqref{m-klimit} imposed on the parameters $\mu$ and $k$ of Theorem~\ref{thmconvnonconductors}, one needs to work out the special case. (b) The role of concentration measure is different in the setting of thermodynamics from the workings in the general mv-weak versus strong uniqueness result in Section~\ref{S2.2.2}. As a matter of fact,
in the definition of entropic measure-valued solutions for the system of adiabatic thermoelasticity, in contrast to the theory of dissipative mv-solutions for general systems of conservation laws, a concentration measure appears in the energy conservation law  rather than in the Clausius-Duhem inequality describing the entropy production. The reason is that the estimates in the example are generated by the energy identity and not by the entropy inequality, which is typical in the example as contrasted to the general theory. This issue is pursued in Section~\ref{S4.uni}.

%
\section{Relative entropy for systems of hyperbolic parabolic conservation laws}
\label{sec-relen}

We consider the system of partial differential equations
\begin{equation}
\label{hyppar}
\begin{aligned}
\del_t A(u) + \del_\alpha F_\alpha (u) = \eps \del_\alpha ( B_{\alpha \beta} (u) \del_\beta u)
\end{aligned}
\end{equation}
where $u = u (t, x) : \R^+ \times \R^d \to \R^n$ and $n$, $d$ are integers representing  the number of the conserved quantities and 
the space dimension. The functions $A , F_\alpha : \R^n \to \R^n $ , $B_{\alpha \beta} : \R^n \to \R^{n \times n}$ are smooth,
$\alpha, \beta = 1, ..., d$ and system \eqref{hyppar} belongs to the general class of hyperbolic-parabolic systems. 

The summation convention over repeated indices is employed throughout this article and some computations may appear in extended coordinates for clarification.

The objective is to develop a relative entropy identity for hyperbolic-parabolic systems \eqref{hyppar}.
Hypotheses on the constitutive functions and the viscosity matrices will be placed and
guided by the goal of rendering this identity useful and applying it to some standard questions of stability, convergence and uniqueness. Also, comparisons are pursued with the Lax-Friedrichs theory of symmetrizable systems  and the Kawashima $L^2$-theory for hyperbolic-parabolic systems.
In later sections, specific systems from thermomechanics are reviewed in connection to these
general hypotheses.

\subsection{Hypotheses}
\label{secprel}

\subsubsection{Relative entropy for a hyperbolic system}
\label{sechyphyp}
Consider first the constituent system of conservation laws
\begin{equation}
\label{hypcl}
\begin{aligned}
\del_t A(u) + \del_\alpha F_\alpha (u) = 0 \, .
\end{aligned}
\end{equation}
It is assumed that $A : \R^n \to \R^n$ is {\it a $C^2$ map which is one-to-one} and satisfies
\begin{equation}
\label{hypns}
\begin{aligned}
&\quad \qquad \qquad  \nabla A (u) \quad \mbox{is nonsingular} \quad \forall u \in \R^n \, .
\end{aligned}
\tag{H$_1$}
\end{equation}
By the inverse function theorem the map $v=A(u)$ is locally invertible with the inverse map $u = A^{-1} (v)$
a $C^2$ map. By assumption \eqref{hypns} the map $v=A(u)$ is globally invertible and the set theoretic inverse coincides
and inherits the smoothness of the inverse induced by the inverse function theorem.

The system \eqref{hypcl} is endowed with an additional conservation law
\begin{equation}
\label{addcl}
\begin{aligned}
\del_t \eta (u) + \del_\alpha q_\alpha (u) = 0 \, .
\end{aligned}
\end{equation}
This structural hypothesis is rendered precise as follows:
The functions $\eta$-$q $, $q = (q_\alpha )$,  $\alpha = 1, ... , d$,  are called an entropy pair ($\eta$ is called entropy and 
$q = (q_\alpha )$, the associated entropy-flux)
 if there exists a smooth function
$G : \R^n \to \R^n$, $G = G(u)$,  such that simultaneously
\begin{equation}
\label{hypep}
\begin{aligned}
\nabla \eta &= G \cdot \nabla A
\\
\nabla q_\alpha &= G \cdot \nabla F_\alpha \, , \quad  \alpha = 1, ..., d \, .
\end{aligned}
\tag{H$_2$}
\end{equation}
If \eqref{hypep} is satisfied then smooth solutions of \eqref{hypcl} satisfy the additional identity \eqref{addcl}.
One checks that \eqref{hypep} is equivalent to requiring that $G$ satisfies the simultaneous equations
\begin{align}
\nabla G^T \,  \nabla A &= \nabla A^T \, \nabla G
\label{compat1}
\\
\nabla G^T \,  \nabla F_\alpha &= {\nabla F_\alpha}^T \, \nabla G  \, , \quad \alpha = 1, ... , d \, .
\label{compat2}
\end{align}
That is, if there exists a multiplier $G(u)$ satisfying \eqref{hypep} (equivalently \eqref{compat1}, \eqref{compat2})
then system \eqref{hypcl} is endowed with the additional conservation law \eqref{addcl}. It is clear that the system \eqref{compat1}, \eqref{compat2} is in general overdetermined; 
nevertheless, systems from mechanics naturaly inherit the entropy pair structure  from the second law of thermodynamics.

Given two solutions $u$, $\bu$ of \eqref{hypcl}, the relative entropy  is defined via 
\begin{equation}
\label{defrelen}
\begin{aligned}
\eta ( u | \bu ) &= \eta (u) - \eta (\bu) - G(\bu) \cdot (A(u) - A(\bu))
\\
&= \eta (u) - \eta (\bu) - \nabla \eta (\bu)  \cdot \nabla A (\bu)^{-1} (A(u) - A(\bu))\;,
\end{aligned}
\end{equation}
while the relative flux(es) by
\begin{equation}
\label{defrelfl}
\begin{aligned}
q_\alpha ( u | \bu ) &= q_\alpha (u) - q_\alpha (\bu) - G(\bu) \cdot (F_\alpha (u) - F_\alpha (\bu))
\\
&= q_\alpha (u) - q_\alpha (\bu) - \nabla \eta (\bu)  \cdot \nabla A (\bu)^{-1} (F_\alpha (u) - F_\alpha (\bu)).
\end{aligned}
\end{equation}

The formula \eqref{defrelen} will be used to estimate the distance between two solutions $u$ and $\bu$.
To make it amenable to analysis,  we note that $\nabla^2 \eta (u) - G(u) \cdot  \nabla^2 A (u)$ is symmetric
and require that it is positive definite, that is
\begin{equation}
\label{hyppd}
\xi \cdot \left ( \nabla^2 \eta (u) - G(u) \cdot  \nabla^2 A (u)  \right ) \xi > 0  \quad \mbox{for $\xi \in \R^n \setminus \{ 0 \} $}\, .
\tag{H$_3$}
\end{equation}
Here, we clarify that 
$$G(u) \cdot  \nabla A (u) :=\sum_{k=1}^nG_k(u)\nabla A_k(u)\in\R^n,\qquad G(u) \cdot  \nabla^2 A (u) :=\displaystyle\sum_{k=1}^n G_k(u)  \nabla^2 A_k (u) \in\R^{n\times n}$$ 
in~\eqref{hypep} and \eqref{hyppd}, respectively. We also add that the expression $\xi\cdot M\xi$, when $\xi\in\R^n$ and $M\in\R^{n\times n}$ is the dot product of the vectors $\xi$ and $M\xi$ and these notations are used repeatedly in this article.

Under \eqref{hyppd}, expression \eqref{defrelen} acquires some characteristics of a distance function
(without however being a metric) which render it useful
for comparing the distance between two solutions $u(t,x)$ and $\bu(t,x)$.
The definition of relative entropy and flux(es) given by \eqref{defrelen}--\eqref{defrelfl} extends to the case of system \eqref{hypcl} a well known definition pursued 
in  \cite{dafermos79,diperna79} for the case $A(u) =u$ with the same  objective of calculating the distance between two solutions.
A precursor to quantity \eqref{defrelen} appears in \cite{Kawashima84} for comparing a general solution $u(t,x)$ to a constant state $\bu$,
in connection to asymptotic behavior problems.

\subsubsection{Hypotheses for the hyperbolic-parabolic system}
\label{sechyphyppar}

Next we return to system \eqref{hyppar}, imposing hypotheses \eqref{hypns}, \eqref{hypep},
\eqref{hyppd} on the hyperbolic part, and examine the assumptions on the
viscosity matrices from the perspective of the development of a relative entropy identity.
Using the multiplier $G(u)$ in \eqref{hypep}, we deduce that 
smooth solutions of  \eqref{hyppar} satisfy the identity 
\begin{equation}
\label{hypparaen}
\begin{aligned}
\del_t  \eta (u) + \del_\alpha q_\alpha (u) = \eps \del_\alpha ( G(u) \cdot B_{\alpha \beta} (u) \del_\beta u) 
- \eps \nabla G(u) \del_{\alpha} u \cdot B_{\alpha \beta} (u) \del_{\beta} u \, .
\end{aligned}
\end{equation}
We will require that \eqref{hypparaen} induces a dissipative structure, namely that the following positive semi-definite structure holds true:
\begin{equation}
\label{hyp1}
\begin{aligned}
\sum_{\alpha,\beta=1}^d\xi_\alpha \cdot \left(\nabla G(u)^T B_{\alpha \beta} (u) \, \xi_\beta\right) =\sum_{\alpha,\beta=1}^d\sum_{i,j=1}^n\xi_\alpha^i D_{\alpha\beta}^{ij} \xi_\beta^j  &\ge 0  \qquad \qquad  \forall   \xi_\alpha,\,\xi_\beta \in \R^n,
\end{aligned}
\tag {H$_4$}
\end{equation}
where $D_{\alpha\beta}\doteq \nabla G(u)^T B_{\alpha \beta} (u)$, for $ \alpha, \beta = 1, ..., d$ and note that~\eqref{hyp1} is rewritten in extended coordinates for the convenience of the reader and the computations in Section~\ref{sec-stab}. In view of the identity~\eqref{hypparaen}, this guarantees that the entropy dissipates along the evolution. 

Hypothesis \eqref{hyp1} is natural in the context of applications to mechanics as it is connected to entropy
dissipation and the Clausius-Duhem inequality. The  analysis of Section \ref{secrelenhyppar} will detail
its relevance to the relative entropy calculation. When exploiting the identities in the abstract context of
\eqref{hyppar}, we often impose a strengthened version of \eqref{hyp1}, that is the strict positive definite case
\begin{equation}
\label{hyp1strong}
\begin{aligned}
\sum_{\alpha,\beta=1}^d\xi_\alpha \cdot \left(\nabla G(u)^T B_{\alpha \beta} (u) \, \xi_\beta\right)  &> 0  \qquad \qquad  \forall   \xi_\alpha,\,\,\xi_\beta \in \R^n\setminus\{0\}.
\end{aligned}
\tag {H$_{4s}$}
\end{equation}
The minimum and maximum eigenvalues of the associated quadratic form, for $u \in \R^n$,
\begin{equation}
\label{defnun}
\begin{aligned}
\nu (u) &= \inf\Big\{\sum_{\alpha,\beta} \xi_\alpha \cdot \left(\nabla G(u)^T B_{\alpha \beta} (u) \, \xi_\beta\right)\, \Big|\,\xi_\alpha,\xi_\beta \in \R^n , \,|\xi_\alpha| =|\xi_\beta|= 1  \Big\} \;  > 0 \,,  \quad
\\
N(u) &= \sup\Big\{ \sum_{\alpha,\beta}  \xi_\alpha \cdot \left(\nabla G(u)^T B_{\alpha \beta} (u) \, \xi_\beta\right)\, \Big|\,\xi_\alpha,\xi_\beta \in \R^n ,\, |\xi_\alpha| =|\xi_\beta|= 1 \Big\} \;   < \infty \, , 
\end{aligned}
\end{equation}
may be used to express \eqref{hyp1strong} in an equivalent (more quantitative) format:
\begin{equation}
\label{hyp1sp}
0  < \nu(u) \sum_\alpha|\xi_\alpha|^2 \le \sum_{\alpha,\beta}\xi_\alpha \cdot\left( \nabla G(u)^T B_{\alpha \beta} (u) \, \xi_\beta\right) \le N(u)\sum_\alpha |\xi_\alpha|^2   \quad \forall   \xi_\alpha,\xi_\beta \in \R^n\setminus\{0\}.
\tag {H$_{4p}$}
\end{equation}

This hypothesis requires that the viscosity matrices are non-degenerate and it will be assumed for establishing stability of viscous solutions in Section~\ref{sec-stab}. Another hypothesis of different nature is set in Section~\ref{sec-conv} that allows degenerate viscosity matrices  to be considered in the zero-viscosity limit and that one, called~\eqref{DH}, will replace~\eqref{hyp1}. Its presentation is postponed for Section~\ref{sec-conv}.

\subsubsection{ A convex entropy for an equivalent system}
\label{sechypconv}

Here, we compare hypotheses \eqref{hypep} and \eqref{hyppd} with the familiar notion of symmetrizable 
first-order systems of Friedrichs and Lax \cite{FL71}. 
Hypothesis \eqref{hypns} is a standing assumption that guarantees the transformation 
 $v = A(u)$ is invertible, and the system \eqref{hypcl} and \eqref{addcl} can be expressed  in terms of the conserved variables $v$,
\begin{align}
\label{sys2}
\del_t v + \del_\alpha (F_\alpha \circ A^{-1}) (v) &= 0 
\\
\del_t (\eta \circ  A^{-1}) (v)  + \del_\alpha ( q_\alpha \circ  A^{-1}) (v)   &= 0\;.
\label{ensys2}
\end{align}
Setting
\begin{equation}
\label{formdef}
f_\alpha (v) = F_\alpha \circ A^{-1} (v) \, , \quad H(v) = \eta \circ  A^{-1} (v) \, , \quad
Q_\alpha (v) = q_\alpha \circ  A^{-1} (v) \, .
\end{equation}
we obtain the formulas
\begin{equation}
\label{formdefin}
\eta (u) = H(A(u) ) \, , \quad q_\alpha(u) = Q_\alpha (A(u) ) \, .
\end{equation}
They readily yield
\begin{equation}
\label{appform1}
\begin{aligned}
\nabla_u \eta (u) &= (\nabla_v H ) (A(u)) \cdot \nabla_u A(u) \;,
\\
\nabla_u q_\alpha (u) &= (\nabla_v Q_\alpha ) (A(u)) \cdot \nabla_u A(u)
\end{aligned}
\end{equation}
and
\begin{equation}
\label{appform2}
\nabla^2_u \eta (u) =  (\nabla_v^2 H)(A(u)) : ( \nabla_u A(u) , \nabla_u A(u) ) + (\nabla_v H ) (A(u)) \cdot \nabla^2_u A(u) \;.
\end{equation}

Suppose that  $\eta - q $  is an entropy pair for \eqref{hypcl} satisfying \eqref{hypep}. 
If the flux $f_\alpha$ and the pair $H - Q$  are defined by  \eqref{formdef} then 
\eqref{appform1} implies
\begin{align}
\label{appform3}
G(u) =  ( \nabla_v H ) (A(u))  
\\
\label{entropyv}
\nabla_v Q_\alpha (v) = \nabla_v H (v) \cdot \nabla_v f _\alpha (v) \, ,
\tag{h$_2$}
\end{align}
i.e. the pair $H - Q$ is an entropy pair for \eqref{sys2} with the multiplier
$G(u)$ is defined via \eqref{appform3}.  By \eqref{appform2}, Hypothesis \eqref{hyppd} translates 
to the requirement  that the entropy  $H(v)$ is convex,
\begin{equation}
\label{hypvucprime}
\zeta \cdot \nabla_v^2 H (v)  \zeta > 0  \quad \mbox{for $\zeta \in \R^n, \; \zeta \ne 0$}  \, .
\tag{h$_3$}
\end{equation}

Conversely, if \eqref{sys2} is endowed with an  entropy pair  $H - Q$ with $H$ convex, 
then  $\eta - q$ defined via \eqref{formdefin} is an entropy pair   for \eqref{hypcl}
where $F(u) := f(A(u))$ and $G(u)$ is selected via \eqref{appform3}. 
The convexity assumption \eqref{hypvucprime} for $H(v)$ translates  to \eqref{hyppd} for $\eta(u)$.
In summary,  \eqref{hypns}, \eqref{hypep} and \eqref{hyppd} are equivalent 
to the usual symmetrizability hypothesis of \cite{FL71}. In particular, they render system \eqref{hypcl} hyperbolic.

Regarding next the hyperbolic-parabolic system~\eqref{hyppar}, it is instructive to compare the structural hypotheses pursued here to the $L^2$ theory of hyperbolic-parabolic systems. As already mentioned, hypothesis~\eqref{hyp1} on the diffusion coefficients render the last term of~\eqref{hypparaen} as semi-positive definite. 
We refer the reader to Kawashima \cite{Kawashima84,Kawashima87} for the early developments and to Liu-Zeng \cite{LZ97} for
the connection to Green's functions for hyperbolic-parabolic systems as well as to Kawashima \cite[Ch~II]{Kawashima84} and Dafermos \cite{daf10} for
results concerning well-posedness of hyperbolic-parabolic systems. A detailed exposition on the structure of dissipative viscous systems is given by Serre in the articles~\cite{S-IMA,S-Cont,S-PhyD} and the symmetry of the dissipation tensor $D_{\alpha\beta}=\nabla G(u)^T B_{\alpha \beta} (u)$ is investigated in connection to the Onsager's principle. We only note here that effective diffusion matrices $D_{\alpha\beta}$ should at least satisfy the well-known Kawashima condition, which guarantees that waves of all characteristic families are properly damped. This can be deduced from hypothesis~\eqref{hyp1} by setting $\xi_\alpha=\nu_\alpha R_i$ with $R_i$ standing for the right eigenvector associated to the $i$-characteristic speed of system~\eqref{hyppar} and $\vec{\nu}\in S^{d-1}$.

\subsection{The relative entropy identity for systems of hyperbolic conservation laws}
\label{secrelen}

In this section, subject to  hypotheses \eqref{hypns}, \eqref{hypep} and \eqref{hyppd},  we extend to the hyperbolic system \eqref{hypcl} a well known calculation developed 
in \cite{dafermos79,diperna79} for the case $A(u) = u$. We note that the calculation here is shown only because it is useful and in preparation for what it comes in the next subsection on hyperbolic-parabolic systems.  We will also return to this in Section \ref{S2.2.2} for proving uniqueness of dissipative measure-valued solutions within the family of strong solutions. Moreover, this calculation is formal, it can however be made rigorous following ideas that are well developed
(see {\it e.g.} \cite[Ch V]{daf10}) and provides a way of comparing a weak entropic to a strong solution of~\eqref{hypcl}.
There exist variants of this calculation that compare entropic measure valued solutions to strong solutions of
hyperbolic conservation laws (see \cite{bds11, dst12}). 

Let $u$ be an entropy weak solution of \eqref{hypcl}, that is $u$ is a weak solution of \eqref{hypcl}
that satisfies in the sense of distributions the inequality
\begin{equation}
\label{entropyhyp}
\begin{aligned}
\del_t \eta (u) + \del_\alpha q_\alpha (u) \le  0 \, .
\end{aligned}
\end{equation}
Let $\bu$ be a strong (conservative) solution of \eqref{hypcl} that is satisfying 
the entropy identity 
\begin{equation}
\label{entropyhypb}
\begin{aligned}
\del_t \eta (\bu) + \del_\alpha q_\alpha (\bu) =  0\,.
\end{aligned}
\end{equation}

We proceed to compute the relative entropy identity for the quantities relative entropy \eqref{defrelen} and relative flux \eqref{defrelfl}.
Observe first that $u$, $\bu$ satisfy the chain of identities
\begin{equation}
\label{diffhyp}
\begin{aligned}
\del_t \Big ( G(\bu) \cdot (A(u) - A(\bu) )& \Big ) + \del_\alpha \Big  ( G(\bu) \cdot ( F_\alpha (u) - F_\alpha (\bu) ) \Big )
\\
&= \nabla G (\bu) \del_t \bu \cdot (A(u) - A(\bu)) 
+
      \nabla G (\bu) \bu_{x_\alpha} \cdot (F_\alpha (u) -  F_\alpha (\bu)  )
\\
&\stackrel{\eqref{hypcl},\eqref{hypns}}{=} - \bu_{x_\alpha} \cdot \nabla F_\alpha (\bu)^T \nabla A (\bu)^{-T} \nabla G( \bu)^T (A(u) - A(\bu)) 
\\
&\qquad 
   + \nabla G(\bu) \bu_{x_\alpha} \cdot (F_\alpha (u) - F_\alpha (\bu))
\\
&\stackrel{\eqref{compat1}}{=}
-  \nabla G(\bu) \bu_{x_\alpha} \cdot \nabla F_\alpha (\bu) \nabla A (\bu)^{-1}  (A(u) - A(\bu)) 
\\
&\qquad
   + \nabla G(\bu) \bu_{x_\alpha} \cdot (F_\alpha (u) - F_\alpha (\bu))
\\
&=:  \nabla G(\bu) \bu_{x_\alpha} \cdot  F_\alpha (u | \bu)\,,
\end{aligned}
\end{equation}
where
\begin{equation}
\label{relflux}
F_\alpha (u | \bu) := F_\alpha (u) - F_\alpha (\bu) - \nabla F_\alpha (\bu)  \nabla A (\bu)^{-1} (A(u) - A(\bu)) \,.
\end{equation}

Combining \eqref{entropyhyp}, \eqref{entropyhypb} and \eqref{diffhyp}, we obtain
\begin{equation}
\label{relenidenhyp}
\del_t  \eta (u | \bu)  + \del_{\alpha} q_\alpha (u | \bu) \le  -   \del_\alpha G (\bu)  \cdot F_\alpha (u | \bu)\,.
\end{equation}
This is the relative entropy inequality associated with the hyperbolic system~\eqref{hypcl} under hypotheses~\eqref{hypns}, \eqref{hypep} and \eqref{hyppd} and the aim is to generalize the above calculation is the presence of viscosity matrices in the next subsection. We emphasize that the above computations via the change of variables $v=A(u)$ is identical to the classical resuls~\cite{dafermos79,diperna79}. However, this is not true for the computations on the hyperbolic-parabolic systems~\eqref{hyppar}.


\subsection{The relative entropy identity for hyperbolic-parabolic systems}
\label{secrelenhyppar}

In this section, the relative entropy calculation is extended between two solutions $u$ and $\bu$ of 
\begin{equation}
\label{hyppara}
\begin{aligned}
\del_t A(u) + \del_\alpha F_\alpha (u) = \eps \del_\alpha ( B_{\alpha \beta} (u) \del_\beta u) \, .
\end{aligned}
\end{equation}
We work under the framework \eqref{hypns}--\eqref{hyppd} and assume that
the viscosity matrices  $D_{\alpha\beta}\doteq \nabla G(u)^T B_{\alpha \beta} (u)$ satisfy hypothesis~\eqref{hyp1}, which guarantees entropy dissipation \eqref{hypparaen} along the evolution.

We consider two solutions $u$ and $\bu$ of \eqref{hyppara} which also satisfy  \eqref{hypparaen} and as before, all calculations will be
performed for strong solutions. However, the reader can easily check that one can assume that 
$u$ is a weak entropy solution while $\bu$ is a strong solution and the calculation can still be
derived in the sense of distributions. This might be useful for problems involving positive semi-definite diffusion matrices where 
one in general expects global existence for weak solutions only.

Subtracting the entropy identities \eqref{hypparaen} for the respective solutions we get 
\begin{equation}
\label{eqn1}
\begin{aligned}
\del_t (\eta(u) - \eta(\bu)) &+ \del_\alpha (q_\alpha (u) - q_\alpha (\bu) ) 
\\
&=  
\eps  \del_\alpha \big ( G(u) \cdot B_{\alpha \beta} (u) \del_\beta u  - G(\bu) \cdot B_{\alpha \beta}(\bu) \del_\beta \bu \big ) 
\\
&\quad -\eps \nabla G(u) u_{x_\alpha} \cdot B_{\alpha \beta} (u) u_{x_\beta}  
+ \eps \nabla G(\bu) \bu_{x_\alpha} \cdot B_{\alpha \beta}(\bu) \bu_{x_\beta}\;.
\end{aligned}
\end{equation}

Proceeding along the steps of the derivation of  \eqref{diffhyp} from starting point the equation  \eqref{hyppara}  we derive the identity 
\begin{equation}
\label{eqn2}
\begin{aligned}
&\del_t \Big ( G(\bu) \cdot (A(u) - A(\bu) ) \Big ) + \del_\alpha \Big  ( G(\bu) \cdot ( F_\alpha (u) - F_\alpha (\bu) ) \Big )
\\
&= \del_t  (G (\bu) ) \cdot  (A(u) - A(\bu)) 
+  \del_\alpha (  G (\bu)) \cdot  (F_\alpha (u) -  F_\alpha (\bu)  )
\\
&\quad + \eps G(\bu) \cdot \del_\alpha \big [ B_{\alpha \beta} (u) \del_\beta u - B_{\alpha \beta} (\bu) \del_\beta \bu \big]
\\
&=   \nabla G(\bu) \bu_{x_\alpha} \cdot  F_\alpha (u | \bu) 
\\
&\quad
+ \eps \nabla G(\bu) (\nabla A(\bu))^{-1} \del_\alpha \Big ( B_{\alpha \beta} (\bu) \del_\beta \bu \Big )  \cdot (A(u) - A(\bu) )
\\
&\quad + \eps G(\bu) \cdot \del_\alpha \big [ B_{\alpha \beta} (u) \del_\beta u - B_{\alpha \beta} (\bu) \del_\beta \bu \big]\;.
\end{aligned}
\end{equation}

We subtract \eqref{eqn2} from \eqref{eqn1} and use \eqref{defrelen},  \eqref{defrelfl}, \eqref{compat1}, \eqref{compat2} and an integration by parts
to obtain
\begin{equation}
\label{eqn3}
\begin{aligned}
\del_t  \eta (u | \bu)  + \del_{\alpha} q_\alpha (u | \bu) = - \nabla G (\bu)  (\del_{\alpha} \bu ) \cdot F_\alpha (u | \bu)
+ \eps \del_{x_\alpha} J_\alpha + \eps K\;,
\end{aligned}
\end{equation}
where the flux $J_\alpha$ and the term $K$ are defined by
\begin{align}
J_\alpha &:= G(u) \cdot B_{\alpha \beta} (u) u_{x_\beta} - G(\bu) \cdot B_{\alpha \beta} (\bu) \bu_{x_\beta}
\nonumber
\\ 
&\quad - G(\bu) \cdot ( B_{\alpha \beta}(u) u_{x_\beta} - B_{\alpha \beta} (\bu) \bu_{x_\beta} )
\nonumber
\\
&\quad - B_{\alpha \beta} (\bu) \bu_{x_\beta} \cdot \nabla A(\bu)^{-T} \nabla G(\bu)^T (A(u) - A(\bu)) \;,
\label{defin10}
\\
K &= - \nabla G(u) u_{x_\alpha} \cdot B_{\alpha \beta}(u)  u_{x_\beta} 
    +  \nabla G(\bu) \bu_{x_\alpha} \cdot B_{\alpha \beta}(\bu)  \bu_{x_\beta}
\nonumber
\\ 
&\quad + \nabla G (\bu) \bu_{x_\alpha} \cdot \big ( B_{\alpha \beta}(u)  u_{x_\beta} - B_{\alpha \beta}(\bu)  \bu_{x_\beta} \big )
\nonumber
\\
&\quad + B_{\alpha \beta}(\bu)  \bu_{x_\beta} \cdot  \nabla G(\bu) \del_\alpha \big ( \nabla A (\bu)^{-1} (A(u) - A(\bu) ) \big )
\nonumber
\\
&\quad + B_{\alpha \beta}(\bu)  \bu_{x_\beta} \cdot  \nabla^2 G(\bu)  : \Big ( \bu_{x_\alpha} ,  \nabla A (\bu)^{-1} (A(u) - A(\bu) )  \Big ) \, .
\label{defin11}
\end{align}
Equation \eqref{eqn3} is the basic relative entropy identity. In the sequel we rearrange the terms so that
 $J_\alpha$ and $K$ are expressed in a more revealing form.

First, using \eqref{compat1}, we rewrite
\begin{equation}
\label{defvf1}
\begin{aligned}
J_\alpha 
&= (G(u) - G(\bu)) \cdot ( B_{\alpha \beta} (u) u_{x_\beta} - B_{\alpha \beta} (\bu) \bu_{x_\beta}  )
\\
&\qquad +  B_{\alpha \beta} (\bu) \bu_{x_\beta} \cdot  \Big [ G(u) - G(\bu) - \nabla G(\bu) \nabla A(\bu)^{-1} (A(u) - A(\bu) ) \Big ]
\\
&= (G(u) - G(\bu)) \cdot ( B_{\alpha \beta} (u) u_{x_\beta} - B_{\alpha \beta} (\bu) \bu_{x_\beta}  ) + 
B_{\alpha \beta} (\bu) \bu_{x_\beta} \cdot G(u | \bu)\;,
\end{aligned}
\end{equation}
where $G(u|\bu)$ is defined by
\begin{equation}
\label{relmult}
G(u| \bu) := G(u) - G(\bu) - \nabla G (\bu) \nabla A(\bu)^{-1} (A(u) - A(\bu)) \, 
\end{equation}
and  is of order $O(|u - \bu|^2)$ as $| u - \bu| \to 0$.

The quantity $K$ in \eqref{defin11} is rewritten as
\begin{align}
K &= - \Big ( \nabla G(u) u_{x_\alpha} - \nabla G (\bu) \bu_{x_\alpha} \Big ) \cdot 
\Big ( B_{\alpha \beta}(u)  u_{x_\beta} - B_{\alpha \beta}(\bu)  \bu_{x_\beta} \Big )
\nonumber
\\
&\quad + B_{\alpha \beta}(\bu)  \bu_{x_\beta} \cdot \Big [  - \nabla G(u) u_{x_\alpha} +  \nabla G (\bu) \bu_{x_\alpha}
\nonumber
\\
&\qquad \qquad  \qquad \qquad   + \nabla G(\bu)  \del_\alpha \big ( \nabla A (\bu)^{-1} (A(u) - A(\bu) ) \big ) 
\nonumber
\\
&\qquad \qquad  \qquad \qquad  + \nabla^2 G(\bu) : \big ( \bu_{x_\alpha} , \nabla A (\bu)^{-1} (A(u) - A(\bu)  )\big )  \Big ]
\nonumber
\\
&= : T_1 + B_{\alpha \beta}(\bu)  \bu_{x_\beta} \cdot T_{\alpha , 2 }\;.
\label{eqn101}
\end{align}
Observe that term $ T_{\alpha , 2 }$ is rewritten as
\begin{equation}
\label{eqn102}
\begin{aligned}
T_{\alpha , 2 } &=  - \nabla G(u) u_{x_\alpha} +  \nabla G (\bu) \bu_{x_\alpha}
\\
&\quad + \nabla G(\bu)  \del_\alpha \big ( \nabla A (\bu)^{-1} (A(u) - A(\bu) )  - (u - \bu) \big ) 
\\
&\quad +\nabla G(\bu) \del_{\alpha} (u - \bu) + \nabla^2 G(\bu) : \big ( \bu_{x_\alpha} , \nabla A (\bu)^{-1} (A(u) - A(\bu)  )\big ) 
\\
&= \nabla G (\bu) \del_{\alpha} \phi(u | \bu) -  \big ( \nabla G(u) -  \nabla G (\bu) \big ) ( u_{x_\alpha} - \bu_{x_\alpha})
- L(u | \bu) \bu_{x_\alpha}\,,
\end{aligned}
\end{equation}
where we set
\begin{align}
\label{defin1}
\phi (u | \bu) &:= \nabla A (\bu)^{-1} (A(u) - A(\bu) )  - (u - \bu)\;,
\\
\label{defin2}
L(u | \bu) &:= \nabla G(u) - \nabla G (\bu) - \nabla^2 G (\bu) \cdot \nabla A(\bu)^{-1} (A(u) - A(\bu) )\;,
\end{align}
and note that both terms are quadratic as $|u - \bu| \to 0$. In turn,
\begin{align}
\label{eqn103}
B_{\alpha \beta} (\bu)  \bu_{x_\beta} \cdot T_{\alpha , 2 } &=
\del_{x_\alpha} \Big ( B_{\alpha \beta}(\bu)  \bu_{x_\beta} \cdot \nabla G (\bu)  \phi(u | \bu) \Big ) 
- \del_{x_\alpha} \big ( \nabla G (\bu)^T B_{\alpha \beta}(\bu)  \bu_{x_\beta} \big ) \cdot \phi(u | \bu) 
\nonumber
\\
&\quad -  B_{\alpha \beta}(\bu)  \bu_{x_\beta} \cdot \big ( \nabla G(u) -  \nabla G (\bu) \big ) ( u_{x_\alpha} - \bu_{x_\alpha})
- B_{\alpha \beta}(\bu)  \bu_{x_\beta} \cdot L(u | \bu) \bu_{x_\alpha}
\nonumber
\\
&= \del_{x_\alpha} j_\alpha + Q_1 + Q_2 + Q_3\,,
\end{align}
where we have set 
\begin{equation}
\label{defvf3}
j_\alpha := B_{\alpha \beta}(\bu)  \bu_{x_\beta} \cdot \nabla G (\bu)  \phi(u | \bu)\;,
\end{equation}
and
\begin{align}
Q_1 &:= - \del_{x_\alpha} \big ( \nabla G (\bu)^T B_{\alpha \beta}(\bu)  \bu_{x_\beta} \big ) \cdot \phi(u | \bu) \;,
\label{definq1}
\\
Q_2 &:= -  B_{\alpha \beta}(\bu)  \bu_{x_\beta} \cdot \big ( \nabla G(u) -  \nabla G (\bu) \big ) ( u_{x_\alpha} - \bu_{x_\alpha})\;,
\label{definq2}
\\
Q_3 &:= - B_{\alpha \beta}(\bu)  \bu_{x_\beta} \cdot L(u | \bu) \bu_{x_\alpha}\,.
\label{definq3}
\end{align}
On the other hand the term $T_1$ is rearranged as
\begin{equation}
\label{eqn107}
\begin{aligned}
T_1 &= - \Big ( \nabla G(u) u_{x_\alpha} - \nabla G (\bu) \bu_{x_\alpha} \Big ) \cdot 
\Big ( B_{\alpha \beta}(u)  u_{x_\beta} - B_{\alpha \beta}(\bu)  \bu_{x_\beta} \Big )
\\
&=
-  \nabla G (u)  \del_\alpha ( u - \bu) \cdot   B_{\alpha \beta} (u) \, \del_{\beta} (u - \bu)
+ Q_4 + Q_5+ Q_6
\\
&= - D +Q_4 + Q_5+ Q_6\;,
\end{aligned}
\end{equation}
where
\begin{align}
Q_4 &:= - \nabla G(u) ( u_{x_\alpha} - \bu_{x_\alpha} ) \cdot ( B_{\alpha \beta} (u) - B_{\alpha \beta} (\bu) ) \bu_{x_\beta}
\label{definq4}
\\
Q_5 &:= -  ( \nabla G(u)  - \nabla G (\bu) ) \bu_{x_\alpha} \cdot B_{\alpha \beta} (u) ( u_{x_\beta} - \bu_{x_\beta} )
\label{definq5}
\\
Q_6 &:=  -  ( \nabla G(u)  - \nabla G (\bu) ) \bu_{x_\alpha} \cdot ( B_{\alpha \beta} (u) - B_{\alpha \beta} (\bu) ) \bu_{x_\beta}
\label{definq6}
\end{align}
are quadratic terms (viewed as errors) while by virtue of \eqref{hyp1} the term $D$ is positive semi-definite, 
$$
D := \sum_{\alpha,\,\beta}\nabla G (u)  \del_\alpha ( u - \bu) \cdot   B_{\alpha \beta} (u) \, \del_{\beta} (u - \bu) \;  \ge  \; 0 \, , 
$$
and captures the effect of dissipation by taking $\xi_\alpha=\del_\alpha(u-\bu)\in\R^n$.

We conclude that the term $K$, defined in \eqref{defin11},  can be  reorganized using \eqref{eqn101}, \eqref{eqn103} and \eqref{eqn107} in the form
$$
K =  \del_{x_\alpha} j_\alpha - D +  Q_1 + Q_2 + Q_3 + Q_4 + Q_5+ Q_6\,.
$$

Putting everything together we obtain the final form of the relative entropy identity
\begin{equation}
\label{eqnrelen}
\begin{aligned}
\del_t  \eta (u | \bu)  &+ \del_{\alpha} q_\alpha (u | \bu)  
+ \eps \nabla G (u)  \del_\alpha ( u - \bu) \cdot   B_{\alpha \beta} (u) \, \del_{\beta} (u - \bu)
\\
&= - \nabla G (\bu)  (\del_{\alpha} \bu ) \cdot F_\alpha (u | \bu)
+ \eps \del_{x_\alpha} (  J_\alpha + j_\alpha)  +  \eps \sum_{i=1}^6 Q_i
\end{aligned}
\end{equation}
where the relative entropy is defined in \eqref{defrelen}, the relative flux in \eqref{defrelfl}, the terms $F_\alpha(u|\bu)$ in~\eqref{relflux}
the viscous fluxes $J_\alpha$ and $j_\alpha$ in \eqref{defvf1},   \eqref{defvf3}, 
and the quadratic "error" terms $Q_i$ in \eqref{definq1}, \eqref{definq2}, \eqref{definq3}, \eqref{definq4}, \eqref{definq5}
and \eqref{definq6} respectively. We will see the significance of the above identity in the following two subsections in establishing theorems for comparing to solutions $u$ and $\bu$ and the importance of having all the terms on right-hand side expressed in this particular form.

\subsection{Stability of viscous solutions}
\label{sec-stab}
The objective in this section is to prove stability of smooth solutions to~\eqref{hyppara} using the relative entropy identity~\eqref{eqnrelen}. 
Once the identity \eqref{eqnrelen} is established, the remainder of the theorem is  an extension to the case of hyperbolic-parabolic systems  of 
theorems regarding stabilization of hyperbolic systems ({\it e.g.} \cite[Thm 5.3.1]{daf10}). Since we will work on the whole space $\R^d\times [0,T]$, we assume that the solutions decay as $|x| \to \infty$; such hypotheses are usually validated as part of an existence theory, but we do not
pursue this aspect here. We also remark that one may  easily extend the theorem below to the case of periodic solutions: $\bar Q_T=\To^d\times[0,T]$.
Moreover, we denote by $B_M$ the bounded set $B_M=\{\bu\in\R^n: |\bu|\le M\}$ of size $M>0$ and in applications 
of the relative entropy method one of the functions (or sometimes both) takes values in $B_M$.

The following theorem establishes the $L^2$-stability of viscous solutions w.r.t. initial data, when both solutions are bounded under the hypotheses assumed in Section~\ref{secprel}

\begin{theorem}
\label{thmstability}
Let $u$, $\bar u$ be smooth solutions of~\eqref{hyppara}  defined on $\R^d\times[0,T]$  such that $u$, $\del_\alpha u$ and $\bu$, $\del_\alpha \bu$ decay 
sufficiently fast at infinity, and emanating from smooth initial data $u_0$, $\bar u_0$, with $u_0, \bu_0 \in (L^\infty \cap L^2 )(\R^d)$. 
Assume the hypotheses~\eqref{hypns},~\eqref{hypep} ~\eqref{hyppd}  and~\eqref{hyp1strong} hold true 
and suppose that $u$ and $\bar u$ take values in a ball $B_M\subset\R^n$ of radius $M>0$. Then there exists a constant 
$C=C(T,\gamma,M,\nabla\bar u)>0$ independent of $\eps$,  such that
\begin{equation}\label{S2.3.1stability}
\|u(t)-\bar u(t)\|_{L^2(\R^d)}\le C_T  \|u_0-\bar u_0\|_{L^2(\R^d)}\;.
\end{equation}
\end{theorem}
\begin{proof}
Let us set 
\begin{equation}
\label{defphi}
\varphi(t)\doteq\int_{\R^d}\eta(u(t)|\bar u(t))\,dx\qquad t\in[0,T].
\end{equation}
As $u$, $\bu$ take values in $B_M$, hypothesis  \eqref{hyppd} implies 
that, for some $c  = c(M) > 0$,
\begin{equation}
\label{hyppdprime}
\nabla^2 \eta (u) - G(u) \cdot  \nabla^2 A (u) \ge c I > 0 \, ,
\tag{H$_3^\prime$}
\end{equation}
while hypothesis \eqref{hyp1strong} implies, for some $\gamma = \gamma (M) > 0$,
\begin{equation}
\label{hyp4prime}
\sum_{\alpha,\beta}\xi_\alpha\cdot \nabla G(u)^T B_{\alpha \beta} (u)\xi_\beta \ge \gamma \sum_\alpha|\xi_\alpha|^2  >0\quad \forall \xi_\alpha, \xi_\beta\in\R^n\setminus \{0\}.
\tag {H$_4^\prime$}
\end{equation}
Moreover, there exists a positive constant $C=C(c,M)$ such that
\begin{equation}\label{S2.33.1L2}
\frac{1}{C} \| A(u ) - A(\bar u) \|_{L^2(\R^d)}^2  \le\varphi(t)\le C \| A(u ) - A(\bar u) \|_{L^2(\R^d)}^2 \, 
\end{equation}
and, by \eqref{hypns},  this implies
\begin{equation}\label{S2.3.1L2}
\frac{1}{C} \|u(t)-\bar u(t)\|_{L^2(\R^d)}^2  \le\varphi(t)\le C\|u(t)-\bar u(t)\|_{L^2(\R^d)}^2.
\end{equation}

We employ~\eqref{eqnrelen} and integrate over $\R^d\times[0,t]$ to get
\begin{equation}\label{S2.3.1 phi}
\begin{aligned}
\varphi(t)+\eps\gamma\int_0^t\int_{\R^d}|\nabla u(s)-\nabla\bar u(s)|^2\,dx\,ds\le&\varphi(0)-\int_0^t\int_{\R^d}\nabla G (\bu)  (\del_{\alpha} \bu ) \cdot F_\alpha (u(s) | \bu(s))\,dx\,ds\\
&+\eps\sum_{i=1}^6 \int_0^t\int_{\R^d} Q_i(x,s)\,dx\,ds\;.
\end{aligned}
\end{equation}
Now, we investigate the terms on the right-hand side of~\eqref{S2.3.1 phi} using that $u,\bu\in B_M$. By~\eqref{relflux}, \eqref{defin1}, \eqref{defin2} and the terms~\eqref{definq1}--\eqref{definq3},~\eqref{definq4}--\eqref{definq6}, we get the bounds
\begin{equation}
\left|\int_{\R^d}\nabla G (\bu)  (\del_{\alpha} \bu ) \cdot F_\alpha (u(s) | \bu(s))\,dx  \right|\le C \|u(s)-\bar u(s)\|_{L^2(\R^d)}^2,
\end{equation}
\begin{equation}
\left| \int_{\R^d}  Q_1(s)+Q_3(s)+Q_6(s)\, dx         \right|\le C\|u(s)-\bar u(s)\|_{L^2(\R^d)}^2,
\end{equation}
\begin{equation}\label{S2.3.1Q5}
\left| \int_{\R^d} Q_2(s)+Q_4(s)+Q_5(s) dx         \right|\le C\|u(s)-\bar u(s)\|_{L^2(\R^d)}^2+\gamma\int_{\R^d}|\nabla u(s)-\nabla\bar u(s)|^2\,dx,
\end{equation}
for $s\in[0,t]$ having $C$ a universal constant depending on $c$, the radius $M$, $\gamma$ and the derivatives of $\bar u$. Combining~\eqref{S2.3.1L2}--\eqref{S2.3.1Q5}, we arrive at
\begin{equation}
\|u(t)-\bar u(t)\|_{L^2(\R^d)} ^2\le C\left(\|u_0-\bar u_0\|_{L^2(\R^d)}^2+(1+\eps)\int_0^t\|u(s)-\bar u(s)\|_{L^2(\R^d)}^2\,ds\right)
\end{equation}
 and conclude the result via Gronwall's lemma for $0<\eps<\eps_0$. Note that the constant $C$ in~\eqref{S2.3.1stability} will also depend on time $T$ and $\eps_0$.
\end{proof}

\subsection{Convergence  in the zero-viscosity limit} 
\label{sec-conv}
In this section, we consider a family  $\{ u_\eps \}$ of smooth solutions of the hyperbolic-parabolic system \eqref{hyppara} 
defined on $\bar Q_T = \To^d \times [0, \infty)$
which satisfy the entropy dissipation identities \eqref{hypparaen}. Let $\bu$ be a smooth solution of the hyperbolic system of conservation laws 
\eqref{hypcl},  defined on a maximal interval of existence $[0, T^*)$ with  $T^* \le \infty$, and satisfying
the entropy conservation identity~\eqref{addcl}.
Our objective is to show convergence of the family $\{ u_\eps \}$
 to the smooth solution $\bu$ on $\To^d \times [0,T]$, for $T < T^*$, as $\eps \to 0$.

First, we establish the analog of the relative entropy identity when comparing a solution $u$ of \eqref{hyppar} to a solution $\bu$ of \eqref{hypcl}. 
With $u$ and $\bu$ as noted, the analog of \eqref{eqn1} now reads
\begin{equation}
\label{conveqn1}
\begin{aligned}
&\del_t (\eta(u) - \eta(\bu)) + \del_\alpha (q_\alpha (u) - q_\alpha (\bu) ) 
\\
&=  
\eps  \del_\alpha \big ( G(u) \cdot B_{\alpha \beta} (u) \del_\beta u  \big ) 
-\eps \nabla G(u) u_{x_\alpha} \cdot B_{\alpha \beta} (u) u_{x_\beta}  \;,
\end{aligned}
\end{equation}
while the analog of \eqref{eqn2} is
\begin{equation}
\label{conveqn2}
\begin{aligned}
&\del_t \Big ( G(\bu) \cdot (A(u) - A(\bu) ) \Big ) + \del_\alpha \Big  ( G(\bu) \cdot ( F_\alpha (u) - F_\alpha (\bu) ) \Big )
\\
&=   \nabla G(\bu) \bu_{x_\alpha} \cdot  F_\alpha (u | \bu) 
+ \eps G(\bu) \cdot \del_\alpha \big (  B_{\alpha \beta} (u) \del_\beta u  \big ) \, , 
\end{aligned}
\end{equation}
hence leading to a relative entropy identity in the form
\begin{equation}
\label{conveqnrelen}
\begin{aligned}
\del_t  \eta (u | \bu)  &+ \del_{\alpha} \Big [ q_\alpha (u | \bu)  -  (G(u) - G(\bu) ) \cdot B_{\alpha \beta} (u) \, \del_{\beta} u \Big ]
+ \eps \nabla G (u)  \del_\alpha u \cdot   B_{\alpha \beta} (u) \, \del_{\beta} u 
\\
&= - ( \del_{\alpha}  G (\bu)  ) \cdot F_\alpha (u | \bu)
+ \eps ( \del_{\alpha}  G (\bu)  ) \cdot B_{\alpha \beta} (u) \del_{\beta} u \, .
\end{aligned}
\end{equation}

We proceed to prove a general convergence theorem for zero viscosity limits  to strong solutions.  To this end condition  \eqref{hyp1sp} on the viscosity matrices is replaced by the following  hypothesis: 
\begin{equation}
\label{DH}
\sum_{\alpha,\beta}\nabla G(u)\del_\alpha u\cdot B_{\alpha\beta}(u)\del_\beta u\ge \mu\sum_\alpha|B_{\alpha\beta}(u)\del_\beta u|^2\;,
\tag {H$_5$}
\end{equation}
for any solution $u$ to~\eqref{hyppar}  with some positive constant $\mu$, depending possibly on $u$. We first note that~\eqref{DH} can be also written in a general form using $\xi_\alpha\in\R^n$ instead of the gradients $\del_\alpha u$ in a similar fashion as in~\eqref{hyp1}. Here it is used in the following theorem as stated above.
We next observe that~\eqref{DH} is a sufficient condition, though not necessary, to render the third term on the left-hand side of~\eqref{conveqnrelen} dissipative and via~\eqref{DH}, this term dominates the last term of~\eqref{conveqnrelen}. This is the motivation given by Dafermos~\cite[Chapter IV]{daf10} for stating this condition that actually allows general viscosity matrices not necessarily non-degenerate for the convergence to the zero-viscosity limit as we establish in the following theorem. 
We also remark that when $B_{\alpha\beta}$ vanishes for $\alpha\ne \beta$ and is the identity otherwise, then condition~\eqref{DH} reduces to~\eqref{hyp1sp} in the general setting that $\xi_\alpha\doteq\del_\alpha u$. Last, we mention that hypothesis~\eqref{DH} is also related to the Kawashima condition when $\xi_\alpha=\nu_\alpha R_i$. See also the articles by Serre~\cite{S-PhyD, S-Cont, S-IMA} for further discussion.

Hypothesis~\eqref{DH} expanded in coordinates is rewritten as
$$
\sum_{i,j}\sum_{\alpha,\,\beta}\del_\alpha G^i(u) B_{\alpha\beta}^{ij}(u)\del_\beta u^j\ge\mu\sum_{\alpha,i}\Big|\sum_{\beta,\,j} B_{\alpha\beta}^{ij}(u)\del_\beta u^j\Big|^2\;.
$$

We now prove the convergence:

\begin{theorem}\label{thmconv}
Let $\bar u$ be a Lipshitz solution of \eqref{hypcl} defined on a maximal interval of existence $\To^d\times[0,T^*)$ and 
let $u^\eps$ be smooth solutions of~\eqref{hyppar} defined on $\To^d\times[0,T]$, $T < T^*$, and  emanating from smooth data $\bar u_0$, $u^\eps_0$, respectively. 
Assume the hypotheses~\eqref{hypns},~\eqref{hypep}, \eqref{hyppd} and \eqref{DH}  hold true and that the solution $\bar u$ takes values in a ball $B_M\subset\R^n$ of radius $M>0$. Then there exists a constant $C=C(T, M, \sup |\nabla\bar u|,\mu )$ independent of $\eps > 0$ such that
\begin{equation}
\label{Sconvergence}
\int_{\To^d} \eta (u^\eps | \bu ) \, dx  \le C  \Big (\int_{\To^d} \eta (u_0^\eps | \bu _0 ) \, dx  +  \eps  \int_0^T \int_{\To^d} | \nabla \bu|^2 \, dx ds \Big ) \,.
\end{equation}
In particular, if $ \int_{\To^d} \eta (u_0^\eps | \bu _0 ) \, dx \to 0$
and the solution $\bu$ satisfies the integrability $\nabla \bu \in L^2 ([0,T] \times \To^d)$ 
then 
\begin{equation}
\label{relenconv}
\sup_{ t \in (0, T)} \int_{\To^d} \eta (u^\eps (t)  | \bu (t)  ) \, dx  \to 0 \qquad \mbox{as $\eps \to 0$}.
\end{equation}
\end{theorem}

\begin{proof}
We define $\varphi^\eps (t)$ as before to be
$$
\varphi^\eps (t)  \doteq\int_{\To^d}\eta(u^\eps (x,t)  | \bar u (x,t) )\,dx\qquad t\in[0,T].
$$
and by integrating the relative entropy identity \eqref{conveqnrelen}, we arrive at 
\begin{equation}
\label{conveqnrelen-2}
\begin{aligned}
\frac{d \varphi^\eps}{dt} &+\eps\int_{\To^d}\big(\nabla G(u^\eps)^TB_{\alpha\,\beta}(u^\eps)\del_\beta u^\eps\big)\cdot\del_\alpha u^\eps\,dx\\
&\le  C  \int_{\To^d}   | F_\alpha (u^\eps | \bu ) | \, dx+\eps\int_{\To^d}|\del_\alpha G(\bu)\cdot\big( B_{\alpha\,\beta}(u^\eps)\del_\beta u^\eps\big)|\,dx,
\end{aligned}
\end{equation}
since $\bu$ takes values in $B_M$. Throughout this proof, $C= C(T,  M, |\nabla \bu|)$ stands for a generic constant depending on $|\nabla \bu|$ but independent of $\eps$. Now under Hypothesis~\eqref{DH}, we estimate
\begin{equation}
\label{thm2.9estdiss}
\begin{aligned}
\eps\int_{\To^d}|\del_\alpha G(\bu)&\cdot B_{\alpha\,\beta}(u^\eps)\del_\beta u^\eps|\,dx 
=\eps\int_{\To^d}\left|\sum_{i,j}\sum_{\alpha,\beta}\del_\alpha G^i(\bu)B_{\alpha\beta}^{ij}(u^\eps)\del_\beta u^{\eps,j}\right|dx\\
&
\le \eps
\int_{\To^d}\Big(\sum_{i,\alpha}|\del_\alpha G^i(\bu)\big|^2\Big)^{1/2} \Big(\sum_{i,\alpha}\big|\sum_{j,\beta} B^{i\,j}_{\alpha,\beta}(u^\eps)\del_\beta u^{\eps,j}\big|^2\Big)^{1/2}dx\\
&\le \frac{\eps}{2\mu} \int_{\To^d} \sum_{i,\alpha}|\del_\alpha G^i(\bu)\big|^2\,dx+\eps\frac{\mu}{2}  \int_{\To^d} \sum_{i,\alpha}\big|\sum_{j,\beta} B^{i\,j}_{\alpha,\beta}(u^\eps)\del_\beta u^{\eps,j}\big|^2 dx\\
&\le \frac{C\,\eps }{2\mu}\int_{\To^d} |\nabla\bu|^2dx+\frac{\eps}{2} \int_{\To^d} \sum_{i,j}\sum_{\alpha,\,\beta}\del_\alpha G^i(u^\eps) B_{\alpha\beta}^{ij}(u^\eps)\del_\beta u^{\eps,j} dx
\end{aligned}
\end{equation}
and hence the last term of~\eqref{conveqnrelen-2} is controlled by dissipation. By bound \eqref{lemform2} in Lemma \ref{lemuseful} on $F_\alpha(u^\eps|\bu)$ and~\eqref{thm2.9estdiss}, we deduce the estimation
$$
\begin{aligned}
\frac{d \varphi^\eps}{dt} +\frac{\eps}{2}\int_{\To^d}\big(\nabla G(u^\eps)^TB_{\alpha\,\beta}(u^\eps)\del_\beta u^\eps\big)\cdot\del_\alpha u^\eps\,dx
&\le  C  \int_{\To^d}   | F_\alpha (u^\eps | \bu ) | \, dx + \frac{\eps}{2\mu}C\int_{\To^d} |\nabla\bu|^2dx
\\
&\le C \int_{\To^d}  \eta(u^\eps  | \bu  )  \, dx  + \frac{\eps}{2\mu}C  \int_{\To^d}  | \nabla \bu |^2 dx .
\end{aligned}
$$
We now obtain the differential inequality
\begin{equation}
\frac{d \varphi^\eps}{dt} \le C \varphi^\eps (t)+ \frac{C\,\eps }{2\mu} \int_{\To^d}  | \nabla \bu |^2 dx \;,
\end{equation}
and then Gronwall's inequality gives \eqref{Sconvergence} and the proof is complete. Note that $C$ in \eqref{Sconvergence} depends also on the positive constant $\mu$ that is present in hypothesis~\eqref{DH}.
\end{proof}

\begin{remark} \rm
The convergence obtained in Theorem \ref{thmconv} uses as ``metric" for measuring distance the relative entropy function.
While this is not a metric, it operates as a combination of norms when $\bu \in B_M$ takes values in a compact set, in view of Lemma \ref{lemvariant}.
Note that \eqref{Sconvergence}  even provides an $O(\eps)$ rate of convergence  when the limit $\bu$ is a smooth solution.  
We note that the measure-valued weak versus strong uniqueness Theorem \ref{thmweakstrong}
offers as a corollary such a convergence result but without a convergence rate. Finally, we note that there are alternative well developed techniques
for obtaining convergence results for the zero-viscosity limit to smooth solutions and in stronger norms - see for instance \cite[Ch V]{Kawashima84} - but the above proof is striking in its simplicity and generality.
\end{remark}

\section{Uniqueness of smooth solutions in the class of dissipative  measure-valued solutions}\label{S2.2.2}
In the sequel, we consider  measure valued solutions first to the system of conservation laws~\eqref{hypcl} 
in the presence of $L^p$ bounds for $1<p<\infty$ and then in the presence of source terms. The goal is to prove uniqueness of smooth solutions with the class of dissipative measure-valued solutions.
For this purpose, we impose a set of growth restrictions on the constitutive functions of the problem:
It is assumed that the entropy $\eta(u)$ has the growth behavior
\begin{equation}\label{hypetap}
\beta_1 (|u|^p+1)-B\le \eta(u)\le \beta_2(|u|^p+1) \qquad \text{for } u\in\mathbb{R}^n
\tag{A$_1$}
\end{equation}
for some positive constants $\beta_1$, $\beta_2$, $B$  and for some $p \in (1,\infty)$. Moreover, that the
functions $F_\alpha$ and $A$ in~\eqref{hypcl} satisfy the growth restrictions
\begin{equation}\label{hypfgrowth}
\frac{|F_\alpha(u)|}{\eta(u)}=o(1)\qquad \text{as } |u|\to\infty,\quad\alpha=1,\dots,d \, ,
\tag{A$_2$}
\end{equation}
\begin{equation}\label{hypAgrowth}
\frac{|A(u)|}{\eta(u)}=o(1)\qquad \text{as } |u|\to\infty \, .
\tag{A$_3$}
\end{equation}

In the appendix, we adapt an idea from \cite{dst12} which leads to useful bounds for the relative entropy and the relative stress when
$\bu$ is restricted to take values in $B_M$. These bounds are used to establish the weak--strong uniqueness theorems in the following two subsections, for systems of conservation laws and balance laws respectively.

To avoid technicalities, we work for the spatially periodic case with domain $\To^d =(\R/2\pi\Z)^d$. 
Set $Q_T= \To^d \times[0,T)$ for $T\in(0,\infty)$ and $\overline{Q_T}= \To^d \times[0,T]$.
\subsection{Systems of hyperbolic conservation laws}
The study of existence of measure valued solutions typically involves the construction of a sequence of solutions  $u^\eps$ of an approximating problem
\begin{equation}
\label{hypcleps}
\begin{aligned}
\del_t A(u^\eps) + \del_\alpha F_\alpha (u^\eps) = \mathcal{P}_\eps,
\end{aligned}
\end{equation}
with $\mathcal{P}_\eps\to0$ in distributions as $\eps\to0+$ satisfying an entropy inequality
\begin{equation}
\label{entropyhypeps}
\begin{aligned}
\del_t \eta(u^\eps) + \del_\alpha q_\alpha (u^\eps) \le \mathcal{Q}_\eps,
\end{aligned}
\end{equation}
again with $\mathcal{Q}_\eps\to0$ in distributions.
A typical example of approximation is provided by the hyperbolic-parabolic system \eqref{hyppar} studied in section \ref{secrelenhyppar}; there
the entropy inequality \eqref{entropyhypeps} results from the identity \eqref{hypparaen} via hypothesis \eqref{hyp1}. It yields uniform bounds for
approximate solutions of the form
\begin{equation}
\label{bound}
\sup_{t \in [0,T]} \int \eta (u^\eps) dx \le M +  \int \eta (u_0^\eps) dx 
\end{equation}
for some $M$ depending possibly on $T$ but independent of $\eps$.

We work under a framework  of growth hypotheses \eqref{hypetap}, \eqref{hypfgrowth} and \eqref{hypAgrowth}, and then \eqref{bound} implies
a framework of uniform $L^p$ bounds. Recall that $\eta (u) = H (A(u))$. We assume that $H(v)$ is convex and positive. We also postulate the 
growth hypothesis 
\begin{equation}\label{S2.2 A4}
\frac{1}{C}( |A(u)|^q + 1) - B \le H(A(u))\le C( |A(u)|^q + 1) ,\qquad q>1,
\tag{A$_4$}
\end{equation}
for some uniform constant $C>0$ and $B > 0$, which amounts  to control on the growth of $v = A(u)$.

Consider the sequence of approximate solutions $\{u^\eps \}$ and $\{ v^\eps \}$, where $v^\eps = A(u^\eps)$, and introduce the associated Young measures
$\boldsymbol{\nu}_{(x,t)}$ and $\boldsymbol{N}_{(x,t)}$, respectively. More precisely, 
$\{u^\eps\}$ is assumed to be a sequence of Lebesgue measurable functions with a convergent subsequence (again called $u^\eps$) 
associated Young measure  $\boldsymbol{\nu}_{(x,t)}$, which is a weak$^*$ measurable family of Radon probability measures. 
At the same time, consider the sequence $\{v^\eps\}$, $v^\eps = A(u^\eps)$,  with the associated Young measure denoted by
$\boldsymbol{N}_{(x,t)}$. We next quote from~\cite[Appendix A]{dst12} what is needed to know about the Young measure description of oscillations and concentrations in weakly convergent sequences of functions $v^\eps$ defined on $\overline{Q_T}$ in the $L^q$ framework, $1<q<\infty$ in which  the development of concentrations in $H(v^\eps)$ is permitted. Using the Young measure $\boldsymbol{N}=({\boldsymbol{N}}_{(x,t)})_{(x,t)\in\overline{Q_T}}$ associated to the family $\{v^\eps\}$ we have
\begin{equation}
g(v^\eps)\rightharpoonup\langle{\boldsymbol{N}}_{x,t},g(\rho)\rangle
\end{equation}
for all continuous functions $g$ such that $\displaystyle\lim_{|\rho|\to\infty}\dfrac{g(\rho)}{1+|\rho|^q}=0$. 
Similarly, 
\begin{equation}
f(u^\eps)\rightharpoonup\langle{\boldsymbol{\nu}}_{x,t},f(\lambda)\rangle
\end{equation}
for all continuous functions $f$ such that $\displaystyle\lim_{|\lambda|\to\infty}\dfrac{f(\lambda)}{1+|\lambda|^p}=0$.
Since $v^\eps=A(u^\eps)$, the two measures are connected via
\begin{equation}\label{nu-N}
\langle{\boldsymbol{\nu}}_{x,t},g(A(\lambda))\rangle=\langle{\boldsymbol{\nu}}_{x,t},f(\lambda)\rangle=\langle{\boldsymbol{N}}_{x,t},g(\rho)\rangle
\end{equation}
whenever $f=g\circ A$, in other words ${\boldsymbol{N}}\circ A={\boldsymbol{\nu}}$.

Next, we employ the analysis in~\cite[Appendix A]{dst12}, which indicates that,  if the function $H(v)$ is convex and positive,  
the oscillations and concentrations can be represented via
\begin{equation}
H(v^\eps) dxdt\rightharpoonup\langle{\boldsymbol{N}}_{x,t},H\rangle dxdt+\boldsymbol{\gamma}(dxdt)
\end{equation}
as $\eps\to0+$. It is worth recalling that $\langle{\boldsymbol{N}}_{x,t},H\rangle$, which cannot be defined via the Young measure theorem due to its growth, is
instead defined via the limiting process 
\begin{equation}
\langle{\boldsymbol{N}}_{x,t},H(\rho)\rangle\doteq\lim_{R\to\infty}\langle{\boldsymbol{N}}_{x,t},H(\rho)\cdot 1_{H(\rho)<R}+R\cdot1_{H(\rho)\ge R}\rangle
\end{equation}
and the monotone convergence theorem,  and that $\boldsymbol{\gamma}$ is the \emph{concentration measure}, a non-negative Radon measure defined by
\begin{equation}\label{gamma1}
\boldsymbol{\gamma}\doteq\text{wk}^*-\lim_{\eps\to0+}(H(v^\eps)- \langle{\boldsymbol{N}}_{x,t},H\rangle  )\in\mathcal{M}^+(\overline{Q_T}).
\end{equation}
In addition,  for the initial data $\{v_0^\eps\}$ of the approximating problem, we assume weak convergence in $L^q$ with associated Young measure $\boldsymbol{N}_0$ and again possible development of concentrations described by the concentration measure $\boldsymbol{\gamma}_0(dx)\ge 0$, i.e.
\begin{equation}
g(v_0^\eps)\rightharpoonup\langle{\boldsymbol{N}_0}_{x},g(\rho)\rangle\qquad\forall g\text{ continuous s.t. }\lim_{|\rho|\to\infty}\dfrac{g(\rho)}{H(\rho)}=0
\end{equation}
and
\begin{equation}
H(v_0^\eps) dx\rightharpoonup\langle{\boldsymbol{N}_0}_{x}, H\rangle dx+\boldsymbol{\gamma}_0(dx).
\end{equation}

From Section~\ref{sechypconv}, we recall that the convexity assumption~\eqref{hypvucprime} for $H(v)$ translates to Hypothesis~\eqref{hyppd} for $\eta(u)$ via the relation $\eta=H\circ A$. Hence, by~\eqref{nu-N}-\eqref{gamma1}, we have
\begin{equation}
\eta(u^\eps) dxdt\rightharpoonup\langle{\boldsymbol{\nu}}_{x,t},\eta\rangle dxdt+\boldsymbol{\gamma}(dxdt)
\end{equation}
as $\eps\to0+$, 
and
\begin{equation}
\langle{\boldsymbol{\nu}}_{x,t},\eta(\lambda)\rangle\doteq \lim_{R\to\infty}\langle{\boldsymbol{N}}_{x,t},H(A(\lambda))\cdot 1_{H(A(\lambda))<R}+R\cdot 1_{H(A(\lambda))\ge R}\rangle
\end{equation}
where $\boldsymbol{\gamma}$ is the same \emph{concentration measure} given in~\eqref{gamma1} that can also be expressed as
\begin{equation}
\boldsymbol{\gamma}=\text{wk}^*-\lim_{\eps\to0+}(\eta(u^\eps)- \langle{\boldsymbol{\nu}}_{x,t},\eta\rangle  )\in\mathcal{M}^+(\overline{Q_T}).
\end{equation}

Now we state the definition of \emph{dissipative measure-valued solutions}, which form a sub-class of the measure valued solutions and  satisfy an averaged and integrated form of the entropy inequality that allows for concentration effects in the $L^p$ framework $p<\infty$.

\begin{definition}\label{defdissi}
A dissipative measure valued solution $(u,\boldsymbol{\nu},\boldsymbol{\gamma})$ with concentration to~\eqref{hypcl} consists of $u\in L^\infty(L^p)$, a Young measure $\boldsymbol{\nu}=({\boldsymbol{\nu}}_{x,t})_{\{(x,t)\in \bar{Q}_T\}}$ and a non-negative Radon measure $\boldsymbol{\gamma}\in\mathcal{M}^+(Q_T)$ such that
$u (x,t)  = \langle {\boldsymbol{\nu}}_{(x,t)} , \lambda \rangle$ and
\begin{equation}\label{dfmv1}
\iint\langle {\boldsymbol{\nu}}_{x,t}, A_i(\lambda)\rangle \del_t\varphi_{i}\,dx\,dt+\iint \langle {\boldsymbol{\nu}}_{x,t}, F_{i,\alpha}(\lambda)\rangle \del_\alpha\varphi_{i}dx\,dt+\int \langle\boldsymbol{\nu}_0,A_i\rangle\varphi_i(x,0)dx=0\qquad i=1,\dots,n,
\end{equation}
for any $\varphi\in C^1_c(Q\times[0,T))$ and
\begin{equation}\label{dfmv2}
\iint\frac{d\xi}{dt} \left[\langle {\boldsymbol{\nu}}_{x,t}, \eta(\lambda)\rangle dxdt+\boldsymbol{\gamma}(dxdt)    \right]+\int \xi(0)\left[\langle {\boldsymbol{\nu}_0}_{x}, \eta\rangle dx+\boldsymbol{\gamma}_0(dx)  \right]\ge 0,
\end{equation}
 for all $\xi=\xi(t)\in C^1_c([0,T))$ with $\xi\ge 0$.
\end{definition}

The following theorem establishes the recovery of \emph{classical} solutions from \emph{dissipative measure--valued} solutions. In other words, this theorem 
states a weak measure-valued versus strong uniqueness theorem in the $L^p$ framework for $1<p<\infty$:

\begin{theorem}
\label{thmweakstrong}
Suppose that \eqref{hypns}--\eqref{hyppd} hold,  the growth properties~\eqref{hypetap}--\eqref{S2.2 A4} are satisfied,  and the entropy  $\eta(u) \ge 0$. 
Assume that $(u,\boldsymbol{\nu}, \boldsymbol{\gamma})$ is a dissipative measure--valued solution,  $u = \langle \boldsymbol{\nu}_{x,t} , \lambda \rangle$,  and  $\bar{u}\in W^{1,\infty}(\overline{Q_T})$ 
is a strong solution to~\eqref{hypcl}. Then, if the  initial data 
satisfy $\boldsymbol{\gamma}_0=0$ and ${\boldsymbol{\nu}_0}_x=\delta_{\bu_0}(x)$, it holds $\boldsymbol{\nu}=\delta_{\bar{u}}$ and $u=\bar{u}$ almost everywhere on $Q_T$.
\end{theorem}

\begin{proof}
Let $K\subset\mathbb{R}^n$ be a compact set containing the values of the strong solution $\bar{u}(x,t)$ for $(x,t)\in Q_T$.
First we define the averaged quantities
\begin{equation}\label{dfH}
\mathcal{H}(\boldsymbol{\nu},u,\bar{u})\doteq\langle \boldsymbol{\nu},\eta\rangle-\eta(\bar{u})-G(\bar{u})\cdot \big ( \langle\boldsymbol{\nu},A\rangle-A(\bar{u}) \big )
\end{equation}
\begin{equation}
\label{dfZ}
Z_\alpha(\boldsymbol{\nu},u,\bar{u})\doteq\langle \boldsymbol{\nu},F_\alpha\rangle-F_\alpha(\bar{u})-\nabla F_\alpha(\bar{u})\nabla A(\bar{u})^{-1}(\langle\boldsymbol{\nu},A\rangle-A(\bar{u}))\;.
\end{equation}
It is easy to check that
\begin{equation}
\mathcal{H}(\boldsymbol{\nu},u,\bar{u})=\int \eta(\lambda|\bar{u})d\boldsymbol{\nu}(\lambda)
\end{equation}
using~\eqref{defrelfl}. As in~\eqref{diffhyp}
\begin{equation}\label{mvdiffhyp}
\del_t \Big ( G(\bu) \Big)\cdot (\langle\boldsymbol{\nu},A\rangle-A(\bu))+\del_{\alpha} \Big(G(\bu)\Big)\cdot(\langle\boldsymbol{\nu},F_\alpha\rangle-F_\alpha(\bu))=\nabla G(\bu)\bu_{x_\alpha}\cdot Z_\alpha(\boldsymbol{\nu},u,\bu)
\end{equation}
for $\alpha=1,\dots,d$. Since $\bar{u}\in W^{1,\infty}(\overline{Q_T})$ is a strong solution of~\eqref{hypcl} then it verifies the strong versions of~\eqref{dfmv1}--\eqref{dfmv2}, i.e.
\begin{equation}\label{dfcl1}
\iint A_i(\bar{u}) \varphi_{i,t}+F_{i,\alpha}(\bar{u})\varphi_{i,\alpha}dx\,dt+\int A_i(\overline{u}_0(x))\varphi_i(x,0)dx=0\qquad i=1,\dots, n
\end{equation}
and
\begin{equation}\label{dfcl2}
\iint\frac{d\xi}{dt}\eta(\bar{u})dx\,dt+\int  \xi (0)\eta(\overline{u}_0(x))dx=0
\end{equation}
for all test functions $\varphi\in C^1_c(Q\times[0,T))$, $\xi\in C^1_c([0,T))$ with $\xi \ge 0$.

Now choosing $\varphi(x,\tau)\doteq\xi(\tau) G(\bu(x,\tau))$ in~\eqref{dfmv1} and~\eqref{dfcl1}, subtracting~\eqref{dfcl1} from~\eqref{dfmv1} and combining with~\eqref{mvdiffhyp} we arrive at
\begin{equation}\label{intmvdiffhyp}
\begin{aligned}
\iint\frac{d\xi}{d\tau} G(\bar{u}) \cdot(\langle\boldsymbol{\nu}, A\rangle-A(\bar{u}))+\xi(\tau) \nabla G(\bar{u})\bar{u}_{x_\alpha} \cdot Z_\alpha(\boldsymbol{\nu},u,\bar{u})dxd\tau
\\
+\int \xi(0)G(\bar{u}_0) (\langle\boldsymbol{\nu}_0,A\rangle-A(\bar{u}_0))dx=0\;.
\end{aligned}
\end{equation}
Subtracting equations~\eqref{intmvdiffhyp} and~\eqref{dfcl2} from~\eqref{dfmv2} and using~\eqref{dfH},  we get
\begin{equation}\label{Hintineq}
\begin{aligned}
\iint\frac{d\xi}{d\tau} \, &\mathcal{H}(\boldsymbol{\nu},u,\bar{u}) dxd\tau  + \iint \frac{d\xi}{d\tau} \gamma (dx d\tau) 
\\
&\ge \iint \xi(\tau)\nabla G(\bar{u}) \bar{u}_{x_\alpha} Z_\alpha dxd\tau
\\
&-\int\xi(0)\left[(\langle\boldsymbol{\nu}_0,\eta\rangle-\eta(\bar{u}_0)-G(\bar{u}_0)\cdot (\langle\boldsymbol{\nu}_0,A\rangle-A(\bar{u}_0)))dx+\boldsymbol{\gamma}_0(dx)\right] \, ,
\end{aligned}
\end{equation}
for any $\xi \in C_c^1([0,T))$ with $\xi \ge 0$. We apply \eqref{Hintineq} to a sequence of smooth, monotone nonincreasing functions $\xi_n \ge 0$ that
approximate the Lipschitz function
\begin{equation}\label{S2.2.2xi}
\xi(\tau)\doteq\left\{
\begin{array}{lll}
1 & \text{if} & 0\le \tau<t\\
\frac{t-\tau}{\eps}+1 & \text{if} & t\le \tau<t+\eps\\
0 & \text{if} & \tau\ge t+\eps
\end{array}
\right. \; .
\end{equation}
Passing to the limit $n \to \infty$ and using that $\gamma \ge 0$, this leads to 
\begin{equation}
\begin{aligned}
-\frac{1}{\eps}\int_t^{t+\eps}\int \mathcal{H}(\boldsymbol{\nu},u,\bar{u}) dx\,d\tau &\ge \int_0^{t+\eps}\int \xi(\tau)\nabla G(\bar{u}) \bar{u}_{x_\alpha} \cdot Z_\alpha dxd\tau
\\
&-\int[\eta(u_0)-\eta(\bar{u}_0) -G(\bar{u}_0)\cdot (\langle\boldsymbol{\nu}_0,A\rangle-A(\bar{u}_0))]dx+\boldsymbol{\gamma}_0(dx)\;.
\end{aligned}
\end{equation}
Taking the limit as $\eps\to0+$, we arrive at
\begin{equation}\label{Hintineq2}
\begin{aligned}
\int \mathcal{H}(\boldsymbol{\nu},u,\bar{u})\,dx\le& C\int_0^t \int\max_{\alpha}|Z_\alpha |dx\,d\tau\\
&+\int [\langle\boldsymbol{\nu}_0,\eta\rangle-\eta(\bar{u}_0)-G(\bar{u}_0)\cdot (\langle\boldsymbol{\nu}_0,A\rangle-A(\bar{u}_0))]dx+\boldsymbol{\gamma}_0(dx)
\end{aligned}
\end{equation}
for $t\in(0,T)$. Note that $C=C(K,|\nabla\bar u| )$.

The rest of the proof is based on the estimate of Lemma \ref{lemuseful}.
Suppose that $K$ is contained in the ball $B_M$ centered at the origin and of radius $M$.
The  properties~\eqref{hypetap}--\eqref{hypAgrowth} allow to employ bound \eqref{lemform2}
and to  estimate \eqref{dfZ} in terms of \eqref{dfH} as follows:
\begin{equation}
\begin{aligned}
Z_\alpha(\boldsymbol{\nu},u,\bar{u})  &= 
 \langle \boldsymbol{\nu},F_\alpha\rangle-F_\alpha(\bar{u})-\nabla F_\alpha(\bar{u})\nabla A(\bar{u})^{-1}(\langle\boldsymbol{\nu},A\rangle-A(\bar{u}))
 \\
 &= \langle \boldsymbol{\nu} , F_\alpha (\lambda | \bu ) \rangle
 \\
 &\le  C_1 \langle \boldsymbol{\nu} , \eta (\lambda | \bu ) \rangle 
 \\
 &= C_1 \mathcal{H} ( \boldsymbol{\nu},u,\bar{u})
\end{aligned}
\end{equation}
Then \eqref{Hintineq2} takes the form
\begin{equation}
\int \mathcal{H} (\boldsymbol{\nu},u,\bar{u})dx\le C_1'\int_0^t\int \mathcal{H} (\boldsymbol{\nu},u,\bar{u})\,dxd\tau+\iint\eta(\lambda|\bu_0) d\boldsymbol{\nu}_0(\lambda)\,dx+\boldsymbol{\gamma}_0(dx)\;.
\end{equation}

For the initial data it is assumed that there are no concentrations, that is  $\boldsymbol{\gamma}_0=0$ . Applying Gronwall's inequality, we conclude that
\begin{equation}
\int \mathcal{H} (\boldsymbol{\nu},u,\bar{u})dx\le C_1''\iint \eta(\lambda|\bu_0)d{\boldsymbol{\nu}}_0(\lambda)dx\, e^{C_1' t}
\end{equation}
for some positive constants $C_1'$ and $C_1''$. The proof of the theorem follows.
\end{proof}

It immediately follows:

\begin{corollary}
Under the hypotheses of Theorem \ref{thmweakstrong}, let $u \in C([0,T] ; L^p (\To^d) )$, $p > 1$,  be an entropy weak solution of \eqref{hypcl}
satisfying \eqref{entropyhyp} and let $\bar{u}\in W^{1,\infty}(\overline{Q_T})$ be a strong solution to~\eqref{hypcl}. Then, if the initial data
$u_0 = \bu_0$ almost everywhere on $\To^d$,  then $u=\bar{u}$ almost everywhere on $Q_T$.
\end{corollary}

A uniqueness theorem for strong solutions in the class of dissipative measure-valued solutions can also easily be established  in a framework of
 $L^\infty$ uniform bounds. 
In such a setting, no concentration effects are present, hence $\boldsymbol{\gamma}=0$ and $\boldsymbol{\gamma}_0=0$ in Definition~\ref{defdissi}. 
We state such a result  for the sake of completeness but omit the proof. We refer the reader to~\cite[Theorem 2.2]{dst12} and \cite{bds11} for details in the case $A(u)=u$.

\begin{theorem}
Let $\bu\in W^{1,\infty}(Q_T)$ be a strong solution and let $(u,\boldsymbol{\nu})$ be  a dissipative measure valued solution to~\eqref{hypcl} respectively. Assume that there exists a compact set $K \subset\R^n$ such that $\bu,u\in K$ for $(x,t)\in Q_T$ and that $\nu$ is also supported in $K$. Then there exist constants $c_1>0$ and $c_2>0$ such that
\begin{equation}
\iint|\lambda-\bu(t)|^2d\boldsymbol{\nu}(\lambda)dx\le c_1\left(\int|u_0-\bu_0|^2dx\right) e^{c_2 t}.
\end{equation}
Moreover, if the initial data agree $u_0=\bu_0$, then $\boldsymbol{\nu}=\delta_{\bu}$ and the dissipative measure valued solution is a strong solution, i.e. $u=\bu$ almost everywhere.
\end{theorem}

\subsection{Systems of hyperbolic balance laws}
\label{secrelenbal}
Consider next the system of balance laws:
\begin{equation}
\label{hypbl1}
\del_t A(u) + \del_\alpha F_\alpha (u) = P(u) \,,
\end{equation}
and the entropy condition
\begin{equation}
\label{hypbl2}
\del_t \eta(u) + \del_\alpha q_\alpha (u) \le G(u) \cdot P(u)\,,
\end{equation}
where $P:\R^n\to\R^n$ is a smooth function called production. Let $u$ be a weak solution of \eqref{hypbl1} that satisfies the weak form of
the entropy inequality \eqref{hypbl2}. Here,  $\eta - q_\alpha$ are an entropy-entropy flux
pair satisfying the hypotheses in Section \ref{sechyphyp}  and  $G(u)$ is the multiplier in \eqref{hypep}.
We wish to compare the entropy weak solution $u$ to a strong conservative solution $\bu$ of
\begin{equation}
\label{hypblstrong}
\begin{aligned}
\del_t A(\bu) + \del_\alpha F_\alpha (\bu) &= P(\bu) 
\\
\del_t \eta(\bu) + \del_\alpha q_\alpha (\bu) &= G(\bu) \cdot P(\bu)\;,
\end{aligned}
\end{equation}
using again the relative entropy function \eqref{defrelen} and a computation in the spirit of Section \ref{secrelen}.

This is accomplished as follows.
The formula for the entropy dissipation of the difference of the two solutions 
$$
\del_t (\eta(u) - \eta(\bu)) + \del_\alpha ( q_\alpha (u) - q_\alpha (\bu) ) \le G(u) \cdot P(u) - G(\bu) \cdot P(\bu)
$$
is combined with the analog of \eqref{diffhyp},  which  in the present case takes the form
$$
\begin{aligned}
&\del_t \Big ( G(\bu) \cdot (A(u) - A(\bu) ) \Big ) + \del_\alpha \Big  ( G(\bu) \cdot ( F_\alpha (u) - F_\alpha (\bu) ) \Big )
\\
&\qquad=   \nabla G(\bu) \del_{x_\alpha} \bu  \cdot  F_\alpha (u | \bu) + P(\bu) \cdot \nabla G(\bu) \nabla A (\bu)^{-1}  \big(A(u) - A(\bu) \big)
+ G(\bu) \cdot (P(u) - P(\bu))
\end{aligned}
$$
Combining the two formulas leads to the relative entropy identity
\begin{equation}
\label{relenbal}
\del_t  \eta (u | \bu)  + \del_{\alpha} q_\alpha (u | \bu) \le  -   \del_\alpha G (\bu)  \cdot F_\alpha (u | \bu) + P(\bu) \cdot G(u | \bu) 
+ ( G(u) - G(\bu) ) \cdot (P(u) - P(\bu) )\;,
\end{equation}
where $F_\alpha (u | \bu)$ is defined in \eqref{relflux} while $G(u| \bu)$ in~\eqref{relmult}.

The derivation of formula \eqref{relenbal} as presented here is formal, but it may be made rigorous by standard arguments and even 
be performed between a dissipative measure-valued solution $u$ and a strong solution $\bu$ using the process outlined in Theorem \ref{thmweakstrong}.
One may easily extend Theorem~\ref{thmweakstrong} to hold for the case of dissipative measure-valued solutions of a balance law  \eqref{hypbl1}--\eqref{hypbl2} if
it is assumed that the vector field $P(u)$ is weakly dissipative, 
meaning that $P(u)$ satisfies the  hypothesis
\begin{equation}
\label{hypdissvf}
( G(u) - G(\bu) ) \cdot (P(u) - P(\bu) ) \le 0  \qquad \forall \, u, \bu \in \R^n 
\tag{H$_d$}
\end{equation}
and a growth condition 
\begin{equation}
\label{Ggrowth}
\frac{|G(u)|}{\eta(u)}=o(1)\qquad \text{as } |u|\to\infty\;.
\tag{A$_G$}
\end{equation}
Hypothesis~\eqref{hypdissvf} has been proposed in a context of relaxation balance laws in \cite{MT14}
and can be easily verified  for the  example  of friction in gas dynamics studied in \cite{lt13}. 
Hypothesis~\eqref{Ggrowth} will lead via an argument as in \eqref{lemform2} to the bound $G(u|\bu) \le C \eta (u | \bu)$.

Hypothesis~\eqref{hypdissvf} of a weakly dissipative field is not critical  for a weak-strong uniqueness theorem
and might be replaced by hypotheses allowing moderate growth for $P(u)$.  An inspection of the proof of Theorem \ref{thmweakstrong} 
indicates that it suffices to require bounds guaranteeing that
\begin{equation}\label{hypdissvf2}
 |( G(u) - G(\bu) ) \cdot (P(u) - P(\bu) )  | \le C \,  \eta (u | \bu)   \, .
\end{equation}
Estimate~\eqref{hypdissvf2} can be derived from~\eqref{Ggrowth} together with the additional growth conditions:
\begin{equation}\label{PGgrowth}
\frac{|G(u) \cdot P(u) |}{\eta (u)}=o(1)\quad\text{and}\quad \frac{|P(u) |}{\eta (u)}=o(1)\qquad\text{as } |u|\to\infty \, ,
\tag{A$_P$}
\end{equation}
following similar analysis as in the proof of~\eqref{lemform2} in Lemma~\ref{lemuseful}.
In view of the above analysis, a proposition on uniqueness of dissipative measure-valued versus strong solutions to balance laws~\eqref{hypbl1}--\eqref{hypbl2} as in Theorem~\ref{thmweakstrong} can be stated:
\begin{theorem}
\label{thmweakstrongbal}
Suppose that \eqref{hypns}--\eqref{hyppd} hold,  the growth properties~\eqref{hypetap}--\eqref{S2.2 A4} are satisfied,  and the entropy  $\eta(u) \ge 0$. Moreover, assume that
either (i) hypothesis~\eqref{hypdissvf} and~ \eqref{Ggrowth} hold true or, alternatively,  (ii) that  both growth conditions~\eqref{Ggrowth} and~\eqref{PGgrowth} are satisfied.
Let $(u,\boldsymbol{\nu}, \boldsymbol{\gamma})$ be a dissipative measure--valued solution,  with $u = \langle \boldsymbol{\nu}_{x,t} , \lambda \rangle$,  and  let 
$\bar{u}\in W^{1,\infty}(\overline{Q_T})$ 
be a strong solution to~\eqref{hypbl1}--\eqref{hypbl2}, respectively. Then, if the  initial data 
satisfy $\boldsymbol{\gamma}_0=0$ and ${\boldsymbol{\nu}_0}_x=\delta_{\bu_0}(x)$, it holds $\boldsymbol{\nu}=\delta_{\bar{u}}$ and $u=\bar{u}$ almost everywhere on $Q_T$.
\end{theorem}

\section{One-dimensional gas dynamics for viscous and heat-conducting gases}
\label{sec-vhcg}

Systems from thermomechanics typically have degenerate viscosity matrices.  It is however conceivable that  the dissipation 
still dominates the errors.  We next outline the relative entropy calculation for the system of thermoviscoelasticity.
For pedagogical reasons, we develop in this section the calculation in one-space
dimension. In the next section, we consider the three-dimensional case and focus on the treatment of the parabolic terms.
Our objective is to depict how the general hypotheses of Section \ref{secrelen} specialize and get  adapted to treat this
paradigm.

We consider the one-dimensional hyperbolic-parabolic system 
\begin{equation}
\label{vhcg}
\begin{aligned}
u_t - v_x &= 0
\\
v_t - \sigma(u,\theta)_x &= ( \mu v_x )_x + f
\\
\big ( \tfrac{1}{2} v^2 + e(u, \theta) \big )_t - ( \sigma(u,\theta) \, v )_x 
&= ( \mu \,  v \, v_x )_x +  (\kappa \theta_x )_x + f \, v + r  \, ,
\end{aligned}
\end{equation}
describing the equations of gas dynamics  in Lagrangean coordinates  for a viscous, heat-conducting gas. 
The gas obeys a Stokes constitutive law for the viscosity and a Fourier law for the heat conduction. 
In this model $u$ stands for the specific volume (the inverse of the density), $v$ for the longitudinal velocity,
and $\theta$ for the temperature,  while the internal energy $e$ and  the stress $\sigma$ are determined via constitutive relations; 
in this interpretation $u > 0$ and $\theta > 0$. Another interpretation of \eqref{vhcg} is as describing one dimensional shear motions
of a thermoviscoelastic material; in this case $u$ is the shear strain, $v$ the velocity in the shear direction , $\theta > 0$ the temperature
and the rest of the variables similar to before. In this interpretation $u$ does not obey any positivity constraint.

We impose the usual relations on the constitutive theory of thermoviscoelasticity implied by the requirement of 
compatibility of the constitutive theory with the  Clausius-Duhem inequality of thermodynamics, \cite{cn63,cm64}.
The  constitutive theory is determined by a free energy function $\psi = \psi (u, \theta)$ via the formulas
\begin{equation}
\label{const}
\sigma = \frac{ \del \psi}{\del u} \, , \quad \eta = - \frac{ \del \psi}{\del \theta}  \, , \quad e = \psi + \theta \, \eta
\end{equation}
and thus satisfies the Maxwell relations
\begin{equation}
\label{maxwell}
\sigma_\theta =  - \eta_u  \, , \quad  e_\theta = \theta \eta_\theta \, , \quad e_u = \sigma - \theta \sigma_\theta \, .
\end{equation}
We  assume the viscous part of the stress is that of a Stokes fluid
$$
\tau_{viscous} = \mu (u,\theta)  v_x \quad \mbox{ with $\mu (u,\theta) \ge 0$},
$$
while heat conduction is given by a Fourier law
$$
Q = \kappa(u, \theta)  \theta_x \quad \mbox{ with $\kappa (u,\theta) \ge 0$}.
$$

Under the above conditions the theory of one-dimensional viscoelasticity satisfies the entropy increase
identity
\begin{equation}
\label{entropyp}
\del_t \eta (u,\theta) - \del_x \frac{Q}{\theta} = \mu \frac{v_x^2}{\theta} + \kappa  \left ( \frac{\theta_x}{\theta} \right )^2 + \frac{r}{\theta}
\end{equation}
which implies the local form of the Clausius-Duhem inequality
$$
\del_t \eta  - \del_x \frac{Q}{\theta}  \ge \frac{r}{\theta} \, .
$$
The latter expresses a (beyond equilibrium) version of the 2nd law of thermodynamics.

\subsection{Properties of the relative entropy}
\label{sec-relen-defvhcg}

Next, we develop  a relative entropy calculation for the system~\eqref{vhcg},
guided by the theory developed in Section \ref{sec-relen} for hyperbolic-parabolic systems \eqref{hyppara}
in one-space dimension,
\begin{equation}
\label{hyppar1d}
\del_t A(U) + \del_x F(U) = \eps \del_x \big ( B(U) \del_x U  \big ) \, .
\end{equation}
Recall that $A(U)$ is assumed to be globally invertible (see  \eqref{hypns}) and that \eqref{hyppar1d} is equipped 
with an entropy - entropy flux pair $\heta(U) - \hq(U)$ generated by the multiplier $G(U)$:
\begin{equation}
\label{1dh2}
G(U) \cdot \nabla A(U) = \nabla \heta (U) \, , \quad G(U) \cdot \nabla F(U) = \nabla \hq (U) \, .
\end{equation}
The entropy is assumed to satisfy hypothesis \eqref{hyppd}, that
\begin{equation}
\label{1dh3}
\nabla^2 \heta (U) - G(U) \cdot \nabla^2 A(U) : = \nabla^2 \heta (U) - \sum_{j=1}^n G^j (U) \nabla^2 A^j (U)  > 0
\end{equation}
is positive definite. For the viscosity matrix $B(U)$, hypothesis  \eqref{hyp1} becomes 
\begin{equation}
\label{1dh4}
\xi \cdot  \nabla G(U)^T B (U) \xi  \ge 0  \qquad \forall \xi \in \R^n \, . 
\end{equation}
that is $\nabla G(U)^T B (U)$ is positive semidefinite. Then, system \eqref{hyppar1d} is endowed with the dissipative structure
\begin{equation}
\label{dissstr1d}
\del_t \heta (U) + \del_x \hq (U) = \eps \del_x \big ( G(U) \cdot B(U) \del_x U \big ) - 
\eps U_x \cdot \nabla G(U)^T B (U) U_x  \, .
\end{equation}

We proceed to place the model~\eqref{vhcg} within the general theory for  \eqref{hyppar1d} described above.
To this end, we set
\begin{equation*}
\begin{split}
U  =  
\begin{pmatrix}
u \\ v \\ \theta
\end{pmatrix}
\quad
A(U) =
\begin{pmatrix}
u \\ v \\ \tfrac{1}{2} v^2 + e(u,\theta)
\end{pmatrix}
\quad
F(U) =
- \begin{pmatrix}
v \\ \sigma (u, \theta)  \\ v \, \sigma(u, \theta)
\end{pmatrix}
\quad
G(U) =
\begin{pmatrix}
\frac{\sigma(u,\theta)}{\theta} \\ \frac{v}{\theta} \\ - \frac{1}{\theta}
\end{pmatrix}
\end{split}
\end{equation*}
and note that
$$
 \nabla G(U)^T B (U) = 
 \begin{pmatrix}
\frac{\sigma_u }{\theta} & 0 & 0 
\\
0 &  \frac{1}{\theta} & 0
\\
- \frac{\sigma }{\theta^2}  &  -  \frac{v}{\theta^2}  & \frac{1}{\theta^2}
\end{pmatrix}
\, 
 \begin{pmatrix}
0 & 0 & 0 
\\
0 &  \mu  & 0
\\
0  &  \mu v   & \kappa
\end{pmatrix}
=
 \begin{pmatrix}
0 & 0 & 0 
\\
0 & \frac{1}{\theta} \mu  & 0
\\
0  &  0   & \frac{1}{\theta^2} \kappa
\end{pmatrix} \, .
$$
The dissipativity hypothesis \eqref{1dh4} translates to $\mu(u,\theta) \ge 0$, $\kappa (u, \theta ) \ge 0$.

The global invertibility hypothesis \eqref{hypns} of $A(U)$ is secured by the assumption $e_\theta (u, \theta ) > 0$.
With a slight abuse of notation, we set 
$$
\heta (U) = - \eta (u, \theta) \, ,
$$ 
where $\heta (U) $ is the "mathematical" entropy  and $\eta (u, \theta)$ the thermodynamic entropy in \eqref{const}.
Note that
$$
G(U) \cdot \nabla A(U) = \nabla_U \heta(U)  = - \nabla_{ \substack{\scalebox{0.6}{$ (u,v,\theta)$ }  } }  \eta (u,\theta) \, , \quad
G(U) \cdot \nabla F(U) = 0
$$
which should be compared to \eqref{hypep} (or \eqref{1dh2}). Solutions of  \eqref{vhcg} satisfy  the entropy production identity \eqref{entropyp} 
which stands for the specification of \eqref{dissstr1d} in the setting of \eqref{vhcg}.

Apply now the formula \eqref{defrelen} to the expression $\heta (U) = - \eta (u, \theta)$ to obtain
\begin{align}
\heta (U | \bU ) &= \heta (U) - \heta (\bU) - G(\bU) \cdot ( A(U) - A(\bU))
\nonumber
\\
&=  - \eta(u,\theta) + \eta ( \bu, \btheta) - \frac{1}{\btheta} ( \bsigma, \bv, -1) 
\cdot \Big ( u-\bu, \, v-\bv,  \, e(u, \theta)  + \tfrac{1}{2} v^2 - e (\bu, \btheta) - \tfrac{1}{2} \bv^2 \Big )
\label{defrelenth}
\\
&\stackrel{\eqref{const}}{=} \frac{1}{\btheta} \Big [ \psi (u, \theta | \bu, \btheta) +\tfrac{1}{2} (v - \bv)^2 + ( \eta (u,\theta) - \eta (\bu, \btheta))(\theta - \btheta) \Big ]\,,
\label{comment}
\end{align}
where 
\begin{equation*}
\begin{aligned}
\psi (u, \theta | \bu, \btheta) &= \psi -  \psi (\bu, \btheta) - \frac{ \del \psi}{\del u} (\bu, \btheta) (u - \bu)
 - \frac{ \del \psi}{\del \theta} (\bu, \btheta) (\theta - \btheta)
 \\
 &= \psi - \bar \psi - \bsigma (u - \bu) + \bar \eta (\theta - \btheta)
 \end{aligned}
 \end{equation*}
 with the notation $\bar \psi = \psi (\bu, \btheta)$, $\bar \eta = \eta (\bu, \btheta)$ and so on.
 Finally, note that
 $$
 \nabla^2 \heta (U) - G(U) \cdot \nabla^2 A(U) =
 \begin{pmatrix}
 \tfrac{1}{\theta} e_{uu} - \eta_{uu} & 0 & \tfrac{1}{\theta} e_{u \theta} - \eta_{u \theta}
 \\
 0 &  \tfrac{1}{\theta} & 0
 \\
 \tfrac{1}{\theta} e_{u \theta} - \eta_{u \theta}  & 0 & \tfrac{1}{\theta} e_{\theta \theta} - \eta_{\theta \theta} 
 \end{pmatrix}
 \stackrel{\eqref{maxwell}}{=} \begin{pmatrix}
  \tfrac{1}{\theta} \psi_{uu} & 0 & 0
  \\
  0 &  \tfrac{1}{\theta} & 0
  \\
0 & 0 & \tfrac{1}{\theta} \eta_{\theta } 
 \end{pmatrix}\,.
 $$
 Hence, the condition \eqref{hyppd} of positive definiteness in \eqref{1dh3}
 is equivalent to the usual Gibbs thermodynamic stability conditions $\psi_{uu} > 0$ and $\eta_\theta > 0$.
 (By \eqref{maxwell} these assumptions are consistent with $e_\theta > 0$.)

 \begin{remark}
 \label{rmkconvexity} \rm The notation \eqref{comment} may be somewhat misleading as  \eqref{1dh3} does not amount
 to convexity of the functions appearing explicitly in \eqref{comment}. Instead, proceeding along the lines
 of  Section \ref{sechypconv}, introduce the conserved variables $V = A(U)$, where $V = (u,v, E)$, and define the entropy 
 $\Heta (V)$ via the relation
 $$
 \heta (U) = \Heta \circ A (U) \, .
 $$
 A tedious but straightforward adaptation of the computation in \eqref{appform1} and \eqref{appform2} indicates that
 $$
 (\Heta_u , \Heta_v, \Heta_E) (A(U) ) = (-\eta_u , 0 , \eta_\theta) \cdot 
 \begin{pmatrix}
 1 & 0 & 0
  \\
  0 &  1 & 0
  \\
e_u  & v & e_{\theta } 
 \end{pmatrix}^{-1}
= G(U)
$$
and
$$
 \begin{pmatrix}
 1 & 0 & 0
  \\
  0 &  1 & 0
  \\
e_u  & v & e_{\theta } 
 \end{pmatrix}^T
\cdot 
\nabla^2_{V} \Heta  (A(U) ) 
 \begin{pmatrix}
 1 & 0 & 0
  \\
  0 &  1 & 0
  \\
e_u  & v & e_{\theta } 
 \end{pmatrix}
 = 
 \begin{pmatrix}
  \tfrac{1}{\theta} \psi_{uu} & 0 & 0
  \\
  0 &  \tfrac{1}{\theta} & 0
  \\
0 & 0 & \tfrac{1}{\theta} \eta_{\theta } 
 \end{pmatrix}
 > 0 \, .
$$
The conditions $\psi_{uu} > 0$, $\eta_\theta > 0$ are thus equivalent to the convexity of $\Heta (u,v,E)$
and to the symmetrizability of the equations of one-dimensional gas dynamics.
 \end{remark}

 \subsection{The relative entropy identity}
 \label{sec-relenvhcg}
 
 We proceed to derive the relative entropy identity following the general procedure outlined in Section \ref{secrelenhyppar}.
 Let $(u, v, \theta)$ and $(\bu, \bv, \btheta)$ be two solutions of system~\eqref{vhcg},  each satisfying the associated entropy production
 identity \eqref{entropyp}. Using \eqref{entropyp} we obtain
 \begin{equation}
 \label{identh1}
 \begin{aligned}
 \del_t \Big [ - \btheta \eta + \btheta \bareta \Big ] 
& + \del_x \Big [ \btheta \frac{\kappa \theta_x}{\theta} - \btheta \frac{\bkappa \btheta_x}{\btheta} \Big ]
= 
 - \btheta_t \, (\eta - \bareta) +  \btheta_x \Big ( \frac{\kappa \theta_x}{\theta} - \frac{\bkappa \btheta_x}{\btheta} \Big )
 \\
 &\quad -\btheta \Big ( \mu \frac{v_x^2}{\theta} - \bmu \frac{ \bv_x^2}{\btheta} \Big )
 - \btheta \Big ( \kappa \frac{\theta_x^2}{\theta^2} - \bkappa \frac{\btheta_x^2}{\btheta^2} \Big )
 - \btheta  \Big ( \frac{r}{\theta} - \frac{\br}{\btheta} \Big ) \, .
 \end{aligned}
 \end{equation}
 Next, we subtract equations \eqref{vhcg} for each of the two solutions $(u,v,\theta)$ and $(\bu,\bv,\btheta)$
 and multiply the result by $-G(\bU)=(- \tfrac{\bsigma}{\btheta} , -\tfrac{\bv}{\btheta}, 1)$ and obtain after rearrangement the following identity:
  \begin{equation}
 \label{identh2}
  \begin{aligned}
  &\del_t \Big ( - \bsigma (u-\bu) - \bv (v - \bv) + ( e + \tfrac{1}{2} v^2 - \be - \tfrac{1}{2} \bv^2 ) \Big ) 
  + \del_x \big ( \bsigma (v -\bv) + \bv (\sigma - \bsigma)  - \sigma v + \bsigma \bv \big )
  \\
  &= - \bsigma_t  (u - \bu)  - \bv_t (v-\bv) + \bsigma_x (v - \bv) + \bv_x (\sigma -\bsigma)
  + (- \bv) \Big [ (\mu v_x - \bmu \bv_x )_x + (f - \barf) \Big ] 
  \\
  &\quad + \Big [ \del_x ( \mu v v_x - \bmu \bv \bv_x ) + \del_x (\kappa \theta_x - \bkappa \btheta_x ) + (r -\br) + (f v - \barf \bv) \Big ]\,.
  \end{aligned}
 \end{equation}

 Next, we add \eqref{identh1} with \eqref{identh2} and use \eqref{const} and \eqref{defrelenth} to obtain
   \begin{equation}
 \label{identh3}
  \begin{aligned}
  \del_t \Big (  \psi(u,\theta | \bu, \btheta) + \tfrac{1}{2} (v - \bv)^2 + (\eta - \bareta) (\theta - \btheta)  \Big )
& -  \del_x \big ( (\sigma - \bsigma ) (v -\bv) \big )
  \\
  &= I_1 + I_2 + I_3 + I_4 + I_5\,,
  \end{aligned}
 \end{equation}
  where
   \begin{align}
 \nonumber
I_1   &= - \btheta_t (\eta - \bareta)  - \bsigma_t  (u - \bu)  - \bv_t (v-\bv) + \bsigma_x (v - \bv) + \bv_x (\sigma -\bsigma)
\\
&= - ( \bv_t - \bsigma_x ) (v- \bv) + \Big [ - \btheta_t (\eta - \bareta)  - \bsigma_t  (u - \bu) + \bu_t (\sigma - \bsigma) \Big ]
\label{idefi1}
\\
&=:  A + B\,,
\nonumber
\end{align}
while
 \begin{align}
  \label{defi2}
 I_2  &= - \del_x \Big (  \btheta \frac{\kappa \theta_x}{\theta} - \btheta \frac{\bkappa \btheta_x}{\btheta} \Big ) 
              + \del_x ( \kappa \theta_x - \bkappa \btheta_x ) 
              + \del_x \Big ( - \bv (\mu v_x - \bmu \bv_x) + \mu v v_x - \bmu \bv \bv_x \Big  )\,,
 \\
  \label{defi3}
 I_3 &= \btheta_x \Big ( \frac{\kappa \theta_x}{\theta} - \frac{\bkappa \btheta_x}{\btheta} \Big ) 
 - \btheta \Big ( \kappa \frac{\theta_x^2}{\theta^2} - \bkappa \frac{\btheta_x^2}{\btheta^2} \Big )\,,
 \\
  \label{defi4}
 I_4 &= \bv_x (\mu v_x - \bmu \bv_x)
 -\btheta \Big ( \mu \frac{v_x^2}{\theta} - \bmu \frac{ \bv_x^2}{\btheta} \Big )\,,
  \\
   \label{defi5}
 I_5 &= (r - \br) - \btheta \Big ( \frac{r}{\theta} - \frac{\br}{\btheta} \Big ) + (f v - \barf \bv) - \bv (f - \barf)\,.
 \end{align}

 The last step is to re-arrange the terms and collect them together in groups of likewise terms. In this direction, we first use
\eqref{vhcg} to obtain
\begin{equation}
\label{identh4}
\begin{aligned}
A  &= - (v-\bv) \big [ (\bmu \bv_x)_x + \bar f \big ] 
\\
&= - \del_x \Big ( (v - \bv) \bmu \bv_x   \Big ) + (v_x - \bv_x ) \bmu \bv_x - (v - \bv) \barf
\\
&=: i_2 + i_4 + i_5\,.
\end{aligned}
\end{equation}
Again using \eqref{vhcg}, \eqref{const} and \eqref{maxwell} we derive
\begin{align}
B  &= - \btheta_t \Big ( \eta(u, \theta)  - \eta (\bu, \btheta) + \sigma_\theta (\bu, \btheta) (u - \bu) 
- \eta_\theta (\bu, \btheta) (\theta - \btheta) \Big )
\nonumber
\\
&\quad + \bu_t \Big ( \sigma (u, \theta) - \sigma(\bu, \btheta) - \sigma_u (\bu, \btheta) (u - \bu) 
+ \eta_u (\bu, \btheta) (\theta - \btheta) \Big )
\nonumber
\\
&\quad -\btheta_t \eta_\theta (\bu, \btheta) (\theta - \btheta)  - \bu_t \eta_u (\bu, \btheta) (\theta - \btheta) 
\nonumber
\\
&= - \btheta_t  \, \eta (u, \theta | \bu, \btheta) + \bu_t  \, \sigma(u, \theta | \bu, \btheta) - \bareta_t  (\theta - \btheta)
\nonumber
\\
&\stackrel{\eqref{entropyp}}{=}
-  \btheta_t  \, \eta (u, \theta | \bu, \btheta) + \bu_t  \, \sigma(u, \theta | \bu, \btheta)  
 - (\theta - \btheta) \Big [ \left ( \frac{\bkappa \btheta_x }{\btheta} \right )_x + \bmu \frac{\bv_x^2}{\btheta} 
 + \frac{\bkappa \btheta_x^2}{\btheta^2} + \frac{\br}{\btheta} \Big ]
 \nonumber
 \\
 &= \Big ( -  \btheta_t  \, \eta (u, \theta | \bu, \btheta) + \bu_t  \, \sigma(u, \theta | \bu, \btheta) \Big ) 
 + \del_x \Big ( - (\theta - \btheta)  \frac{\bkappa \btheta_x }{\btheta} \Big ) 
 +  \Big ( (\theta_x - \btheta_x) \frac{\bkappa \btheta_x }{\btheta}  + \frac{\bkappa \btheta_x^2}{\btheta^2} \Big )
\nonumber
\\
&\qquad + \Big ( - (\theta - \btheta) \bmu \frac{\bv_x^2}{\btheta}  \Big ) +  \Big (  - (\theta - \btheta) \frac{\br}{\btheta} \Big )
 \label{identh5}
\\
&=:  j_1 + j_2 + j_3 +  j_ 4 + j_5\,.
\nonumber
\end{align}
where the $j_i$'s stand for each of the last five terms in \eqref{identh5}. Observe next that the terms can be regrouped as follows:
$$
\begin{aligned}
I_2 + i_2 +  j_2 &=  \del_x \Big ( (\theta - \btheta) \big ( \frac{\kappa \theta_x}{\theta} - \frac{\bkappa \btheta_x}{\btheta} \big ) 
+ (\mu v_x -  \bmu \bv_x ) (v - \bv) \Big )
\\
I_3 + j_3  &=  - \Big ( \frac{ \theta_x}{\theta} - \frac{ \btheta_x}{\btheta}   \Big ) 
           \Big (  \btheta \frac{\kappa \theta_x}{\theta} - \theta  \frac{\bkappa \btheta_x}{\btheta}   \Big )
           = - \btheta \kappa \Big ( \frac{ \theta_x}{\theta} - \frac{ \btheta_x}{\btheta}   \Big )^2 
               - \Big ( \frac{ \theta_x}{\theta} - \frac{ \btheta_x}{\btheta}   \Big ) \frac{ \btheta_x}{\btheta}   (\btheta \kappa - \theta \bkappa )
\\
I_4 + i_4 + j_4 &= - \theta \btheta \Big ( \mu \frac{v_x}{\theta} - \bmu \frac{\bv_x}{\btheta} \Big ) 
\Big ( \frac{v_x}{\theta} - \frac{\bv_x}{\btheta} \Big ) =
- \theta \btheta \mu \Big (  \frac{v_x}{\theta} - \frac{\bv_x}{\btheta} \Big )^2 
- \theta \btheta (\mu - \bmu)  \frac{\bv_x}{\btheta} (  \frac{v_x}{\theta} - \frac{\bv_x}{\btheta} \Big )
\\
I_5 + i_5 + j_5 &= (f - \barf)( v - \bv) + (\theta - \btheta) \Big ( \frac{r}{\theta} - \frac{\br}{\btheta} \Big)\,.
\end{aligned}
$$

We then substitute \eqref{idefi1}, \eqref{identh4} and \eqref{identh5} to \eqref{identh3} and arrive at the final
form of the relative entropy identity
\begin{equation}
 \label{idenrelenth}
\begin{aligned}
\del_t  &\Big (  \psi(u,\theta | \bu, \btheta) + \tfrac{1}{2} (v - \bv)^2 + (\eta - \bareta) (\theta - \btheta)  \Big )
\\
&\quad-  \del_x \left [ (\sigma - \bsigma ) (v -\bv)  +   (\mu v_x -  \bmu \bv_x ) (v - \bv)  + 
 (\theta - \btheta) \big ( \frac{\kappa \theta_x}{\theta} - \frac{\bkappa \btheta_x}{\btheta} \big )   \right ]
 \\
&\quad + \btheta \kappa \Big ( \frac{ \theta_x}{\theta} - \frac{ \btheta_x}{\btheta}   \Big )^2 
   + \theta \btheta  \mu \Big (  \frac{v_x}{\theta} -  \frac{\bv_x}{\btheta} \Big )^2 
  \\
  &= -  \btheta_t  \, \eta (u, \theta | \bu, \btheta) + \bu_t  \, \sigma(u, \theta | \bu, \btheta)  
  + (f - \barf)( v - \bv) + (\theta - \btheta) \Big ( \frac{r}{\theta} - \frac{\br}{\btheta} \Big)
  \\
 &\quad - \Big ( \frac{ \theta_x}{\theta} - \frac{ \btheta_x}{\btheta}   \Big ) \frac{ \btheta_x}{\btheta}   (\btheta \kappa - \theta \bkappa )
 - \theta \btheta (\mu - \bmu)  \frac{\bv_x}{\btheta} (  \frac{v_x}{\theta} - \frac{\bv_x}{\btheta} \Big )\,,
 \end{aligned}
 \end{equation}
where 
\begin{equation}
\begin{aligned}
\eta (u, \theta | \bu, \btheta) &= \eta(u, \theta)  - \eta (\bu, \btheta) - \eta_u (\bu, \btheta) (u - \bu) 
- \eta_\theta (\bu, \btheta) (\theta - \btheta) 
\\
\sigma (u, \theta | \bu, \btheta) &= \sigma(u, \theta)  - \sigma (\bu, \btheta) - \sigma_u (\bu, \btheta) (u - \bu) 
- \sigma_\theta (\bu, \btheta) (\theta - \btheta) \,.
\end{aligned}
\end{equation}



%
%
%

\section{Application to the constitutive theory of thermoviscoelasticity}
\label{sec-thermov}

In this section we perform the relative entropy calculation for the system of thermoviscoelasticity in several
space dimensions. This calculation is an extension of that in Section \ref{sec-vhcg} as the system of thermoviscoelasticity 
when restricted to one-space dimension (when restricted to the particular case of Stokes viscosity and Fourier heat conduction) produces precisely 
the  system \eqref{vhcg}.  

\subsection{The constitutive theory}
\label{S4.const}
The requirements imposed from thermodynamics on the constitutive theory of thermoviscoelasticity 
were developed in \cite{cn63, cm64} and a summary can be found in \cite[Sec 3.2]{daf10}. The constitutive functions depend
on the deformation gradient $F$, the strain rate $\dot F$, the temperature $\theta$ and the temperature gradient $g = \nabla \theta$, hence
the name thermoviscoelasticity. The elastic part is generated by a free energy 
function $\psi$:
\begin{equation}
\label{thermovcr}
\begin{aligned}
\psi &= \psi ( F, \theta)\,,
\\
\Sigma &=  \; \frac{\del \psi}{\del F} (F, \theta) \,,
\\
\eta &= - \frac{\del \psi}{\del \theta} (F, \theta) \,,
\\
e &= \psi + \theta \eta \, ;
\end{aligned}
\end{equation}
note that \eqref{thermovcr} imply the Maxwell relations
\begin{equation}
\label{maxwellrel}
\frac{\del \Sigma_{i \alpha}}{\del \theta} = - \frac{\del \eta}{\del F_{i \alpha}} \, , \quad 
\frac{\del \Sigma_{i \alpha}}{\del F_{j \beta}} = \frac{\del^2 \psi}{\del F_{i \alpha} \del F_{j \beta}} = \frac{\del \Sigma_{j \beta} }{\del F_{i \alpha}}\,.
\end{equation}

The total stress is decomposed into an elastic part $\Sigma$ and a viscoelastic part $Z = Z (F, \theta, g, \Fdot )$ where $\Sigma$ and $Z$ are
both symmetric tensor valued functions, $Z (F, \theta, 0, 0) = 0$ so that $\Sigma$ is indeed the elastic part, according to the formula
\begin{equation}
\label{thermovcr2}
\begin{aligned}
S &= \Sigma (F, \theta) + Z (F, \theta, g, \Fdot )  
\\
&= \frac{\del \psi}{\del F} (F, \theta) + Z (F, \theta, g, \Fdot ) \,,
\\
Q &= Q( F, \theta, g)\,.
\end{aligned}
\end{equation}
Moreover, the heat flux $Q$ and  the viscoelastic contribution to the stress $Z$ have to satisfy
\begin{equation}
\label{cdineq}
\tfrac{1}{\theta} G \cdot Q(F, \theta, g) + \Fdot : Z (F, \theta, g, \Fdot) \ge 0   \quad  \forall (F, \theta,  g, \Fdot)\, ,
\tag{H}
\end{equation}
which along with \eqref{thermovcr} guarantee consistency for smooth processes with the Clausius-Duhem inequality \cite{cn63, cm64}.

Here, for simplicity, we place the additional assumption  $Z = Z(F, \theta, \Fdot)$, that is $Z$ is taken independent of $g$. 
Then condition~\eqref{cdineq} implies $Q (F, \theta, 0) = 0 $, $Z (F, \theta, 0) = 0$, and accordingly \eqref{cdineq} decomposes into two
distinct inequalities
\begin{equation}
\label{cdineqp}
\tfrac{1}{\theta} g \cdot Q(F, \theta, g) \ge 0 \quad \mbox{and} \quad  \Fdot : Z (F, \theta, \Fdot) \ge 0  \,.
\tag{H$^\prime$}
\end{equation}

In summary, the system of thermoviscoelasticity reads
\begin{equation}
\label{thermov}
\begin{aligned}
F_t &= \nabla v
\\
v_t  &= \div ( \Sigma + Z ) + f
\\
\del_t ( \tfrac{1}{2} |v|^2 + e ) &= \div ( v \cdot  \Sigma + v \cdot Z) + \div Q + v \cdot f + r\,.
\end{aligned}
\end{equation}
where $x$ stands for the Lagrangean variable,  $\div$ is the usual divergence operator (in referential coordinates), while $\del_t$ stands here for
 the material derivative.
 Smooth solutions of \eqref{thermov} satisfy the energy dissipation identity
 $$
 \del_t e = \nabla v :  ( \Sigma  + Z ) + \div Q + r
 $$
 and, using the constitutive hypotheses of the theory, one arrives at the entropy production identity
 \begin{equation}
 \label{entropypr}
 \del_t \eta - \div \frac{Q}{\theta} = \frac{1}{\theta^2} \nabla \theta \cdot Q + \frac{1}{\theta} \nabla v : Z + \frac{r}{\theta}\,.
 \end{equation}
 
\subsection{The relative entropy identity}\label{S4.1}
 
In a similar fashion to  Section \ref{sec-relen-defvhcg}, we set
\begin{equation*}
\begin{split}
U  =  
\begin{pmatrix}
F \\ v \\ \theta
\end{pmatrix}\in \R^{d^2 + d +1}\;,
\quad
A(U) =
\begin{pmatrix}
F \\ v \\ \tfrac{1}{2} v^2 + e(F,\theta) 
\end{pmatrix}
\;,
\quad
G (U) =
\begin{pmatrix}
\frac{\Sigma(F,\theta)}{\theta} \\ \frac{v }{\theta} \\ -\frac{1}{\theta}
\end{pmatrix}
\end{split}
\end{equation*}
where the tensor $F \in \R^{d^2}$ is viewed as a column vector in forming $U$, while $\Sigma$ is determined by \eqref{thermovcr} and is
viewed again as column vector.
We impose $e_\theta (F, \theta) > 0$, so that $A(U)$ is globally invertible, and set  
\begin{equation*}
\heta (U) := - \eta (F, \theta) \, , 
\end{equation*}
where $\heta (U)$ is the mathematical entropy
and $\eta (u, \theta)$ the thermodynamic one. The relative entropy is defined by the formula
\begin{equation}
\begin{aligned}
\heta (U | \bU ) &=  - \eta(F,\theta) + \eta ( \bF , \btheta) - \frac{1}{\btheta} ( \bSigma, \bv, -1) 
\cdot \Big ( F-\bF, \, v-\bv,  \, e(F, \theta)  + \tfrac{1}{2} |v|^2 - e (\bF, \btheta) - \tfrac{1}{2} |\bv|^2 \Big )
\\
&\stackrel{\eqref{thermovcr}}{=} \frac{1}{\btheta} \Big [ \psi (F, \theta | \bF, \btheta) +\tfrac{1}{2} |v - \bv|^2 + ( \eta (F,\theta) - \eta (\bF, \btheta))(\theta - \btheta) \Big ]\,,
\end{aligned}
\label{defrelenmd}
\end{equation}
where
\begin{equation}
\label{defrelpsi}
\begin{aligned}
\psi (F, \theta | \bF, \btheta) &= \psi (F, \theta)  -  \psi (\bF, \btheta) - \frac{ \del \psi}{\del F}  (\bF, \btheta) :  (F  - \bF)
 - \frac{ \del \psi}{\del \theta} (\bF, \btheta) (\theta - \btheta)
 \\
 &= \psi - \bar \psi - \bSigma  : (F - \bF) + \bar \eta (\theta - \btheta)
 \end{aligned}
 \end{equation}
 with  $\bar \psi = \psi (\bF, \btheta)$, $\bar \eta = \eta (\bF, \btheta)$ and so on.
 We again note that
 $$
 \nabla^2 \heta (U) - G(U) \cdot \nabla^2 A(U) 
  \stackrel{\eqref{thermovcr}}{=} \begin{pmatrix}
  \tfrac{1}{\theta} \psi_{ FF} & 0 & 0
  \\
  0 &  \tfrac{1}{\theta} & 0
  \\
0 & 0 & \tfrac{1}{\theta} \eta_{\theta } 
 \end{pmatrix}
 $$
and the positivity for the matrix $\nabla^2 \heta (U) - G(U) \cdot \nabla^2 A(U)$
 is equivalent to the usual Gibbs thermodynamic stability conditions $\psi_{FF} > 0$ and $\eta_\theta > 0$.
 
 We next follow Section \ref{sec-relenvhcg} adapted to the present multi-dimensional case. Similar calculations can be
 found in \cite{dafermos79} for the case when viscosity and heat conduction are absent. 
 Let $(F, v, \theta)$ and $(\bF, \bv, \btheta)$
 be two smooth solutions of \eqref{thermov}  with temperatures  $\theta > 0$ and $\btheta > 0$ that satisfy \eqref{entropypr}.
 We subtract equations \eqref{entropypr} for the two respective solutions,
 multiply by $\btheta$ and rewrite the result in the form
 \begin{equation}
 \label{thermoiden1}
 \begin{aligned}
 \del_t \big ( - \btheta \eta  + \btheta \bareta \big )  + \div \big ( \btheta \frac{Q}{\theta} - \btheta \frac{ \bQ}{\btheta} \big ) = - (\del_t \btheta) (\eta - \bareta) 
  + \nabla_x \btheta \cdot \big ( \frac{Q}{\theta}  -  \frac{ \bQ}{\btheta} \big )
 \\
 - \btheta \Big ( \frac{ \nabla v : Z }{\theta} -  \frac{ \nabla \bv : \bZ }{\btheta} \Big )
 - \btheta \Big ( \frac{ \nabla \theta \cdot Q}{\theta^2} - \frac{ \nabla \btheta \cdot \bQ}{\btheta^2} \Big ) - \btheta \big (\frac{r}{\theta} - \frac{\br}{\btheta} \big ) \,.
\end{aligned}
\end{equation}

Next, write the difference between equations \eqref{thermov} for the two solutions, multiply the resulting identity
by $( - \bSigma_{i \alpha}, - \bv_i , 1)$ and perform some re-organization of the terms to obtain
 \begin{equation}
 \label{thermoiden2}
 \begin{aligned}
 &\del_t \Big ( - \bSigma_{i \alpha} (F_{i \alpha} - \bF_{i \alpha} ) - \bv_i (v_i -\bv_i) + (e + \tfrac{1}{2} |v|^2 - \be - \tfrac{1}{2} |\bv|^2 ) \Big )
 \\
 &\qquad+ \del_{\alpha} \Big ( (v_i -\bv_i) \bSigma_{i \alpha} + \bv_i (\Sigma_{i \alpha} - \bSigma_{i \alpha}) - v_i \Sigma_{i \alpha} + \bv_i \bSigma_{i \alpha} \Big )
 \\
&\quad=
 - ( \del_t \bSigma_{i \alpha} ) (F_{i \alpha} - \bF_{i \alpha} ) - (\del_t \bv_i ) (v_i -\bv_i) + (\del_\alpha \bSigma_{i \alpha} ) (v_i -\bv_i) 
 + (\del_\alpha \bv_i ) (\Sigma_{i \alpha} - \bSigma_{i \alpha}) 
\\
&\qquad+ (-\bv_i) \Big [ (Z_{i\alpha} - \bZ_{i\alpha} )_{x_\alpha} + (f_i -\barf_i) ) \Big ] 
\\
&\qquad + \Big [ \del_\alpha ( v_i Z_{i\alpha} -  \bv_i \bZ_{i\alpha} ) + \del_\alpha (Q_\alpha - \bQ_\alpha) + ( r - \br) + ( v \cdot f - \bv \cdot \barf ) \Big ]\,.
\end{aligned}
\end{equation}
We then combine \eqref{thermoiden1}, \eqref{thermoiden2} and the identity
$$
- (\del_t \bv_i  - \del_\alpha \bSigma_{i \alpha} ) (v_i -\bv_i)   = - \del_\alpha [  (v_i  -  \bv_i ) \bZ_{i\alpha}  ] + ( \del_\alpha v_i  -  \del_\alpha \bv_i ) \bZ_{i\alpha} 
- (v_i  -  \bv_i ) \barf_i\;,
$$
together with \eqref{defrelenmd} and \eqref{maxwellrel} to obtain
\begin{equation}
\label{thermoiden3}
\begin{aligned}
&\del_t \Big [ \psi (F, \theta | \bF, \btheta ) + (\eta - \bareta)(\theta - \btheta) + \tfrac{1}{2} |v -\bv|^2 \Big ]
\\
&\quad
+ \div \Big ( - (v - \bv) \cdot ( \Sigma + Z - \bSigma - \bZ ) + \btheta ( \frac{Q}{\theta} - \frac{\bQ}{\btheta} ) - (Q - \bQ) \Big )
\\
&\quad =
-\btheta_t  \Big [ \eta(F,\theta)  - \eta (\bF, \btheta) + \frac{\del \Sigma_{i\alpha}}{\del \theta} (\bF, \btheta)  (F_{i\alpha} -\bF_{i\alpha}) 
- \frac{\del \eta}{\del \theta} (\bF, \btheta) \,  (\theta - \btheta) \Big ]
\\
&\qquad +  (\bF_{j\beta})_t \Big [ \Sigma_{j\beta} (F, \theta) - \Sigma_{j\beta} (\bF, \btheta) - \frac{\del \Sigma_{i\alpha}}{\del F_{j\beta}}  (\bF, \btheta)  \, (F_{i\alpha} -\bF_{i\alpha}) 
 - \frac{\del \Sigma_{j\beta}}{\del \theta}   (\bF, \btheta) \, (\theta - \btheta) \Big ]
\\
&\qquad  -\btheta_t  \frac{\del \eta}{\del \theta} (\bF, \btheta) \, (\theta - \btheta) - (\bF_{j\beta})_t  \frac{\del \eta}{\del F_{j\beta}} (\bF, \btheta) \, (\theta - \btheta)  + I_1 + I_2 + I_3
\\
&\quad = - (\btheta_t ) \eta(F, \theta | \bF, \btheta) +  \bF_t : \Sigma (F, \theta | \bF, \btheta) - (\del_t \bareta) (\theta - \btheta) + I_1 + I_2 + I_3\,,
\end{aligned}
\end{equation}
where we have set
\begin{equation}
\label{defrelstress}
\begin{aligned}
\eta(F, \theta | \bF, \btheta) &:= \eta(F,\theta)  - \eta (\bF, \btheta) + \frac{\del \Sigma_{i\alpha}}{\del \theta} (\bF, \btheta)  (F_{i\alpha} -\bF_{i\alpha}) 
- \frac{\del \eta}{\del \theta} (\bF, \btheta) \,  (\theta - \btheta)\,,
\\
\Sigma_{j \beta} (F, \theta | \bF, \btheta) &:= 
\Sigma_{j\beta} (F, \theta) - \Sigma_{j\beta} (\bF, \btheta) - \frac{\del \Sigma_{i\alpha}}{\del F_{j\beta}}  (\bF, \btheta)  \, (F_{i\alpha} -\bF_{i\alpha}) 
 - \frac{\del \Sigma_{j\beta}}{\del \theta}   (\bF, \btheta) \, (\theta - \btheta)\,,
 \end{aligned}
\end{equation}
and
\begin{equation}
\label{thermoiden4}
\begin{aligned}
I_1 &= \nabla_x \bv : ( Z - \bZ) - \btheta \Big ( \frac{\nabla v : Z}{\theta} - \frac{\nabla \bv : \bZ}{\btheta} \Big ) + (\nabla v - \nabla \bv) : \bZ\,,
\\
I_2 &= \nabla \btheta \cdot \Big ( \frac{Q}{\theta} - \frac{\bQ}{\btheta} \Big ) - \btheta \Big ( \frac{\nabla \theta \cdot Q}{\theta^2} -  \frac{\nabla \btheta \cdot \bQ}{\btheta^2} \Big ) \,,
\\
I_3 &= (v - \bv) \cdot (f - \barf) + (r - \br) - \btheta \big ( \frac{r}{\theta} - \frac{\br}{\btheta} \big )\,.
\end{aligned}
\end{equation}

Next, note that identity~\eqref{entropypr} yields
$$
- (\del_t \bareta) (\theta - \btheta ) = - \del_\alpha \big ( \frac{ \bQ_\alpha}{\btheta} (\theta - \btheta ) \big ) + \frac{ \bQ_\alpha}{\btheta}  \del_\alpha (\theta - \btheta )
- \Big [  \frac{1}{\btheta^2}  \nabla \btheta \cdot \bQ + \frac{1}{\btheta} \nabla \bv :  \bZ  + \frac{\br}{\btheta} \Big ] (\theta - \btheta )\,.
$$
The latter, together with \eqref{thermoiden4}, allows  to rewrite \eqref{thermoiden3} in its final form
\begin{equation}
\label{relenthermov}
\begin{aligned}
&\del_t \Big ( \psi (F, \theta | \bF, \btheta ) + (\eta - \bareta)(\theta - \btheta) + \tfrac{1}{2} |v -\bv|^2 \Big )
\\
&\quad -  \div \Big ( (v - \bv) \cdot ( \Sigma + Z - \bSigma - \bZ ) +  (\theta - \btheta ) \big ( \frac{Q}{\theta} - \frac{\bQ}{\btheta} \big ) \Big )
\\
&\qquad =  - \btheta_t \,  \eta(F, \theta | \bF, \btheta) + \bF_t : \Sigma (F, \theta | \bF, \btheta)
\\ 
&\qquad \quad - \theta \btheta \Big ( \frac{\nabla v}{\theta} - \frac{\nabla \bv}{\btheta} \Big ) : \Big ( \frac{Z}{\theta} - \frac{\bZ}{\btheta} \Big )
- \Big ( \btheta \frac{Q}{\theta} - \theta \frac{\bQ}{\btheta} \Big ) \cdot \Big ( \frac{\nabla \theta}{\theta} - \frac{\nabla \btheta}{\btheta} \Big )
\\
&\qquad \quad + (v - \bv) \cdot (f - \barf) + (\theta - \btheta) \big ( \frac{r}{\theta} - \frac{\br}{\btheta} \big ) \, ,
\end{aligned}
\end{equation}
and provides the relative entropy formula for the system of thermoviscoelasticity~\eqref{thermov}. 
In  \eqref{relenthermov}, the effect of viscous dissipation and heat conduction is captured respectively by the terms
$$
\begin{aligned}
D_v :=  \theta \btheta \Big ( \frac{\nabla v}{\theta} - \frac{\nabla \bv}{\btheta} \Big ) : \Big ( \frac{Z}{\theta} - \frac{\bZ}{\btheta} \Big )
\\
D_q := \Big ( \btheta \frac{Q}{\theta} - \theta \frac{\bQ}{\btheta} \Big ) \cdot \Big ( \frac{\nabla \theta}{\theta} - \frac{\nabla \btheta}{\btheta} \Big )\,.
\end{aligned}
$$

\medskip
\begin{remark}\label{Rmtwothermo} \rm
The same relative entropy formula is derived when we compare two
constitutive theories that have the same thermoelastic part but different viscoelastic and heat conduction formulas.
Indeed, given two constitutive theories
$$
\psi_i = \psi_i (F,\theta) \, , \quad Q_i = Q_i (F, \theta, g) \, , \quad Z_i = Z_i (F , \theta, g , \Fdot) \, , \quad i = 1,2 \, ,
$$
such that $\psi_1 (F, \theta) = \psi_2 (F , \theta ) = : \psi (F, \theta)$ but  $Q_1 \ne Q_2$ and $Z_1 \ne Z_2$,
if  $(F, \theta)$ is a smooth solution associated to the first constitutive theory 
and $(\bF, \btheta)$  a smooth solution  associated to the second, then setting
$$
\begin{aligned}
&Q = Q_1 ( F , \theta, \nabla \theta) \, , \quad Z = Z_1 ( F , \theta, \nabla \theta , \Fdot) \, ,
\\
&\bQ = Q_2 ( \bF , \btheta, \nabla \btheta) \, , \quad \bZ = Z_2 ( \bF , \btheta, \nabla \btheta , \bFdot)  \, ,
\end{aligned}
$$
the two solutions satisfy the same relative entropy identity \eqref{relenthermov}.
In particular, this identity holds when we compare a theory of thermoviscoelasticity to the formal limiting theory
of thermoelastic non-conductors of heat (by directly setting in \eqref{relenthermov}  $\bQ = 0$ and $\bZ = 0$).
\end{remark}

\subsection{Convergence to the system of adiabatic thermoelasticity}\label{S4.conv}
Next, we consider the limiting process from the system of thermoviscoelasticity~\eqref{thermov} 
for a Newtonian viscous fluid with Fourier heat conduction in
the limit $k \to 0$, $\mu \to 0$ to the system of adiabatic thermoelasticity.
Let $U  = (F,v,\theta)^T \in \R^{d^2 + d +1}$, $\theta > 0$,  be a smooth solution of the system of thermoviscoelasticity
\begin{equation}
\label{thermov2}
\begin{aligned}
F_t &= \nabla v
\\
v_t  &= \div ( \Sigma + \mu(F,\theta)\nabla v) 
\\
\del_t ( \tfrac{1}{2} |v|^2 + e ) &= \div ( v \cdot  \Sigma +  \mu(F,\theta)v \cdot\nabla v) + \div(k(F,\theta)\nabla \theta)
\end{aligned}
\end{equation}
satisfying~\eqref{thermovcr}, \eqref{thermovcr2} for a Newtonian viscous fluid $Z=\mu(F,\theta)\nabla v$ with Fourier heat conduction $Q=k(F,\theta)\nabla \theta$, 
$\mu>0,k>0$. Let $\bar U  =  (\bF,\bv,\btheta)^T $ be a smooth solution, satisfying $\btheta \ge \delta > 0$ for some $\delta > 0$, of the equations of adiabatic thermoelasticity :
\begin{equation}
\label{athermoel}
\begin{aligned}
F_t &= \nabla v
\\
v_t  &= \div \Big ( \frac{\del \psi}{\del F}(F,\theta) \Big )
\\
\del_t ( \tfrac{1}{2} |v|^2 + e ) &= \div \Big ( v \cdot   \frac{\del \psi}{\del F}(F,\theta)  \Big )  \, .
\end{aligned}
\end{equation}
The latter system is obtained formally from \eqref{thermov} (with $f =0, r=0$)  and \eqref{thermovcr}, \eqref{thermovcr2} by setting $\bQ = 0$, $\bZ = 0$.
In what follows we compare the two solutions $U$ and $\bar U$.

By Remark~\ref{Rmtwothermo} using the identities \eqref{thermoiden3}, \eqref{thermoiden4}, and $\bar Z=0$, $\bar Q = 0$, we derive
the relative entropy identity comparing the two solutions,
\eqref{relenthermov}, we can write
\begin{equation}
\label{relenthermo0}
\begin{aligned}
&\del_t \Big ( \psi (F, \theta | \bF, \btheta ) + (\eta - \bareta)(\theta - \btheta) + \tfrac{1}{2} |v -\bv|^2 \Big )
 -  \div \Big ( (v - \bv) \cdot ( \Sigma - \bSigma +\mu\nabla v) +  (\theta - \btheta ) k\frac{\nabla\theta}{\theta} \Big )
\\
&\qquad =  - \btheta_t \,  \eta(F, \theta | \bF, \btheta) + \bF_t : \Sigma (F, \theta | \bF, \btheta)
 - \btheta \mu\,\frac{|\nabla v|^2}{\theta}
- \btheta k\,\frac{|\nabla\theta|^2}{\theta^2}
+\mu\,\nabla \bv\cdot\nabla v+ k\,\frac{\nabla\btheta\cdot\nabla\theta}{\theta}\;.
\end{aligned}
\end{equation}
(It should be noted that if we were to start with a solution $U=(F,v,\theta)^T$ which is  an entropy weak solution of~\eqref{thermov2}
 then the above identity would hold as an inequality.)
 
Throughout this section, we assume the usual Gibbs thermodynamics stability conditions, 
\begin{equation}\label{S4H}
\psi_{FF}>0,\quad\text{and}\quad\eta_\theta>0.
\tag{$\text{G}$}
\end{equation}
Before we proceed, lets give some remarks related to~\eqref{S4H}:
\begin{enumerate}
\item[(i)] Since $e_\theta=\theta\eta_\theta >  0$, it follows that $\nabla A$ is nonsingular. The map 
$$
U = (F, v, \theta) \mapsto A(U) = \big ( F, v, \tfrac{1}{2} |v|^2 + e(F, \theta) \big )
$$ 
is globally one-to one. We may invert the map $V=A(U)$ and express $U=A^{-1}(V)$.
\item[(ii)]  If we write $\Heta (V)=\heta (A^{-1}(V))$, then a calculation as in Remark \ref{rmkconvexity}
shows that 
\begin{equation}
\text{\eqref{S4H}} \Longleftrightarrow \Heta (V)\,\text{ is convex in }V\,.
\end{equation}
\item[(iii)]  The positivity of the matrix  $\nabla^2 \heta (U) - G(U) \cdot \nabla^2 A(U)$ is equivalent to~\eqref{S4H}.
\item[(iv)]  The existence of entropy-entropy flux is guaranteed by the consistency of the theory with the second law of thermodynamics.
\end{enumerate}

Next, we place some growth hypotheses: For the internal energy we assume there  is a  constant $c > 0$ and $p,\,q>1$ such that
\begin{equation}\label{littlea11}
c(|F|^p+\theta^q)-c\le e(F,\theta)\le c(|F|^p+\theta^q)+c\,,  \quad \mbox{  $\forall \, (F,\theta) \in \R^{d\times d}\times \R^+ $} \, .
\tag{a$_{1}$}
\end{equation}
For the stress $\Sigma$ and entropy $\eta$ in \eqref{thermovcr} we place the growth restrictions
\begin{equation}
\label{littlea2}
    \lim_{ |F|^p + \theta^q \to \infty}  \frac{ |\Sigma(F,\theta)| }{ |F|^p + \theta^q  }  = 0  \, ,
    \tag{a$_2$}\end{equation}
\begin{equation}
\label{littlea3}
   \lim_{ |F|^p + \theta^q \to \infty}  \frac{ | \eta(F,\theta) | }{ |F|^p + \theta^q  }  = 0  \, .
\tag{a$_3$}\end{equation}

In the sequel, we employ the notation
\begin{equation}
I(F,v,\theta|\bF,\bv,\btheta):=\psi (F, \theta | \bF, \btheta)  + ( \eta (F,\theta) - \eta (\bF, \btheta))(\theta - \btheta) +\tfrac{1}{2} |v - \bv|^2 \,,
\end{equation}
so that $\btheta\,\heta (U|\bU)=I(U|\bU)$ and define the compact set
$$\Gamma_{M,\delta}=\left\{ (\bF,\bv,\btheta):\ |\bF|\le M \,, \, |\bv|\le M,\,0<\delta\le \btheta\le M \right\}$$
where $M$ and $\delta$ are some positive constants. When employing this set, the constants are selected so that the smooth solution
$\bU = (\bF,\bv,\btheta)^T$ takes values in $\Gamma_{M,\delta}$. The following lemma establishes bounds on $I(U|\bU)$ 
that serve later for comparing two solutions $U$ and $\bU$.

\begin{lemma}\label{LemmaI}
Assume that $(\bF,\bv,\btheta)\in \Gamma_{M,\delta}$ and that $\psi(F,\theta)\in C^3 \big (\R^{d\times d}\times[0, \infty) \big )$, 
$\eta(F,\theta)$, $\Sigma(F,\theta)\in C^2(\R^{d\times d}\times [0,\infty))$ satisfy the conditions~\eqref{thermovcr} and~\eqref{S4H}.
Under the growth hypotheses~\eqref{littlea11},~\eqref{littlea2} and~\eqref{littlea3}, the following hold true:
\begin{enumerate}
\item[{(i)}] There exist $R=R(\delta,M)$, $K_1=K_1(\delta,M,c)$ and $K_2=K_2(\delta,M,c)$ such that
\begin{equation}
\label{renbo1}
I (F,v,\theta|\bF,\bv,\btheta)\ge
\begin{cases}
\frac{1}{2} K_1 \Big(|F|^p+\theta^q+|v|^2\Big) & |F|^p+\theta^q+|v|^2>R\\
K_2\Big(|F-\bF|^2+|\theta-\btheta|^2+|v-\bv|^2\Big) & |F|^p+\theta^q+|v|^2\le R
\end{cases}
\end{equation}
for all $(\bF,\bv,\btheta)\in \Gamma_{M,\delta}$.
\item[{(ii)}] There exists a constant $C>0$
\begin{equation}\label{etaboundthermov}
|\eta(F,\theta|\bF,\btheta)|\le C \,I (F,v,\theta|\bF,\bv,\btheta)\qquad \forall \,\,(F,v,\theta)
\end{equation}
for all $(\bF,\bv,\btheta)\in \Gamma_{M,\delta}$.
\item[{(iii)}] There exists a constant $C>0$
\begin{equation}\label{sigmaboundthermov}
|\Sigma(F,\theta|\bF,\btheta)|\le C \,I (F,v,\theta|\bF,\bv,\btheta)\qquad \forall \,\,(F,v,\theta)
\end{equation}
for all $(\bF,\bv,\btheta)\in \Gamma_{M,\delta}$.
\item[{(iv)}] There exist constants $K_1$ and $K_2$ such that 
\begin{equation}
\label{normlike}
I (F,v,\theta|\bF,\bv,\btheta)\ge
\begin{cases}
\frac{1}{4} K_1 \Big(|F-\bF|^p+|\theta-\btheta|^q+|v-\bv|^2\Big) & |F|^p+\theta^q+|v|^2>R\\
K_2\Big(|F-\bF|^2+|\theta-\btheta|^2+|v-\bv|^2\Big) & |F|^p+\theta^q+|v|^2\le R
\end{cases}
\end{equation}
for all $(\bF,\bv,\btheta)\in \Gamma_{M,\delta}$.
\end{enumerate}
\end{lemma}
\begin{proof}
Fix $p>1$ and $q>1$ and consider $(\bF,\bv,\btheta)\in\Gamma_{M,\delta}$. 
If we select the radius $r= r(M) := M^p+M^q+M^2$, it follows that $\Gamma_{M,\delta} \subset B_r=\{(F,v,\theta):\,\,|F|^p+\theta^q+|v|^2\le r\}$. 
The proof is divided into four steps.

{\bf Step 1.} We rewrite the quantity $I(F,v,\theta|\bF,\bv,\btheta)$ as 
\begin{equation}
I(F,v,\theta|\bF,\bv,\btheta) = e(F,\theta)- \psi( \bF, \btheta) - \psi_F(\bF,\btheta):(F-\bF)-\btheta\eta(F,\theta)+\frac{1}{2}|v-\bv|^2
\end{equation}
and proceed to estimate it  for $(\bF,\bv,\btheta)\in\Gamma_{M,\delta}$. Using \eqref{littlea11} and
Young's inequality  we obtain
$$
\begin{aligned}
I &\ge \min \{ c, \tfrac{1}{2} \} \big ( |F|^p+\theta^q+|v|^2 \big )  - C_1 | \eta(F,\theta) | - C_2 |F| - C_3 |v| - C_4
\\
&\ge K_1   \big ( |F|^p+\theta^q+|v|^2 \big )  -  C_1 | \eta(F,\theta) |  -  C_5
\end{aligned}
$$
where $K_1 = \tfrac{1}{2} \min \{ c, \tfrac{1}{2} \} $. 
Using  next \eqref{littlea3}, we select $R > r(M)+1$ sufficiently large such that 
\begin{equation}
\label{bound5}
I(F,v,\theta|\bF,\bv,\btheta) \ge  \frac{K_1}{2}(|F|^p+\theta^q+|v|^2) 
\end{equation}
for $|F|^p+\theta^q+|v|^2\ge R$ and $ (\bF,\bv,\btheta)\in \Gamma_{M,\delta}$.

In the complementary region $|F|^p+\theta^q+|v|^2\le R$, equivalently $U\in B_R$, we use the expression
\begin{equation}
\begin{aligned}
\frac{1}{\btheta} I(F,v,\theta|\bF,\bv,\btheta)&=\hat{H}(U|\bU)=H(A(U)|A(\bU))\\
&=H(A(U))-H(A(\bU))-H_V(A(\bU))(A(U)-A(\bU))
\end{aligned}
\end{equation}
and recall that $H(V)$ is convex in $V=(F,v,E)^T$ to get
\begin{equation}
\begin{aligned}
I(F,v,\theta|\bF,\bv,\btheta)&=\btheta H(A(U)|A(\bU))\\
&\ge\min_{\substack{\tilde{U}\in B_R\\\delta\le\btheta\le M}       }\left\{\btheta H_{VV}(A(\tilde U))\right\}|A(U)-A(\bU)|^2
=:K_2|A(U)-A(\bU)|^2\,,
\end{aligned}
\end{equation}
where $K_2:=\delta\displaystyle\min_{V\in A(B_R)} H_{VV}(V)>0$. Note that at this point we use the regularity assumptions of $\psi$ and $\eta$ in $(F,\theta)$. Next, we write
\begin{equation}
\begin{aligned}
|U-\bU|&=\Big|\int_0^1\frac{d}{d\tau}\left[ A^{-1}(\tau A(U)+(1-\tau)A(\bU))\right]d\tau\Big|\\
&\le \left|\int_0^1(\nabla_V(A^{-1})(\tau A(U)+(1-\tau)A(\bU))d\tau\right|\,|A(U)-A(\bU)|\le C\,|A(U)-A(\bU)|
\end{aligned}
\end{equation}
where $C:=\displaystyle\sup_{U\in B_R,\,\bU\in\Gamma_{M,\delta}}\left|\int_0^1 (\nabla_V(A^{-1})(\tau A(U)+(1-\tau)A(\bU))d\tau \right|<\infty$. Hence,
\begin{equation}
I(F,v,\theta|\bF,\bv,\btheta)\ge \frac{K_2}{C}|U-\bU|^2\qquad \text{for }U\in B_R.
\end{equation}
Thus the proof of part (i) is complete.

{\bf Step 2.}
Now, using $\eta=-\frac{\partial\psi}{\partial\theta}$ and the expansion
\begin{equation}\label{S4etaexp}
\eta(F,\theta|\bF,\btheta)=\eta(F,\theta)-\eta(\bF,\btheta)-\eta_F(\bF,\btheta):(F-\bF)-\eta_\theta(\bF,\btheta)(\theta-\btheta)\,,
\end{equation}
we have 
$$
\big | \eta(F,\theta|\bF,\btheta) \big |
\le
| \eta( F, \theta) | + C_1 |F| + C_2 \theta + C_3
$$
for all $(\bF, \btheta, \bv) \in \Gamma_{M, \delta}$ with $C_i$ constants depending only on $\Gamma_{M, \delta}$.
It follows by~\eqref{littlea3} that
\begin{equation}\label{S4etaupbd}
\limsup_{|F|^p+\theta^q\to\infty}\frac{ |\eta(F,\theta|\bF,\btheta) |  }{|F|^p+\theta^q}
=\limsup_{|F|^p+\theta^q+ |v|^2 \to\infty}\frac{ |\eta(F,\theta|\bF,\btheta) |  }{|F|^p+\theta^q + |v|^2}=0.
\end{equation}
Using~\eqref{S4etaexp} and~\eqref{S4etaupbd} and selecting $R>r(M)+1$ sufficiently large, there exists $C>0$ such that
\begin{equation}
\label{renbo2}
\begin{aligned}
|\eta(F,\theta|\bF,\btheta)|
&\le
\begin{cases}
C(|F|^p+\theta^q+|v|^2) +\bar C & \text{for }|F|^p+\theta^q+|v|^2\ge R\\
C(|F-\bF|^2+|\theta-\btheta|^2+|v-\bv|^2) &  \text{for }|F|^p+\theta^q+|v|^2< R
\end{cases}
\; ,
\end{aligned}
\end{equation}
for all $(F, v, \theta)$ and for $(\bF,\bv,\btheta)\in\Gamma_{M,\delta}\subset B_r$. 
Note that the same $R$ can be used in both bounds \eqref{renbo1} and \eqref{renbo2} by adjusting the constants in
these bounds. Hence, we conclude
$$
|\eta(F,\theta|\bF,\btheta)| \le C I(F,v,\theta|\bF,\bv,\btheta) \, .
$$

{\bf Step 3.}
Similarly, using $\Sigma=\displaystyle\frac{\partial\psi}{\partial F}$ and the expansion
\begin{equation}\label{S4sigmaexp}
\Sigma(F,\theta|\bF,\btheta)=\Sigma(F,\theta)-\Sigma(\bF,\btheta)-\psi_F(\bF,\btheta):(F-\bF)-\psi_\theta(\bF,\btheta)(\theta-\btheta)\,,
\end{equation}
it follows by~\eqref{littlea2}
\begin{equation}\label{S4sigmaupbd}
\limsup_{|F|^p+\theta^q\to\infty}\frac{\Sigma(F,\theta|\bF,\btheta)}{|F|^p+\theta^q}=0.
\end{equation}
and we proceed as in Step $2$ to prove part (iii).

{\bf Step 4.}  To show (iv), recall that $(\bF,\bv,\btheta)\in\Gamma_{M,\delta}\subset B_r$. We then have
$$
\begin{aligned}
|F - \bF|^p + |\theta - \btheta|^q + |v - \bv|^2 \le ( | F|  + M)^p + ( \theta + M )^q + ( |v| + M )^2
\end{aligned}
$$
Since
$$
\limsup_{|F|^p+\theta^q+|v|^2\to\infty} \frac{ ( | F|  + M)^p + ( \theta + M )^q + ( |v| + M )^2 }{ |F|^p+\theta^q+|v|^2} = 1
$$
we may select $R$ such that 
\begin{equation}
\label{bound12}
|F - \bF|^p + |\theta - \btheta|^q + |v - \bv|^2 \le 2 \big ( |F|^p+\theta^q+|v|^2 \big ) \quad \text{for } \; |F|^p+\theta^q+|v|^2 > R
\end{equation}
Equation \eqref{normlike} follows by combining \eqref{bound5} with \eqref{bound12} and the proof of Step 1.
\end{proof}

Now, we state and prove a convergence result recovering the smooth solution $\bU$ of thermoelastic nonconductors of heat from solutions of~\eqref{thermov2} as $\mu,\,k\to0+$. From here  on, we denote the solution to the system of thermoviscoelasticity by $U^{\mu,k}=(F,v,\theta)^T$ to give an emphasis on the dependence of the solution on the functions $\mu = \mu(F,\theta)$ and $k=k(F,\theta)$.
We note that the dependence of the solution $U^{\mu,k}=(F,v,\theta)^T$ on $\mu$ and $k$ will be specified on the state vector $U^{\mu,k}$ and for convenience we drop it from the components $F$, $v$ and $\theta$. As before, to avoid technicalities, we work for now in the spatially periodic case with domain $Q_T={\To}^d\times[0,T)$ for $T\in[0,\infty)$.

\begin{theorem}\label{thmconvnonconductors}
Let $U^{\mu,k}$ be a strong solution of the system of thermoviscoelasticity~\eqref{thermov2} 
satisfying the constitutive relations~\eqref{thermovcr}, \eqref{thermovcr2} with
$$Z=\mu(F,\theta)\nabla v,\qquad Q=k(F,\theta)\nabla \theta \, , $$
defined on a maximal domain of existence $Q_{T^*}$,
and let $\bar U$ be a smooth solution to the system of thermoelastic nonconductors of heat \eqref{athermoel} 
defined on $\overline{Q_T}$, $0<T<T^*$ and emanating from initial data $U_0^{\mu,k}$ and $\bU_0$, respectively. 
Assume that Hypotheses~\eqref{S4H},~\eqref{littlea11},~\eqref{littlea2} and~\eqref{littlea3} hold true and suppose that $\bar U\in \Gamma_{M,\delta}$ for some constants $M>0$ and $\delta>0$. Then there exists a constant $C=C(T)$ such that for $t\in(0,T)$,
\begin{equation}\label{Imk}
\int I (U^{\mu,k}(t)|\bU(t))dx\le C\left( \int I(U^{\mu,k}_0|\bU_0)dx+\int_0^T\int  \mu\frac{\theta(s)}{\btheta(s)}|\nabla\bv(s)|^2+ k\frac{|\nabla\btheta(s)|^2}{\btheta(s)}dx\,ds  
\right)\,.
\end{equation}
Moreover, if $ \int_{\To^d} \theta(t)\,dx$ is uniformly bounded for $t\in[0,T)$ and
\begin{equation}\label{m-klimit} 
|\mu(F,\theta)|\le\mu_0,\qquad |k(F,\theta)|\le k_0\,\theta\,,
\tag{\text{H}$_{\mu,k}$}
\end{equation}
 then for every data satisfying $\displaystyle \lim_{\substack{\mu_0\to0+\\ k_o\to0+}}\int I(U^{\mu,k}_0|\bU_0)dx= 0$, it follows
\begin{equation}\label{S4convergence}
\sup_{t\in(0,T)}\int I (U^{\mu,k}(t)|\bU(t))dx\to 0\qquad\text{as  }\mu_0,k_0\to0+\,.
\end{equation}
\end{theorem}
\begin{proof}
Integrating  the relative entropy identity~\eqref{relenthermo0} and combining with estimates~\eqref{etaboundthermov} and~~\eqref{sigmaboundthermov} of Lemma~\ref{LemmaI}, we get
\begin{equation}
\begin{aligned}
\frac{d}{dt}&\int I(F,v,\theta|\bF,\bv,\btheta)\,dx+\int\btheta\,\left(\mu\frac{|\nabla v|^2}{\theta}+\,k\,\frac{|\nabla\theta|^2}{\theta^2}\right)\,dx\le\\
&\le \int|\btheta_t||\eta(F,\theta|\bF,\btheta)|+|\bF_t|\,|\Sigma(F,\theta|\bF,\btheta)|\,dx+\int\mu|\nabla \bv|\,|\nabla v|+k\frac{|\nabla\btheta|\,|\nabla\theta|}{\theta}dx\\
&\le C\int|I(F,v,\theta|\bF,\bv,\btheta)|dx+\left(\int\btheta\mu\frac{|\nabla v|^2}{\theta}dx\right)^{1/2}\left(\int\mu\frac{\theta}{\btheta}|\nabla\bv|^2  dx\right)^{1/2}\\
&\qquad+\left(\int\btheta\,k\frac{|\nabla \theta|^2}{\theta^2}\right)^{1/2}\left(\int k\frac{|\nabla\btheta|^2}{\btheta}dx\right)^{1/2}\\
&\le C\int|I(F,v,\theta|\bF,\bv,\btheta)|dx+\tfrac{1}{2}\int\btheta\,\left(\mu\frac{|\nabla v|^2}{\theta}+\,k\,\frac{|\nabla\theta|^2}{\theta^2}\right)\,dx
+\tfrac{1}{2}\int  \mu\frac{\theta}{\btheta}|\nabla\bv|^2+ k\btheta\frac{|\nabla\btheta|^2}{\btheta^2}dx
\end{aligned}
\end{equation} 
for some constant $C=C(|\btheta_t|,|\bF_t|)$. Hence, Gronwall's inequality gives
\begin{equation}\label{Iestimatemk}
\begin{aligned}
\int I (U^{\mu,k}(t)|\bU(t))dx&\le e^{Ct} \int \psi (F_0 , \theta_0 | \bF_0 , \btheta_0 ) +\tfrac{1}{2} |v_0  - \bv_0 |^2 + ( \eta (F_0 ,\theta_0) - \eta (\bF_0 , \btheta_0 ))(\theta_0  - \btheta_0 )dx\\
&+ \tfrac{1}{2} \int_0^t e^{C(t-s)}\left[\int \mu\frac{\theta(s)}{\btheta(s)}|\nabla\bv(s)|^2+ k\frac{|\nabla\btheta(s)|^2}{\btheta(s)}dx\right] ds 
\end{aligned}
\end{equation} 
and~\eqref{Imk} follows. Last, if $\|\theta(t)\|_{L^1({\To}^d)}\le K$ for all $t\in[0,T]$ and~\eqref{m-klimit} is satisfied then, taking the limit in~\eqref{Iestimatemk} as $\mu_0\to0+$ and $k_0\to0+$,~\eqref{S4convergence} follows. 
\end{proof}

\begin{remark} \rm
The uniform estimate $\int\theta(t)\,dx\le K$ is expected for solutions of system~\eqref{thermov}. Uniform energy estimates 
are obtained by integrating the energy equation \eqref{thermov}$_3$ which, for periodic boundary conditions and
for $r = 0$, $f=0$, gives
$$
\int_{\To^d} \frac{1}{2} |v|^2 + e (F, \theta)  dx \le C\;.
$$
The $L^1$ estimate on the temperature is then a consequence of \eqref{littlea11}.
Hence, Hypothesis~\eqref{m-klimit} suffices to provide the convergence \eqref{S4convergence}.  Hypothesis ~\eqref{m-klimit} may
be weakened  if the growth assumption \eqref{littlea11} provides higher integrability for the temperature $\theta$.
\end{remark}

\subsection{Uniqueness of smooth solutions in the class of entropic measure-valued solutions}\label{S4.uni}
In this section, we consider the system of  adiabatic thermoelasticity,
\begin{equation}
\label{thermoad}
\begin{aligned}
F_t &= \nabla v
\\
v_t  &= \div \Sigma 
\\
\del_t ( \tfrac{1}{2} |v|^2 + e ) &= \div ( v \cdot  \Sigma)  + r
\end{aligned}
\end{equation}
subject to the entropy inequality for weak solutions
\begin{equation}
\label{thermocd}
\del_t \eta \ge \frac{r}{\theta}
\end{equation}
under the constitutive theory~\eqref{thermovcr}
\begin{equation*}
\begin{aligned}
\Sigma =  \; \frac{\del \psi}{\del F} (F, \theta) \,, \quad \eta = - \frac{\del \psi}{\del \theta} (F, \theta) \,, \quad
e = \psi + \theta \eta 
\end{aligned}
\end{equation*}
for some Helmhotz free energy function $\psi = \psi (F, \theta)$. The system \eqref{thermoad} is satisfied by the class of materials 
termed thermoelastic non-conductors of heat, a particular subclass of which is the gas dynamics equations in Lagrangean coordinates.

In Section \ref{S4.1} we showed that  a weak solution $U = (F, v, \theta)$ of \eqref{thermoad} satisfying the entropy inequality \eqref{thermocd}
and $\bU = (\bF, \bv, \btheta)$ and a strong solution to~\eqref{thermoad}, which necessarily satisfies the entropy identity
\begin{equation}
\label{eq10}
\del_t \eta (\bF, \btheta ) = \frac{\br}{\btheta} \, ,
\end{equation}
can be compared via the relative entropy inequality
\begin{equation}
\label{relenthermoad}
\begin{aligned}
&\del_t \Big ( I(F,v,\theta|\bF,\bv,\btheta) \Big )
-  \div \Big ( (v - \bv) \cdot ( \Sigma - \bSigma )    \Big ) \le  - \btheta_t \,  \eta(F, \theta | \bF, \btheta) + \bF_t : \Sigma (F, \theta | \bF, \btheta)\;.
\end{aligned}
\end{equation}
In this section we establish an analog of \eqref{relenthermoad} valid for entropic measure-valued solutions and eventually establish the
uniqueness of \emph{classical} solutions in the class of \emph{dissipative measure--valued} solutions for the equations of adiabatic thermoelasticity  \eqref{thermoad}. 
This theory will be the analog of the general theory in Section~\ref{S2.2.2} when specified
to \eqref{thermoad}. However, there are some important differences with the general case that have to do with the
treatment of concentrations, and Theorem~\ref{thmweakstrong} does not apply directly and has to be adapted.

\subsubsection{Entropic-mv solutions for adiabatic thermoelasticity}\label{S4.3.1}

An entropic measure-valued (mv) solution for \eqref{thermoad} consists of a Young measure $\bonu =({\bonu}_{x,t})_{\{(x,t)\in \bar{Q}_T\}}$
a non-negative Radon measure $\bomu \in\mathcal{M}^+(Q_T)$ describing concentrations and functions $(F, v, \theta)$,
$$
F =  \langle \bonu_{(x,t)} , \lambda_F \rangle  \, , \quad v =  \langle \bonu_{(x,t)} , \lambda_v  \rangle \, , \quad
\theta =  \langle \bonu_{(x,t)}  ,  \lambda_\theta \rangle
$$
with $F \in L^\infty ( L^p )$, $v \in L^\infty ( L^2 )$, $\theta \in L^\infty ( L^q )$ that satisfies in the sense of distributions the averaged equations
\begin{equation}
\label{defemv}
\begin{aligned}
\del_t \langle \bonu, \lambda_F \rangle &= \nabla \langle  \bonu,  \lambda_v \rangle 
\\
\del_t  \langle  \bonu,  \lambda_v \rangle  &= \div \langle \bonu , \Sigma (\lambda_F , \lambda_\theta ) \rangle
\\
\del_t \Big ( \big\langle \bonu , \tfrac{1}{2} | \lambda_v |^2 + e (\lambda_F , \lambda_\theta)  \big\rangle  + \boldsymbol{\mu}
 \Big ) 
&= \div \langle \bonu , \lambda_v \cdot \Sigma (\lambda_F , \lambda_\theta ) \rangle 
+ \langle \bonu, r \rangle
\end{aligned}
\end{equation}
and the averaged form of the entropy production equation
\begin{equation}
\label{defcdmv}
\del_t  \langle \bonu , \eta (\lambda_F , \lambda_\theta ) \ge \langle \bonu , \frac{r}{\lambda_\theta} \rangle\,.
\end{equation}

Some justification of the above definition is needed: Typically mv-solutions of \eqref{thermoad} will appear as limits of some
approximating problem like the system of thermoviscoelasticity. The natural  available bounds 
are provided by the energy conservation equation (given some mild hypothesis on
the energy radiation term $r$ which for simplicity is assumed here as a given bounded function). It leads to the bound
\begin{equation}
\label{avbound}
\int_{\To^d} e(F^\eps , \theta^\eps  ) + \tfrac{1}{2} |v^\eps|^2 dx \le C \, ,
\end{equation}
where $\eps$ stands for the approximation parameter. Under the growth assumption \eqref{littlea11}, the uniform bound \eqref{avbound} in turn yields that
$F^\eps$ is uniformly bounded in $L^p$, $v^\eps$ in $L^2$ and $\theta^\eps$ in $L^q$, with $p , \, q > 1$. The family generates  (along subsequences)
a Young measure $\nu = \nu_{(t,x)}$ that characterizes the weak limits  ({\it e.g.}~\cite{tartar79,ball88}). The action of the Young measure is well defined 
for functions that grow slower than the energy and characterizes their weak limits:
\begin{equation}
\label{defclasym}
\begin{aligned}
&\qquad \qquad \text{wk}-\lim  f( F^\eps , v^\eps , \theta^\eps) = \langle \bonu , f(\lambda_F , \lambda_v , \lambda_\theta ) \rangle  \quad
\\
 \forall  \;  &\mbox{continuous $f$ such that} \;   \lim_{|\lambda_F|^p + |\lambda_v|^2+ (\lambda_\theta)^q \to \infty} \frac { | f(\lambda_F , \lambda_v , \lambda_\theta ) | }{ |\lambda_F|^p + |\lambda_v|^2+ (\lambda_\theta)^q} = 0\,.
 \end{aligned}
\end{equation}

We impose the growth assumptions \eqref{littlea2} on the stress, \eqref{littlea3} on the entropy and in addition the growth restriction
\begin{equation}
\label{littlea4}
    \lim_{ |F|^p + \theta^q  + |v^2| \to \infty}  \frac{ | v \cdot \Sigma(F,\theta)| }{ |F|^p + \theta^q  + |v|^2  }  = 0  \, ,
    \tag{a$_4$}
\end{equation}
on the power of the stresses. With these restrictions all the actions of Young measures appearing in \eqref{defemv} and \eqref{defcdmv} are well defined,
except for that on the total energy.

Classical Young measures do not suffice to characterize the weak limit of $e(F^\eps, \theta^\eps) + \tfrac{1}{2} |v^\eps|^2$ 
due to the appearance of concentrations. This problem is undertaken by DiPerna and Majda \cite{dm87} and leads to the introduction of  generalized
Young measures with concentrations; a general representation theorem is obtained by using the recession function, see Alibert-Bouchitt\'{e}  \cite{ab97}.
Let $\Omega$ be an open subset of $\R^n$ and $\{ u_n\}$ a bounded sequence in $L^1 (\Omega ; \R^n)$. The goal in \cite{dm87,ab97} is to represent
weak limits of the form
$$
\lim_{n \to \infty}  \int_\Omega \varphi (y) g ( u_n (y) ) \, dy
$$
for $\varphi \in C_0 (\Omega)$ and for test functions $g$ of the form
$$
g(\xi) = \bar g (\xi) ( 1 + |\xi|) \quad \mbox{for some $\bar g \in BC(\R^n)$},
$$
where $BC(\R^n)$ denotes the bounded continuous functions on $\R^n$. 
We list below the representation result and refer to \cite{dm87,ab97}  for details and to  \cite{bds11} for a quick presentation
that can serve as an introduction to the subject. Define
\begin{equation}
\label{hyprec}
\begin{aligned}
\cF_0 &= \{ h \in BC (\R^n) : \quad h^\infty (\xi) = \lim_{s \to \infty} h(s \xi) \; \mbox{exists and is continuous on $\cS^{n-1}$} \; \}
\\[2pt]
\cF_1 &= \{ \; \; g \in C(\R^n) : \quad g(\xi ) = h(\xi ) ( 1 + |\xi|) \; \mbox{for  $h \in \cF_0$} \; \}
\end{aligned}
 \tag{H$_{rec}$}
\end{equation}
Given $X$ a locally, compact Hausdorff space, let $\cM (X)$ denote the Radon measures on $X$, $\cM_+ (X)$ the positive Radon measures, and 
$\Prob (X)$ the probability measures. Given a Radon measure $\lambda$ on $\Omega$ set $\cP (\lambda ; X) = L^\infty_w (d\lambda ; \Prob (X) )$
be the parametrized families of probability measures $( \bonu_y)_{y \in \Omega}$  acting on $X$ which are weakly measurable on the parameter $y \in \Omega$.
When $\lambda$ is the Lebesgue measure we denote $\cP (\lambda ; X) = \cP (\Omega ; X)$.

\begin{theorem}[{\sc DiPerna and Majda \cite{dm87}, Alibert and Bouchitt\'{e} \cite{ab97}}]\label{abthm}
Let $\{ u_n \}$ be bounded in $L^1 (\Omega ; \R^n)$. There exists a subsequence $\{ u_{n_k} \}$, a nonnegative Radon measure 
$\mu \in \cM_+ (\Omega)$ and parametrized families of probability measures $\bonu \in \cP (\Omega ; \R^n)$ and
$\bonu^\infty \in \cP( \lambda ; \cS^{n-1} )$ such that
$$
g ( u_{n_k} ) \rightharpoonup \langle \bonu , g \rangle + \langle \bonu^\infty , g^\infty \rangle \, \mu \quad \mbox{ weak-$\ast$ in $\cM (\Omega)$}
$$
for $g \in \cF_1$.
\end{theorem}

The theorem is applied to represent the weak limits $\text{wk}-\lim  f( F^\eps , v^\eps , \theta^\eps)$, where  the family $(F^\eps, v^\eps, \theta^\eps)$ 
satisfies the uniform bound \eqref{avbound}, $e(F,\theta)$ grows according to  \eqref{littlea11}, and $f$ is  a
continuous test function with growth
$$
|f (F, v, \theta) | \le  C (1 + |F|^p + |v|^2 + \theta^q)  \qquad F \in \R^{d \times d} \, , \; v \in \R^d \, , \; \theta \in \R^+ \, .
$$
To this end apply the change of variables $(A, b, c) = ( |F|^{p-1}  F , |v| v ,  \theta^q ) \in \R^{d \times d} \times  \R^d \times  \R^+ $ to the test function $f$ and define 
$$
f(F,  v , \theta ) = :  g ( |F|^{p-1}  F , |v| v  , \theta^q )
$$
The test function $g(A, b, c)$ satisfies the growth hypothesis $|g(A, b, c)| \le C (1 + |A| + |b| + c))$ and Theorem \ref{abthm} is used to represent the weak-$\ast$ limits
for test functions $g$  satisfying  \eqref{hyprec}:
$$
g^\infty ( A, b, c) = \lim_{s \to \infty} \frac{g( s A, sb, sc)}{1 + s (|A| + |b| + c)}  \quad \mbox{exists and is continuous for $(A, b, c) \in \cS^{d^2 + d} \cap \{c > 0\}$ } \, .
$$
There are probability measures $N_{(t,x)} \in \cP ( \barQT ; \R^{d^2 + d+1})$, $N^\infty_{(t,x)} \in \cP ( \barQT ; \cS^{d^2 + d})$ and a positive
Borel measure $M \in \cM^+(\barQT )$ such that
$$
g(A_n , b_n , c_n ) \rightharpoonup \langle N , g(\lambda_A , \lambda_b , \lambda_c )  \rangle 
+ \big\langle N^\infty , g^\infty (\lambda_A , \lambda_b , \lambda_c )  \big\rangle  M \, .
$$
This in turn implies
\begin{equation}
\label{wkrepym}
f(F_n , v_n , \theta_n ) \rightharpoonup  \langle \bonu , f(\lambda_F , \lambda_v , \lambda_\theta )  \rangle 
+ \big\langle \bonu^\infty , f^\infty (\lambda_F , \lambda_v , \lambda_\theta )  \big\rangle  M
\end{equation}
where $\bonu_{(t,x)}$ and $\bonu^\infty_{(t,x)}$ are defined via
\begin{equation}
\label{defgenym}
\begin{aligned}
\langle \bonu , f( \lambda_F, \lambda_v , \lambda_\theta) \rangle 
&= \langle N , g(  |\lambda_F|^{p-1}\lambda_F  , |\lambda_v| \lambda_v , ( \lambda_\theta )^q )  \rangle 
\\
\langle \bonu^\infty , f^\infty ( \lambda_F, \lambda_v , \lambda_\theta) \rangle 
&= \langle N , g^\infty (  |\lambda_F|^{p-1}\lambda_F  , |\lambda_v| \lambda_v , ( \lambda_\theta )^q )  \rangle 
\end{aligned}
\end{equation}

Formulas \eqref{wkrepym}, \eqref{defgenym} are applied to represent the weak limit of the total energy 
 $e(F,\theta) + \tfrac{1}{2} |v^2|$.  It is necessary to assume that the recession function
\begin{equation}
\label{littlea5}
\begin{aligned}
&\big ( e(F, \theta) + \frac{1}{2} | v|^2 \big )^\infty  := 
\lim_{s \to \infty} 
\frac{ e \left (  s^{\tfrac{1}{p}} F,  s^{\tfrac{1}{q}}  \theta \right )  + \frac{1}{2} s |v|^2 }{ 1 +   s ( |F|^p  + \theta^q + |v|^2 )} 
\\[2pt]
&\qquad \mbox{ exists and is continuous for $( |F|^{p-1}  F , |v| v ,  \theta^q )  \in  \cS^{d^2 + d} \cap \{c > 0\}$}
\end{aligned}
 \tag{a$_5$}
\end{equation}
 and the theorem gives that along a subsequence
\begin{equation}
\label{repformula}
\mbox{wk-$\ast$-lim} \Big ( e(F^\eps , \theta^\eps) + \frac{1}{2} |v^\eps|^2 \Big ) 
= \langle \bonu , e(\lambda_F, \lambda_\theta) + \frac{1}{2} |\lambda_v|^2  \rangle 
+ \big\langle \bonu^\infty ,  \big ( e(\lambda_F, \lambda_\theta) +  \frac{1}{2} |\lambda_v|^2 \big )^\infty \big\rangle \, M
\end{equation}
Due to the hypothesis \eqref{littlea11}, we have $\big ( e(\lambda_F, \lambda_\theta) +  \frac{1}{2} |\lambda_v|^2 \big )^\infty > 0$ and thus
$$
\boldsymbol{\mu} := \big\langle \bonu^\infty ,  \big ( e(\lambda_F, \lambda_\theta) +  \frac{1}{2} |\lambda_v|^2 \big )^\infty \big\rangle \, M  \, \in  \, \cM_+ (\barQT) \, .
$$
The representation formula \eqref{repformula} justifies the nature of the definition of entropic-mv solutions for the equations of
adiabatic thermoelasticity, as it pertains to the format of the energy conservation equation \eqref{defemv}$_3$.

\subsubsection{The averaged relative entropy inequality for adiabatic thermoelasticity}
The objective is to  compare an entropic-mv solution with the strong solution $(\bF, \bv, \btheta)$. Motivated by 
 \eqref{defrelenmd} and \eqref{dfH}, we define the averaged relative entropy
\begin{equation}\label{S4dfH}
\begin{aligned}
\mathcal{H}(\bonu ,U,\bU)=
- \langle \bonu ,\eta\rangle  +\bar{\eta}
-\frac{\bSigma}{\btheta} : \langle \bonu ,  \lambda_F -\bF \rangle
-\frac{\bv}{\btheta} \cdot \langle \bonu ,  \lambda_v -  \bv \rangle
\\
+ \frac{1}{\btheta}   \Big (  \big\langle \bonu , e(\lambda_F, \lambda_\theta) + \tfrac{1}{2}|\lambda_v |^2    - e (\bF, \btheta) - \tfrac{1}{2}|\bv|^2) \big\rangle  +\boldsymbol{\mu} \Big )
\end{aligned}
\end{equation}
The reader should note that the formula, which  now involves the concentration measure $\mu$,  is easily recast  
in the form
\begin{equation}\label{S4dfH2}
\begin{aligned}
\mathcal{H}(\bonu ,U,\bU) &=  \frac{1}{\btheta} \Big ( \big\langle \bonu , I ( \lambda_F , \lambda_v , \lambda_\theta | \bF , \bv , \btheta ) \big\rangle + \boldsymbol{\mu} \Big )
\\[3pt]
\mbox{where} \qquad I ( \lambda_U | \bU ) &= I ( \lambda_F , \lambda_v , \lambda_\theta | \bF , \bv , \btheta ) 
\\
&:= 
 \psi ( \lambda_F , \lambda_\theta | \bF , \btheta)  +  (\eta (\lambda_F, \lambda_\theta )  - \eta ( \bF, \btheta) )(\lambda_\theta-\btheta)  
+ \tfrac{1}{2} |\lambda_v-\bv|^2\rangle  \end{aligned}
\end{equation}
and  $\psi ( \lambda_F , \lambda_\theta | \bF , \btheta)$ is given by \eqref{defrelpsi}.

Using equations \eqref{defemv}, \eqref{defcdmv} for $(\lambda_F,\lambda_v,\lambda_\theta)$ and~\eqref{thermoad},~ \eqref{eq10} for $(\bF, \bv, \btheta)$, we obtain
$$
\begin{aligned}
-\btheta &\del_t  \big\langle \bonu, \eta (\lambda_F, \lambda_\theta )  - \eta ( \bF, \btheta) \big\rangle - \bSigma : \del_t \langle \bonu, \lambda_F - \bF \rangle 
- \bv \cdot \del_t  \langle \bonu, \lambda_v - \bv \rangle 
\\
&\quad + \del_t \Big ( \big\langle \bonu , e(\lambda_F, \lambda_\theta) + \tfrac{1}{2}|\lambda_v |^2    - e (\bF, \btheta) - \tfrac{1}{2}|\bv|^2) \big\rangle  + \boldsymbol{\mu} \Big )
\\[2pt]
&\le -\bSigma : \nabla \langle \bonu, \lambda_v - \bv \rangle  
- \bv \cdot \div  \big\langle \bonu , \Sigma (\lambda_F , \lambda_\theta ) - \Sigma (\bF, \btheta) \big\rangle 
\\
&\quad + \div \Big\langle \bonu ,  \lambda_v \cdot \Sigma (\lambda_F , \lambda_\theta ) -  \bv \cdot \Sigma (\bF, \btheta) \Big\rangle 
- \btheta \Big\langle \bonu , \frac{r}{\lambda_\theta} - \frac{\br}{\btheta} \Big\rangle  + \langle \bonu, r - \br \rangle \;.
\end{aligned}
$$
This inequality is understood in the sense of distributions, meaning that one multiplies by a test function $\varphi(x,t) > 0$ and integrates by parts;
this is possible since $(\bF, \bv, \btheta) \in W^{1, \infty}$ and since $\boldsymbol{\mu}$ is a measure, exploiting the fact that there is no multiplier in the term coming
from the energy. Next we integrate by parts, exploiting again the fact that $(\bF, \bv, \btheta)$ is a strong solution,  and take account of \eqref{S4dfH2}, 
to derive the  inequality 
\begin{align}
\del_t \Big [ \big\langle \bonu  &, I ( \lambda_U | \bU )  \big\rangle + \boldsymbol{\mu} \Big ]
-   \div \big\langle \bonu , (\lambda_v - \bv ) \cdot ( \Sigma (\lambda_F , \lambda_\theta ) -   \Sigma (\bF, \btheta) ) \big\rangle
\nonumber
\\[2pt]
&\le - (\del_t \btheta) \big\langle \bonu, \eta (\lambda_F, \lambda_\theta )  - \eta ( \bF, \btheta) \big\rangle 
- \del_t \bSigma : \langle \bonu, \lambda_F - \bF \rangle  -  \del_t \bv \cdot \nabla \langle \bonu, \lambda_v - \bv \rangle  
\nonumber
\\
&\quad + \langle \bonu, \lambda_v - \bv \rangle  \cdot \div \bSigma + \langle \bonu , \Sigma (\lambda_F , \lambda_\theta ) - \Sigma (\bF, \btheta) \rangle  : \nabla \bv
- \btheta \Big\langle \bonu , \frac{r}{\lambda_\theta} - \frac{\br}{\btheta} \Big\rangle + \langle \bonu, r - \br \rangle 
\nonumber
\\[2pt]
&\le - \btheta_t \, \Big\langle \bonu, \eta (\lambda_F, \lambda_\theta | \bF , \btheta ) \Big\rangle 
+ \bF_t : \Big\langle \bonu, \Sigma (\lambda_F, \lambda_\theta | \bF , \btheta ) \Big\rangle
 + \Big\langle \bonu,  \big ( \frac{r}{\lambda_\theta} - \frac{\br}{\btheta} \big ) (\lambda_\theta - \btheta )  \Big\rangle \, .
\label{avgrelen}
\end{align}
The last inequality follows from  \eqref{defrelpsi}, \eqref{maxwellrel},~ \eqref{eq10}. The derivation is analogous to the one  leading to the derivation of \eqref{relenthermov} and is omitted. The final identity \eqref{avgrelen} is the averaged version of the relative entropy inequality comparing 
 the entropic measure valued solution \eqref{defemv}, \eqref{defcdmv} to a Lipschitz solution $\bU$. 
 The reader should note how the concentration measure $\mu$ enters in both \eqref{defcdmv} and \eqref{avgrelen}.

Next, we establish the recovery of \emph{classical} solutions from \emph{entropic measure--valued} solutions as defined above in the framework of adiabatic thermoelasticity. The theorem should be contrasted to the theory for general hyperbolic systems in Section~\ref{S2.2.2}. 
Interestingly the example deviates from the workings of
the general theory with respect to the function of the concentration measure. For simplicity, from here and on, we set $r=\br=0$ and denote by $(U,\boldsymbol{\nu}, \boldsymbol{\mu})$ the entropic mv solution to the system of adiabatic thermoelasticity with associated Young measure $\boldsymbol{\nu}$ and concentration measure $\boldsymbol{\mu}\in\cM_+$ as constructed in Section~\ref{S4.3.1}.

The result is the following: 

\begin{theorem}
\label{thermoelweakstrong}
Suppose that \eqref{S4H}, the growth assumptions \eqref{littlea11}, \eqref{littlea2}, \eqref{littlea3}, \eqref{littlea4} and hypothesis \eqref{littlea5} hold true. Let
$(U,\boldsymbol{\nu}, \boldsymbol{\mu})$ be an entropic measure--valued solution to~\eqref{defemv}, \eqref{defcdmv} subject to the constitutive assumptions~\eqref{thermovcr}
with $r = 0$. Let $\bar{U}\in W^{1,\infty}(\overline{Q_T})$ be a strong solution to~\eqref{thermoad},~\eqref{eq10} with $\br=0$ such that 
$\bU\in\Gamma_{M,\delta}$, $\forall(x,t)\in Q_T$ for some $M>0$ and $\delta>0$.. Then, if the  initial data 
satisfy $\boldsymbol{\mu}_0=0$ and ${\boldsymbol{\nu}_0}_x=\delta_{\bU_0}(x)$, it holds $\boldsymbol{\nu}=\delta_{\bar{U}}$ and $U=\bar{U}$ almost everywhere on $Q_T$.
\end{theorem}
\begin{proof}
We use the average quantity \eqref{S4dfH2} which satisfies the inequality~\eqref{avgrelen} in distributions.
Let $\{\xi_n\}$ be a sequence of monotone decreasing functions, with $\xi_n\ge 0$ $\forall n\in\mathbb{N}$, converging to the Lipschitz function $\xi$ given by~\eqref{S2.2.2xi} as $n\to\infty$. We multiply the inequality~\eqref{avgrelen} by $\varphi(x,\tau):= \xi_n(\tau)\in C^1_0([0,T])$, and get the integral relation
\begin{equation}
\begin{aligned}
\iint\frac{d\xi_n}{d\tau} \Big [ \big\langle &\bonu  , I ( \lambda_U | \bU )  \big\rangle\,dxd\tau+ \boldsymbol{\mu}(dxd\tau)\Big]\ge\\
\ge&-\iint\xi_n(\tau)
\left[- \btheta_t \, \Big\langle \bonu, \eta (\lambda_F, \lambda_\theta | \bF , \btheta ) \Big\rangle 
+ \bF_t : \Big\langle \bonu, \Sigma (\lambda_F, \lambda_\theta | \bF , \btheta ) \Big\rangle
\right]  ]dxd\tau\\
&-\int\xi_n(0) \Big [ \big\langle \bonu  , I ( \lambda_{U_0} | \bU_0 )  \big\rangle dx+ \boldsymbol{\mu}_0(dx)\Big]\;,
\end{aligned}
\end{equation}
for all $n\in\mathbb{N}$. Passing to the limit as $n\to\infty$ and then  $\varepsilon\to0+$, we arrive at
\begin{equation}
\int \big\langle \bonu  , I ( \lambda_{U(t)} | \bU(t) )  \big\rangle \,dx\le C\int_0^t\int\big\langle \bonu  , I ( \lambda_{U(\tau)} | \bU(\tau) )  \big\rangle \,dx\,d\tau+
\int\Big [ \big\langle \bonu  , I ( \lambda_{U_0} | \bU_0 )  \big\rangle dx+ \boldsymbol{\mu}_0(dx)\Big]\;.
\end{equation}
Indeed, the above estimate holds true because $\boldsymbol{\mu}\ge 0$ and
 $|\eta (\lambda_F, \lambda_\theta | \bF , \btheta )|\le C I(\lambda_U|\bU) $ and $|\Sigma (\lambda_F, \lambda_\theta | \bF , \btheta )| \le C  I(\lambda_U|\bU)$.
by Lemma~\ref{LemmaI} since $\bU\in\Gamma_{M,\delta}$, $\forall(x,t)\in Q_T$ and~\eqref{littlea11},~\eqref{littlea2} and~\eqref{littlea3}  are satisfied under the constitutive theory~\eqref{thermovcr}. For data with $\boldsymbol{\mu}_0=0$, applying Gronwall's inequality we conclude
\begin{equation}
\int \big\langle \bonu  , I ( \lambda_{U(t)} | \bU(t) )  \big\rangle \,dx\le e^{Ct}\int  \big\langle \bonu  , I ( \lambda_{U_0} | \bU_0 )  \big\rangle     \,dx\,.
\end{equation} 
The proof follows easily by~\eqref{normlike}.

\end{proof}
%
\appendix
\section{Useful Lemmas}\label{s-A}
Here, we adapt an idea from \cite{dst12} which leads to useful bounds for the relative entropy and the relative stress when
$\bu$ is restricted to take values in $B_M$ under the growth restrictions~\eqref{hypetap}, ~\eqref{hypfgrowth} and~\eqref{hypAgrowth} on the constitutive functions $\eta(u)$, $F_\alpha(u)$ and $A(u)$. These bounds have been used to establish the uniqueness results of Section~\ref{S2.2.2} and also to interpret the relative entropy as a ``distance formula".

Using  \eqref{formdefin}, \eqref{appform3}
the relative entropy is expressed in two forms :
\begin{equation}
\label{workrelen}
\begin{aligned}
\eta ( u | \bu ) &\doteq \eta (u) - \eta (\bu) - G(\bu) \cdot (A(u) - A(\bu))
\\
&= H(A(u))  -H(A(\bu)) - (\nabla_v H) (A(\bu))  \; (A(u) - A(\bu)) \, .
\\
&\doteq H \big (A(u) | A(\bu) \big ) \, .
\end{aligned}
\end{equation}
As $H(v)$ is uniformly convex on compact subsets of $\R^n$, it follows
$$
\eta(u | \bu) >  0   \quad \mbox{ for $u, \bu \in \R^n$, $u \ne \bu$}
$$
with  $\eta(u | \bu) = 0$ if and only if $A(u) = A(\bu)$ and, by \eqref{hypns}, if and only if $u =\bu$.

\begin{lemma} 
\label{lemuseful}
Let  \eqref{hypns}, \eqref{hypep}, \eqref{hyppd} and the growth assumptions  \eqref{hypetap}, \eqref{hypfgrowth} and \eqref{hypAgrowth} be satisfied, and let 
$\bu \in B_M$. There exists $R > M$ and constants $c_1$, $c_2 > 0$ depending on $M$ such that
\begin{equation}
\label{lemform1}
\eta ( u | \bu ) \ge 
\begin{cases}
c_1  \big | A(u) - A(\bu) \big |^2   &   \; |u| \le R,  |\bu | \le M  \\
c_2  \, \eta (u)                   &   \; |u| \ge R, |\bu | \le M
\end{cases}
\end{equation}
Moreover, there exists a constant $C_3$ depending on $M$ such that for each $\alpha=1, ..., d$
\begin{equation}
\label{lemform2}
| F_\alpha (u | \bu) | \le C_3  \eta ( u | \bu )   \quad \mbox{for $u \in \R^n$, $\bu \in B_M$} \, .
\end{equation}
\end{lemma}

\begin{proof}
Observe that \eqref{workrelen} gives
$$
\begin{aligned}
\eta ( u | \bu ) &= \eta (u) - \eta (\bu) - G(\bu) \cdot (A(u) - A(\bu))
\\
&\ge \eta (u) - C_1 - C_2 |A(u) |  \qquad \qquad \qquad   \mbox{for $\bu \in B_M$} \, .
\end{aligned}
$$
Using successively \eqref{hypAgrowth} and \eqref{hypetap} we select $R > M$ sufficiently large so that
\begin{equation}
\label{selectR}
\left \{
\begin{aligned}
|A(u)|  &\le \frac{1}{2C_2} \eta (u) 
\\
4C_1  & \le \eta (u) 
\end{aligned}
\right .     \qquad \qquad |u| \ge R
\end{equation}
Then
$$
\eta ( u | \bu ) \ge \frac{1}{4} \eta (u) \qquad \qquad   \mbox{for $|u| \ge R$,  $\bu \in B_M$}.
$$

On the complementary interval $|u| \le R$,  $\bu \in B_M$, we employ \eqref{workrelen}$_2$.
Hypothesis \eqref{hyppd} states that $H(v)$ is uniformly convex on compact subsets of $\R^n$, hence
\begin{equation}
\label{equivnorm1}
 c |A(u) - A(\bu) |^2 \le \eta ( u | \bu ) \le C | A(u) - A(\bu) |^2    \qquad \qquad   \mbox{for $|u| \le R$},  
\end{equation}
since $\bu \in B_M \subset B_R$, where 
$$\displaystyle{  c = \inf_{ |x| \le R}  (\nabla^2_v H ) (A(x)) } > 0 \, ,
\qquad 
\displaystyle{  C = \sup_{ |x| \le R}  (\nabla^2_v H ) (A(x)) } \, ,
$$
and \eqref{lemform1} follows. 
Let us also note that on account of \eqref{hypns} there are constants $\bar c$ and $\bar C$ such that
\begin{equation}
\label{equivnorm2}
 \bar c |u - \bu |^2 \le \eta ( u | \bu ) \le \bar C | u - \bu |^2    \qquad \qquad   \mbox{for $|u| \le R$,  $\bu \in B_M \subset B_R$} \, .
\end{equation}
Indeed this follows easily using the mean value theorem, 
$$
A(u) - A(\bu) = \Big ( \int_0^1 \nabla A (t u + (1-t)\bu ) \, dt \Big ) (u - \bu) =  \nabla A ( t^* u + (1-t^*) \bu)  (u - \bu)
$$
for $u, \bu \in B_R$ and the invertibility of $\nabla A(u)$.

Coming next to the proof of \eqref{lemform2}, observe that \eqref{relflux} is estimated for $\bu \in B_M$ as follows:
$$
\begin{aligned}
| F_\alpha (u | \bu)  | &=  | F_\alpha (u) - F_\alpha (\bu) - \nabla F_\alpha (\bu)  \nabla A (\bu)^{-1} (A(u) - A(\bu)) |
\\
&\le  | F_\alpha (u) | + K_1 |A(u)| + K_2
\end{aligned}
$$
In view of \eqref{hypfgrowth} and for $R$ as selected in \eqref{selectR} we have
$$
| F_\alpha (u | \bu)  | \le K_3 \eta (u)  \qquad \qquad   \mbox{for $|u| \ge R$,  $\bu \in B_M$} .
$$
On the complementary region
$$
\begin{aligned}
| F_\alpha (u | \bu)  | 
&\le | F_\alpha (u) - F_\alpha (\bu) - \nabla F_\alpha (\bu)  (u - \bu)) |
\\
&\qquad +  \left |  \nabla F_\alpha (\bu)  \nabla A (\bu)^{-1} \big (A(u) - A(\bu) - \nabla A (\bu) (u - \bu)  \big ) \right |
\\
&\le  K_4 |A(u) - A(\bu)|^2 \qquad \qquad   \mbox{for $|u| \le R$,  $\bu \in B_M$} \, ,
\end{aligned}
$$
and we conclude via \eqref{lemform1} that \eqref{lemform2} holds.
\end{proof}

There is also a variant of Lemma \ref{lemuseful} that indicates the relation of $\eta (u| \bu)$ to a norm for $\bu \in B_M$.

\begin{lemma}
\label{lemvariant}
Under the hypotheses of Lemma \ref{lemuseful}, one may select $R > M + 1 $  and constants 
$\bar c_1$, $\bar c_2$ depending on $M$ so that 
\begin{equation}
\label{lemform3}
\eta ( u | \bu ) \ge 
\begin{cases}
c_1  \big | u - \bu \big |^2        &   \; |u| \le R,  |\bu | \le M  \\
c_2  \,   \big | u - \bu \big |^p    &   \; |u| \ge R, |\bu | \le M\,.
\end{cases}
\end{equation}
\end{lemma}

\begin{proof}
The proof proceeds as in Lemma \ref{lemuseful} up to the point of selecting $R > M$ using formula \eqref{selectR}
such that
$$
\eta ( u | \bu ) \ge \frac{1}{4} \eta (u) \qquad \qquad   \mbox{for $|u| \ge R$,  $\bu \in B_M$}.
$$
By Hypothesis \eqref{hypetap}, we may select $R$ so that
$$
\eta(u) \ge \frac{\beta_1}{2} |u|^p  \qquad   \mbox{ for $|u| \ge  R$}.
$$
In addition, we select $R$ even larger (if necessary) so that for $\bu \in B_M$ and any $|u| \ge R$
$$
\frac{ |u - \bu|^p}{|u|^p} \le \Big ( 1 + \frac{M}{|u|} \Big )^p \le 2 \, .
$$
Combining we conclude
$$
\begin{aligned}
\eta ( u | \bu ) &\ge \frac{1}{4} \eta (u) \ge \frac{\beta_1}{8} |u|^p 
\ge \frac{\beta_1}{16} |u - \bu |^p \qquad  |u| \ge  R , \; \bu \in B_M
\end{aligned}
$$

On the complementary interval $|u| \le R$,  $\bu \in B_M$,  we have as in Lemma \ref{lemuseful}
\begin{equation*}
  \eta ( u | \bu ) \ge c | A(u) - A(\bu) |^2    \qquad \qquad   \mbox{for $|u| \le R$,  $\bu \in B_M $},
\end{equation*}
where $\displaystyle{  c = \inf_{ |x| \le R}  (\nabla^2_v H ) (A(x)) } > 0 $.
Using  the mean value theorem we have for $u, \bu \in B_R$
$$
A(u) - A(\bu) = \Big ( \int_0^1 \nabla A (t u + (1-t)\bu ) \, dt \Big ) (u - \bu) =  \nabla A ( t^* u + (1-t^*) \bu)  (u - \bu)
$$
and by \eqref{hypns}, 
$$
 |u - \bu| = \big | \nabla A ( t^* u + (1-t^*) \bu)^{-1}  ( A(u) - A(\bu) ) \big | \le C^* | A(u) - A(\bu) |
$$
what completes the proof of \eqref{lemform3}.
\end{proof}

%

\bigskip
\noindent
{\bf Acknowledgement.} The authors would like to thank the anonymous referee and also Denis Serre for their valuable comments and suggestions that lead to improve the manuscript.\\
Research partially supported by the European Commission ITN project "Modeling and computation of shocks and interfaces".
AET acknowledges the support of the King Abdullah University of Science and Technology (KAUST).\\
\emph{Conflict of interest:} The authors have no conflicts of interest to declare.

\end{document}